\documentclass[reqno]{amsart}
\usepackage{amsmath}
\usepackage{amsthm}
\usepackage{amssymb}
\usepackage{mathrsfs} 
\usepackage{bm}
\usepackage{hyperref}
\usepackage{comment}
\usepackage[shortlabels]{enumitem}
\usepackage{tikz}
\usetikzlibrary{arrows.meta}
\tikzset{
  .../.tip={[sep=2pt 1]
    Round Cap[]. Circle[length=0.5pt 1] Circle[length=0.5pt 1] Circle[length=0.5pt 1, sep=2pt]}}
    
\usepackage{pgffor}
\usetikzlibrary{shapes.misc}
\usepackage{pgflibraryshapes}

\newcommand{\di}{\mathop{\text{\Large$\bigtriangleup$}}}

\newcommand{\res}{\mathrm{|}}

\newcommand{\REG}{{\rm REG}}

\newcommand{\ORD}{\mathop{{\rm ORD}}}

\renewcommand{\emptyset}{\varnothing}

\renewcommand{\P}{{\mathbb P}}

\newcommand{\B}{{\mathcal B}}

\newcommand{\T}{{\mathcal T}}

\newcommand{\restrict}{\upharpoonright}

\newcommand{\<}{\langle}
\renewcommand{\>}{\rangle}

\newcommand{\st}{:}

\newcommand{\ot}{\mathop{\rm ot}\nolimits}

\newcommand{\id}{\mathop{\rm id}}

\newcommand{\crit}{\mathop{\rm crit}}

\newcommand{\NS}{{\mathop{\rm NS}}}

\renewcommand{\and}{\mathop{\&}}

\newcommand{\s}{$\hat{\text{s}}$}

\newtheorem{theorem}{Theorem}[section]
\newtheorem{lemma}[theorem]{Lemma}
\newtheorem{corollary}[theorem]{Corollary}
\newtheorem{proposition}[theorem]{Proposition}

\newtheorem{fact}[theorem]{Fact}

\theoremstyle{definition}

\newtheorem{remark}[theorem]{Remark}

\newtheorem{definition}[theorem]{Definition}

\thanks{The author would like to thank Peter Holy for suggesting the use of generic ultrapowers for proving the results of Section 2, and for many additional helpful comments.}
\date{\today}

\begin{document}

\title{Higher indescribability and derived topologies}

\author[Brent Cody]{Brent Cody}
\address[Brent Cody]{ 
Virginia Commonwealth University,
Department of Mathematics and Applied Mathematics,
1015 Floyd Avenue, PO Box 842014, Richmond, Virginia 23284, United States
} 
\email[B. ~Cody]{bmcody@vcu.edu} 
\urladdr{http://www.people.vcu.edu/~bmcody/}

\begin{abstract}
We introduce reflection properties of cardinals in which the attributes that reflect are expressible by infinitary formulas whose lengths can be strictly larger than the cardinal under consideration. This kind of generalized reflection principle leads to the definitions of $L_{\kappa^+,\kappa^+}$-indescribability and $\Pi^1_\xi$-indescribability of a cardinal $\kappa$ for all $\xi<\kappa^+$. In this context, universal $\Pi^1_\xi$ formulas exist, there is a normal ideal associated to $\Pi^1_\xi$-indescribability and the notions of $\Pi^1_\xi$-indescribability yield a strict hierarchy below a measurable cardinal. Additionally, given a regular cardinal $\mu$, we introduce a diagonal version of Cantor's derivative operator and use it to extend Bagaria's \cite{MR3894041} sequence $\langle\tau_\xi\st\xi<\mu\rangle$ of derived topologies on $\mu$ to $\langle\tau_\xi\st\xi<\mu^+\rangle$. Finally, we prove that for all $\xi<\mu^+$, if there is a stationary set of $\alpha<\mu$ that have a high enough degree of indescribability, then there are stationarily-many $\alpha<\mu$ that are nonisolated points in the space $(\mu,\tau_{\xi+1})$.

\end{abstract}



\maketitle

\tableofcontents


\section{Introduction}\label{section_introduction}


When working with certain large cardinals, set theorists often use reflection arguments. For example, if $\kappa$ is a measurable cardinal then it is inaccessible, and furthermore, there are normal measure one many $\alpha<\kappa$ which are inaccessible; we say that the inaccessibility of a measurable cardinal $\kappa$ \emph{reflects} below $\kappa$. In this article we consider generalizations of this kind of reflection so that we may reflect attributes of large cardinals that are expressible by formulas whose lengths can be strictly longer than the large cardinal under consideration. We will see that in many cases, if $\kappa$ is a measurable cardinal and $\kappa$ has some property, which is expressible by a formula $\varphi$ whose length is less than $\kappa^+$, then the set of $\alpha<\kappa$ such that a canonically defined \emph{restricted} version of this formula $\varphi\res^\kappa_\alpha$ is true of $\alpha$, is normal measure one. We use this kind of generalized reflection to define the $L_{\kappa^+,\kappa^+}$-indescribability and $\Pi^1_\xi$-indescribability of a cardinal $\kappa$ for all $\xi<\kappa^+$, thus generalizing the notions of indescribability previously considered in \cite{MR3894041}. Let us note that a precursor to this type of reflection principle was studied by Sharpe and Welch (see \cite[Definition 3.21]{MR2817562}). We then use our notion of $\Pi^1_\xi$-indescribability to establish the nondiscreteness of certain topological spaces which are generalizations of the derived topologies considered in \cite{MR3894041}, and which are defined by using a diagonal version of the Cantor derivative operator (see the definition of $\tau_\xi$ and $d_\xi$ in Section \ref{section_higher_derived_topologies} and see Remark \ref{remark_example} for a simple case). 

We believe the results presented below will open up new avenues for future work in many directions. For example, in order to define the restriction of formulas (Definition \ref{definition_restriction} and Definition \ref{definition_restriction_2}) and then to establish basic properties of $\Pi^1_\xi$-indescribability, we introduce the \emph{canonical reflection functions} (see Definition \ref{definition_canonical_reflection_functions} and Section \ref{section_canonical_reflection_functions}), which are interesting in their own right and will likely have applications in areas far removed from this paper. We also expect that the notion of restriction of formulas defined below will have applications in the study of infinitary logics and model theory. Note that \cite{MR457191} and \cite{MR360274} both contain results involving a notion of restriction of $L_{\infty,\omega}$ formulas to countable sets; we suspect that these results, as well as other results in this area \cite{MR457191}, will have analogues involving our notion of restriction. Furthermore, let us note that the notion of higher $\Pi^1_\xi$-indescribability also allows for a finer analysis of the large cardinal hierarchy as in \cite{MR4206111} and \cite{cody_holy_2022}. Finally, the notions and results contained herein, particularly those on higher $\xi$-stationarity and higher derived topologies (see Section \ref{section_higher_derived_topologies}), should also allow for generalizations of many results concerning iterated stationary reflection properties and characterizations of indescribability in G\"{o}del's constructible universe (see \cite{MR1029909}, \cite{MR3416912}, \cite{MR3894041} and \cite{MR4094556}).

Before we discuss the restriction of formulas in general, let us give some examples. For cardinals $\kappa$ and $\mu$, recall that $L_{\kappa,\mu}$ denotes the infinitary logic which allows for conjunctions of $<\kappa$-many formulas that together contain $<\mu$-many free variables and quantification (universal and existential) over $<\mu$-many variables at once. If $\kappa$ is a measurable cardinal and $\varphi$ is any sentence in the $L_{\kappa,\kappa}$ language of set theory such that $V_\kappa\models\varphi$, then the set of $\alpha<\kappa$ such that $V_\alpha\models\varphi$ is normal measure one in $\kappa$. On the other hand, for any cardinal $\kappa$ there are $L_{\kappa^+,\kappa^+}$ sentences which are true in $V_\kappa$ and false in $V_\alpha$ for all $\alpha<\kappa$. For example, for each $\eta<\kappa$ there is a natural $L_{\kappa^+,\kappa^+}$ formula $\chi_\eta(x)$ such that for all $\alpha\leq\kappa$ and all $a\in V_\alpha$ we have $V_\alpha\models\chi_\eta(a)$ if and only if $a$ is an ordinal and $a$ has order type at least $\eta$. Now $\chi=\bigwedge_{\eta<\kappa}\exists x\chi_\eta(x)$ is an $L_{\kappa^+,\kappa^+}$ sentence such that $V_\kappa\models\chi$, and yet there is no $\alpha<\kappa$ such that $V_\alpha\models\chi$. However, the \emph{restriction} $\chi\res^\kappa_\alpha:=\bigwedge_{\eta<\alpha}\exists x\chi_\eta(x)$ of $\chi$ to $\alpha$ holds in $V_\alpha$ for all $\alpha<\kappa$. In what follows we will define the restriction of $L_{\kappa^+,\kappa^+}$ formulas in generality, which will allow for similar reflection results. However, the main focus of this article is on a different kind of infinitary formula. 

Generalizing the notions of $\Pi^1_n$ and $\Sigma^1_n$ formulas (see \cite{MR0281606} or \cite[Section 0]{MR1994835}), Bagaria \cite{MR3894041} defined the classes of $\Pi^1_\xi$ and $\Sigma^1_\xi$ formulas for all ordinals $\xi$. For example, if $\xi$ is a limit ordinal, a formula is $\Pi^1_\xi$ if it is of the form $\bigwedge_{\zeta<\xi}\varphi_\zeta$ where each $\varphi_\zeta$ is $\Pi^1_\zeta$. A formula is $\Sigma^1_{\xi+1}$ if it is of the form $\exists X\psi$ where $\psi$ is $\Pi^1_\xi$. Throughout this article, first-order variables will be written as lower case letters and second-order variables will be written as upper case. For more on the definition of $\Pi^1_\xi$ and $\Sigma^1_\xi$ formulas, see Section \ref{section_definition_pi1xi}. Given a cardinal $\kappa$, Bagaria defined a set $S\subseteq\kappa$ to be $\Pi^1_\xi$-indescribable in $\kappa$ if and only if for all $A\subseteq V_\kappa$ and all $\Pi^1_\xi$ sentences $\varphi$, if $(V_\kappa,\in,A)\models\varphi$ then there is an $\alpha\in S$ such that $(V_\alpha,\in,A\cap V_\alpha)\models\varphi$. Bagaria pointed out that, using his definition, no cardinal $\kappa$ can be $\Pi^1_\kappa$-indescribable because the $\Pi^1_\kappa$ sentence $\chi$ defined above is true in $V_\kappa$ but false in $V_\alpha$ for all $\alpha<\kappa$. We introduce a modification of Bagaria's notion of $\Pi^1_\xi$-indescribability which allows for a cardinal $\kappa$ to be $\Pi^1_\xi$-indescribable for all $\xi<\kappa^+$. Given a cardinal $\kappa$ and an ordinal $\xi<\kappa^+$, we say that a set $S\subseteq\kappa$ is \emph{$\Pi^1_\xi$-indescribable in $\kappa$} if and only if for all $\Pi^1_\xi$ sentences $\varphi$ (with first and second-order parameters from $V_\kappa$), if $V_\kappa\models\varphi$\footnote{Note that $\varphi$ may involve finitely-many second-order parameters $A_1,\ldots,A_n\subseteq V_\kappa$, and when we write $V_\kappa\models\varphi$ we mean $(V_\kappa,\in,A_1,\ldots,A_n)\models\varphi$. Since this abbreviated notion will not cause confusion and greatly simplifies notation, we will use it throughout the paper without further comment.} then there is some $\alpha\in S$ such that a canonically defined restriction of $\varphi$ is true in $V_\alpha$, which we express by writing $V_\alpha\models\varphi\res^\kappa_\alpha$ (see Definition \ref{definition_indescribability} for details).

In order to define the notions of restriction of $L_{\kappa^+,\kappa^+}$ formulas and restriction of $\Pi^1_\xi$ formulas, we use a sequence of functions $\<F^\kappa_\xi\st\xi<\kappa^+\>$ we call the \emph{sequence of canonical reflection functions at $\kappa$}, which is part of the set theoretic folklore and which is closely related to the sequence $\<f^\kappa_\xi\st\xi<\kappa^+\>$ of canonical functions at $\kappa$. Before defining the canonical reflection functions, let us recall some basic properties of canonical functions. Given a regular cardinal $\kappa$, the ordering defined on $^\kappa\ORD$ by letting  $f<g$ if and only if $\{\alpha<\kappa\st f(\alpha)<g(\alpha)\}$ contains a club, is a well-founded partial ordering. The Galvin-Hajnal \cite{MR376359} norm $\|f\|$ of such a function is defined to be the rank $f$ in the relation $<$. For each $\xi<\kappa^+$, there is a \emph{canonical} function $f^\kappa_\xi:\kappa\to\kappa$ of norm $\xi$, in the sense that $\|f^\kappa_\xi\|=\xi$ and whenever $\|h\|=\xi$ the set $\{\alpha<\kappa\st f^\kappa_\xi(\alpha)\leq h(\alpha)\}$ contains a club (see \cite[Page 99]{MR2768680}). For concreteness, we will use the following definition of $f^\kappa_\xi$ for $\xi<\kappa^+$. If $\xi<\kappa$ we let $f^\kappa_\xi:\kappa\to\kappa$ be the function with constant value $\xi$. If $\kappa\leq\xi<\kappa^+$ we fix a bijection $b_{\kappa,\xi}:\kappa\to\xi$ and define $f^\kappa_\xi$ by letting $f^\kappa_\xi(\alpha)=\ot(b_{\kappa,\xi}[\alpha])$ for all $\alpha<\kappa$. For convenience, we take $b_{\kappa,\kappa}$ to be the identity function $\id_\kappa:\kappa\to\kappa$, which implies that $f^\kappa_\kappa=\id_\kappa$. It is easy to see that for all $\zeta<\xi<\kappa^+$ we have $f^\kappa_\zeta<f^\kappa_\xi$ and that $f^\kappa_\xi$ is a canonical function of norm $\xi$. The sequence $\vec{f}=\<f^\kappa_\xi\st\xi<\kappa^+\>$ is sometimes referred to as the sequence of canonical functions at $\kappa$. Although, this terminology is slightly misleading as the canonical functions are only well-defined modulo the nonstationary ideal.

\begin{definition}\label{definition_canonical_reflection_functions}
Suppose $\kappa$ is a regular cardinal. For each $\xi\in\kappa^+\setminus\kappa$ let $b_{\kappa,\xi}:\kappa\to\xi$ be a bijection. We define the corresponding sequence of \emph{canonical reflection functions $\vec{F}=\<F^\kappa_\xi\st\xi<\kappa^+\>$ at $\kappa$} where $F^\kappa_\xi:\kappa\to P_\kappa\kappa^+$ for each $\xi<\kappa^+$ as follows.
\begin{enumerate}
\item For $\xi<\kappa$ we let $F^\kappa_\xi(\alpha)=\xi$ for all $\alpha<\kappa$.
\item For $\kappa\leq\xi<\kappa^+$ we let $F^\kappa_\xi(\alpha)=b_{\kappa,\xi}[\alpha]$ for all $\alpha<\kappa$.
\end{enumerate}
For each $\xi<\kappa^+$ and $\alpha<\kappa$ we let $\pi^\kappa_{\xi,\alpha}:F^\kappa_\xi(\alpha)\to f^\kappa_\xi(\alpha)$ be the transitive collapse of $F^\kappa_\xi(\alpha)$.
\end{definition}

Notice that for all $\xi<\kappa^+$ we have $f^\kappa_\xi(\alpha)=\ot(F^\kappa_\xi(\alpha))$ by definition. It is not difficult to see that for $\xi\in\kappa^+\setminus\kappa$, the $\xi^{th}$ canonical reflection function $F^\kappa_\xi$ is independent, modulo the nonstationary ideal, of which bijection $b_{\kappa,\xi}:\kappa\to\xi$ is used in its definition. That is, if $b_{\kappa,\xi}^1:\kappa\to\xi$ and $b_{\kappa,\xi}^2:\kappa\to\xi$ are two bijections then the set $\{\alpha<\kappa\st b_{\kappa,\xi}^1[\alpha]=b_{\kappa,\xi}^2[\alpha]\}$ contains a club.

In Section \ref{section_canonical_reflection_functions}, we establish many basic structural properties of the canonical reflection functions which will be used later in the paper. A particularly useful application of canonical functions \cite[Proposition 2.34]{MR2768692} is that that the $\xi^{th}$ canonical function at a regular cardinal $\kappa$ represents the ordinal $\xi$ in any generic ultrapower by any normal ideal on $\kappa$. An easy result below (see Proposition \ref{proposition_useful_object}) shows that whenever $I$ is a normal ideal on $\kappa$, $G\subseteq P(\kappa)/I$ is generic and $j:V\to V^\kappa/G\subseteq V[G]$ is the corresponding generic ultrapower embedding, the $\xi^{th}$ canonical reflection function $F^\kappa_\xi$ represents $j"\xi$ in the generic ultrapower, that is, $j(F^\kappa_\xi)(\kappa)=j"\xi$.

In Section \ref{section_definition_pi1xi}, given a regular cardinal $\kappa$, we review the definitions of $\Pi^1_\xi$ and $\Sigma^1_\xi$ formulas over $V_\kappa$; when we say that $\varphi$ is $\Pi^1_\xi$ \emph{over} $V_\kappa$ we mean that $\varphi$ is $\Pi^1_\xi$ in Bagaria's sense, but $\varphi$ is also allowed to have any number of first-order parameters from $V_\kappa$ and finitely-many second-order parameters from $V_\kappa$ (see Definition \ref{definition_over}). In Definition \ref{definition_restriction}, we use canonical reflection functions to define the notion of restriction of $\Pi^1_\xi$ and $\Sigma^1_\xi$ formulas by transfinite induction on $\xi<\kappa^+$. For example, if $\varphi=\varphi(X_1,\ldots,X_m,A_1,\ldots, A_n)$ is a $\Pi^1_\xi$ formula over $V_\kappa$ and $\xi<\kappa$, then we define
\[\varphi\res^\kappa_\alpha=\varphi(X_1,\ldots,X_m,A_1\cap V_\alpha,\ldots,A_n\cap V_\alpha).\]
As another example, suppose $\xi\in\kappa^+\setminus\kappa$ and $\xi$ is a limit ordinal. If
\[\varphi=\bigwedge_{\zeta<\xi}\varphi_\zeta\] is a $\Pi^1_\xi$ formula and $\alpha<\kappa$, then we define
\[\varphi\res^\kappa_\alpha=\bigwedge_{\zeta<f^\kappa_\xi(\alpha)}\varphi_{(\pi^\kappa_{\xi,\alpha})^{-1}(\zeta)}\res^\kappa_\alpha\]
provided that this formula is a $\Pi^1_{f^\kappa_\xi(\alpha)}$ formula over $V_\alpha$. As a consequence of this definition, it follows that there is a club $C$ in $\kappa$ such that for all regular $\alpha\in C$, $\varphi$ is a $\Pi^1_\kappa$ formula over $V_\kappa$ and $\varphi\res^\kappa_\alpha=\bigwedge_{\zeta<\alpha}\varphi_\alpha$. One nice feature of our definition of restriction is that it leads to a convenient way to represent $\Pi^1_\xi$ formulas in normal generic ultrapowers. Suppose $\kappa$ is weakly Mahlo, $I$ is a normal ideal on $\kappa$ and $\varphi$ is a $\Pi^1_\xi$ formula over $V_\kappa$ for some $\xi<\kappa^+$. Then, a result of \cite{cody_holy_2022} (see Lemma \ref{lemma_represent} below) shows that whenever $G\subseteq P(\kappa)/I$ is generic over $V$ and $j:V\to V^\kappa/G$ is the corresponding generic ultrapower embedding, we have $j(\Phi)(\kappa)=\varphi$ where $\Phi$ is the function with domain $\kappa$ defined by $\Phi(\alpha)=\varphi\res^\kappa_\alpha$.

For a given cardinal $\kappa$ and ordinal $\xi<\kappa^+$, in Definition \ref{definition_indescribability} we say that $S\subseteq\kappa$ is \emph{$\Pi^1_\xi$-indescribable} if and only if for all $\Pi^1_\xi$ sentences $\varphi$ over $V_\kappa$, whenever $V_\kappa\models\varphi$ there must be an $\alpha\in S$ such that $V_\alpha\models\varphi\res^\kappa_\alpha$.\footnote{Sharpe and Welch \cite[Definition 3.21]{MR2817562} extended the notion of $\Pi^1_n$-indescribability of a cardinal $\kappa$ where $n<\omega$ to that of $\Pi^1_\xi$-indescribability where $\xi<\kappa^+$ by demanding that the existence of a winning strategy for a particular player in a certain finite game played at $\kappa$ implies that the same player has a winning strategy in the analogous game played at some cardinal less than $\kappa$. The relationship between their notion and the one defined here is not known.} Our last result in Section \ref{section_definition_pi1xi} states that if $\kappa$ is a measurable cardinal then $\kappa$ is $\Pi^1_\xi$-indescribable for all $\xi<\kappa^+$, and furthermore, the set of $\alpha<\kappa$ such that $\alpha$ is $\Pi^1_\zeta$-indescribable for all $\zeta<\alpha^+$ is normal measure one in $\kappa$.

In Section \ref{section_other}, given a regular cardinal $\kappa$ and an $L_{\kappa^+,\kappa^+}$ formula $\varphi$ in the language of set theory, we use the canonical reflection functions at $\kappa$ to define a notion of restriction $\varphi\res^\kappa_\alpha$ by induction on subformulas, for all $\alpha\leq\kappa$. For a regular cardinal $\kappa$, we say that a set $S\subseteq\kappa$ is \emph{$L_{\kappa^+,\kappa^+}$-indescribable} if and only if for all $L_{\kappa^+,\kappa^+}$ sentences in the language of set theory with $V_\kappa\models\varphi$ there is an $\alpha<\kappa$ such that $V_\alpha\models\varphi\res^\kappa_\alpha$. Proposition \ref{proposition_Lkappa+} states that if $\kappa$ is a measurable cardinal, then $\kappa$ is $L_{\kappa^+,\kappa^+}$-indescribable and furthermore, the set of regular cardinals $\alpha<\kappa$ that are $L_{\alpha^+,\alpha^+}$-indescribable is normal measure one in $\kappa$. 

Generalizing the results of L\'evy \cite{MR0281606} and Bagaria \cite{MR3894041} on universal formulas, in Section \ref{section_universal}, we establish the existence of universal $\Pi^1_\xi$ and $\Sigma^1_\xi$ formulas at a regular cardinal $\kappa$ for all $\xi<\kappa^+$ in an appropriate sense. Using universal formulas, we prove Theorem \ref{theorem_normal_ideal}, which states that if $\kappa$ is $\Pi^1_\xi$-indescribable where $\xi<\kappa^+$, then the collection
\[\Pi^1_\xi(\kappa)=\{X\subseteq\kappa\st\text{$X$ is not $\Pi^1_\xi$-indescribable}\}\]
is a nontrivial normal ideal on $\kappa$.

In Section \ref{section_hierarchy}, again using the existence of universal $\Pi^1_\xi$ formulas discussed above, we prove Theorem \ref{theorem_expressing_indescribability}, which states that given a regular cardinal $\kappa$ and $\xi<\kappa^+$, the $\Pi^1_\xi$-indescribability of a set $S\subseteq\kappa$ is, in an appropriate sense, expressible by a $\Pi^1_{\xi+1}$ formula. We then prove two hierarchy results for $\Pi^1_\xi$-indescribability. For example, as a consequence of these results, if $\kappa$ is $\kappa+n+1$-indescribable, where $n<\omega$, then the set of $\alpha<\kappa$ which are $\alpha+n$-indescribable is in the filter $\Pi^1_{\kappa+n+1}(\kappa)^*$. More generally, our first hierarchy result, Corollary \ref{corollary_hierarchy}, states that if $\kappa$ is $\Pi^1_\xi$-indescribable where $\xi<\kappa^+$ and $\zeta<\xi$, then the set of $\alpha<\kappa$ which are $\Pi^1_{f^\kappa_\zeta(\alpha)}$-indescribable is in the filter $\Pi^1_\xi(\kappa)^*$. Our second hierarchy result, Corollary \ref{corollary_proper}, states that if $\kappa$ is $\Pi^1_\xi$-indescribable where $\xi<\kappa^+$, then for all $\zeta<\xi$ we have $\Pi^1_\zeta(\kappa)\subsetneq\Pi^1_\xi(\kappa)$. The proofs of these two hierarchy results require several lemmas which are interesting in their own right. For example, Proposition \ref{proposition_double_restriction} state that for any weakly Mahlo cardinal $\kappa$ and ordinal $\xi<\kappa^+$, if $\varphi$ is any $\Pi^1_\xi$ or $\Sigma^1_\xi$ formula then there is a club $C\subseteq\kappa$ such that for all regular $\alpha\in C$, the set of $\beta<\alpha$ for which 
\[(\varphi\res^\kappa_\alpha)\res^\alpha_\beta=\varphi\res^\kappa_\beta\]
is club in $\alpha$.

Recall that, by results of Sun \cite{MR1245524} and Hellsten \cite{MR2026390}, one can characterize $\Pi^1_n$-indescribable subsets of a $\Pi^1_n$-indescribable cardinal $\kappa$ by using a natural base for the filter $\Pi^1_n(\kappa)^*$ dual to $\Pi^1_n(\kappa)$. For a regular cardinal $\kappa$, a set $C\subseteq\kappa$ is a \emph{$\Pi^1_0$-club in $\kappa$} if it is club in $\kappa$. We say that $C\subseteq\kappa$ is \emph{$\Pi^1_{n+1}$-club} in $\kappa$, where $n<\omega$, if it is $\Pi^1_n$-indescribable in $\kappa$ and whenever $C\cap\alpha$ is $\Pi^1_n$-indescribable in $\alpha$ we have $\alpha\in C$. Then, if $\kappa$ is $\Pi^1_n$-indescribable, a set $S\subseteq\kappa$ is $\Pi^1_n$-indescribable if and only if $S\cap C\neq\emptyset$ for all $\Pi^1_n$-clubs $C\subseteq\kappa$. This result is due to Sun \cite{MR1245524} for $n=1$ and to Hellsten \cite{MR2026390} for $n<\omega$. In Section \ref{section_higher_xi_clubs}, we generalize this to $\Pi^1_\xi$-indescribable subsets of $\Pi^1_\xi$-indescribable cardinals for all $\xi<\kappa^+$. That is, for all $\xi<\kappa^+$, we introduce a notion of $\Pi^1_\xi$-club subset of $\kappa$ such that if $\kappa$ is $\Pi^1_\xi$-indescribable then a set $S\subseteq\kappa$ is $\Pi^1_\xi$-indescribable if and only if $S\cap C\neq\emptyset$ for all $\Pi^1_\xi$-clubs $C\subseteq\kappa$. For more results involving $\Pi^1_\xi$-clubs, one should consult \cite{MR3985624}, \cite{MR4050036}, \cite{MR4230485} and \cite{MR4082998}.

Finally, in Section \ref{section_higher_derived_topologies}, we generalize some of the results of Bagaria \cite{MR3894041} on derived topologies on ordinals. 
Given a nonzero ordinal $\delta$, Bagaria defined a transfinite sequence of topologies $\<\tau_\xi\st\xi\in\ORD\>$ on $\delta$, called the \emph{derived topologies on $\delta$}, and proved---using the definitions of \cite{MR3894041}---that if there is an $\alpha<\delta$ which is $\Pi^1_\xi$-indescribable then the $\tau_{\xi+1}$ topology on $\delta$ is non-discrete. However, using the definitions of \cite{MR3894041}, $\alpha$ can be $\Pi^1_\xi$-indescribable only if $\xi<\alpha$. Thus, Bagaria obtained the non-discreteness of the $\tau_\xi$ topologies on $\delta$ only for $\xi<\delta$. Given a regular cardinal $\mu$, in Section \ref{section_higher_derived_topologies}, using \emph{diagonal Cantor derivatives}, we present a natural extension of Bagaria's notion of derived topologies on $\mu$ by defining a transfinite sequence of topologies $\<\tau_\xi\st\xi<\mu^+\>$ on $\mu$ such that for $\xi<\mu$ our $\tau_\xi$ is the same as that of \cite{MR3894041} and Bagaria's conditions for the nondiscreteness of the topologies $\tau_{\xi+1}$ for $\xi<\mu$ can be generalized to all $\xi<\mu^+$ (see Theorem \ref{theorem_xi_s_nonisolated} and Corollary \ref{corollary_nondiscreteness_from_indescribability}). 
\begin{remark}\label{remark_example}
Let us describe the simplest of the new topologies introduced in this article. If $\<\tau_\xi\st\xi<\mu\>$ is Bagaria's sequence of derived topologies on a regular $\mu$, we define $d_\mu:P(\mu)\to P(\mu)$ by letting
\[d_\mu(A)=\{\alpha<\mu\st\text{$\alpha$ is a limit point of $A$ in the $\tau_\alpha$ topology on $\mu$}\}.\]
We then define a new topology $\tau_\mu$ declaring $C\subseteq\mu$ to be closed in the space $(\mu,\tau_\mu)$ if and only if $d_\mu(C)\subseteq C$. That is, we let $U\in\tau_\mu$ if and only if $d_\mu(\mu\setminus U)\subseteq\mu\setminus U$ for $U\subseteq\mu$.
\end{remark}

\section{Canonical reflection functions}\label{section_canonical_reflection_functions}

In this section we establish the basic properties of the canonical reflection functions at a regular cardinal. Although some of these results are folklore, we include proofs for the reader's convenience. Many of the proofs in the current section will establish that certain sets defined using canonical reflection functions are in the club filter on a given regular cardinal $\kappa$. These results will be established by using generic ultrapower embeddings\footnote{The author would like to thank Peter Holy for suggesting the use of generic ultrapowers in the arguments of the current section.}; some background material on generic ultrapowers may be found in \cite{MR2768692}, but we will only that which is summarized here. Recall that if $\kappa$ is a regular cardinal, $I$ is a normal ideal on $\kappa$ and $G\subseteq P(\kappa)/I$ is generic over $V$, then, working in the forcing extension $V[G]$ there is a canonical $V$-normal $V$-ultrafilter $U_G\subseteq P(\kappa)$ obtained from $G$ such that $U_G$ extends the filter $I^*$ dual to $U$ and we may form the corresponding generic ultrapower $j:V\to V^\kappa/U_G\subseteq V[G]$. Further recall that the critical point of $j$ is $\kappa$ and equals the equivalence class of the identity function $\id:\kappa\to\kappa$. Thus, for all $X\in P(\kappa)^V$ we have $X\in U$ if and only if $\kappa\in j(X)$. Furthermore, the ultrapower $V^\kappa/U_G$ is wellfounded up to $(\kappa^+)^V$, $H(\kappa^+)\subseteq V^\kappa/U_G$ and when $\kappa$ is inaccessible we have $H(\kappa)=H(\kappa)^{V^\kappa/U_G}$. As is standard practice, in what follows we will often write $V^\kappa/G$ to mean $V^\kappa/U_G$. The following two propositions will be used throughout the article.

\begin{proposition}\label{proposition_framework}
Suppose $\kappa$ is a regular uncountable cardinal and $S\subseteq\kappa$. Then $S$ contains a club subset of $\kappa$ if and only if whenever $G$ is generic for $P(\kappa)/\NS_\kappa$ it follows that $\kappa\in j(S)$ where $j:V\to V^\kappa/G$ is the generic ultrapower embedding obtained from $G$.
\end{proposition}

\begin{proposition}\label{proposition_framework2}
The following are equivalent when $\kappa$ is a regular uncountable cardinal\footnote{Notice that when the set of regular cardinals less than $\kappa$ is not stationary in $\kappa$, i.e. when $\kappa$ is not weakly Mahlo, then both (1) and (2) hold trivially. In later sections, when we apply Proposision \ref{proposition_framework2}, and the results derived from it in the current section, $\kappa$ will in fact be weakly Mahlo.} and $E\subseteq\kappa$.
\begin{enumerate}
\item There is a club $C\subseteq\kappa$ such that for all regular uncountable $\alpha\in C$ we have $\alpha\in E$.
\item Whenever $G\subseteq P(\kappa)/\NS_\kappa$ is generic over $V$ such that $\kappa$ is regular in $V^\kappa/G$ and $j:V\to V^\kappa/G$ is the corresponding generic ultrapower embedding, we have $\kappa\in j(E)$.
\end{enumerate}
\end{proposition}

\begin{proof}
It is trivial to see that (1) implies (2). If (1) is false then the set $S$ of regular cardinals in $\kappa\setminus E$ is stationary in $\kappa$. Let $G\subseteq P(\kappa)/\NS_\kappa$ be generic over $V$ with $S\in G$ and let $j:V\to V^\kappa/G$ be the corresponding generic ultrapower embedding. Then $\kappa\in j(S)$, which implies $\kappa$ is regular in $V^\kappa/G$ and $\kappa\notin j(E)$ contradicting (2).
\end{proof}

The next result shows that, for regular $\kappa$, the $\xi^{th}$ canonical reflection function $F^\kappa_\xi$ (see Definition \ref{definition_canonical_reflection_functions}) represents a useful object in any generic ultrapower obtained from a normal ideal on $\kappa$.

\begin{proposition}\label{proposition_useful_object}
Suppose $\kappa$ is a regular cardinal and $I$ is a normal ideal on $\kappa$. Let $G$ be generic for $P(\kappa)/I$ and let $j:V\to V^\kappa/G$ be the corresponding generic ultrapower embedding. Then, for all $\xi<\kappa^+$, the $\xi^{th}$ canonical reflection function $F^\kappa_\xi$ represents $j"\xi$ in the generic ultrapower, that is, $j(F^\kappa_\xi)(\kappa)=j"\xi$.
\end{proposition}

\begin{proof}
Let $j:V\to V^\kappa/G$ be the generic ultrapower obtained from a generic filter $G\subseteq P(\kappa)/I$ over $V$. Since $\crit(j)=\kappa$, it is easy to see that for $\xi\leq\kappa$ we have $j(F^\kappa_\xi)(\kappa)=j"\xi$. Now suppose $\kappa<\xi<\kappa^+$ and let $b_{\kappa,\xi}:\kappa\to\xi$ be the bijection such that $F^\kappa_\xi(\alpha)=b_{\kappa,\xi}[\alpha]$ for all $\alpha<\kappa$. By elementarity, $j(b_{\kappa,\xi}):j(\kappa)\to j(\xi)$ is a bijection in $M$ and $j(b_{\kappa,\xi})(\alpha)=j(b_{\kappa,\xi}(\alpha))$ for all $\alpha<\kappa$. Thus, $j(F^\kappa_\xi)(\kappa)=j(b_{\kappa,\xi})[\kappa]=j"\xi$.
\end{proof}

\begin{corollary}
Suppose $U$ is a normal measure on $\kappa$ and $j:V\to M$ is the corresponding ultrapower embedding. For all $\xi<\kappa^+$, the $\xi^{th}$ canonical reflection function $F^\kappa_\xi$ represents $j"\xi$ in the ultrapower, that is, $j(F^\kappa_\xi)(\kappa)=j"\xi$.
\end{corollary}

Next we show that at least some of the canonical reflection functions at a regular $\kappa$ are, in fact, canonical; in Remark \ref{remark_not_canonical}, we show that this partial canonicity result is the best possible.

\begin{lemma}\label{lemma_canonicity}
Suppose $\kappa$ is regular.\begin{enumerate}
\item For all $\xi<\kappa^+$ the set $\{\alpha<\kappa\st F^\kappa_\zeta(\alpha)\subsetneq F^\kappa_\xi(\alpha)\}$ contains a club subset of $\kappa$ for all $\zeta<\xi$.
\item If $\xi<\kappa^+$ is a limit ordinal then the set
\[\{\alpha<\kappa\st F^\kappa_\xi(\alpha)=\bigcup_{\zeta\in F^\kappa_\xi(\alpha)} F^\kappa_\zeta(\alpha)\}\]
contains a club subset of $\kappa$.
\item If $\xi<\kappa^+$ is a limit ordinal the function $F^\kappa_\xi$ is canonical in the sense that whenever $F:\kappa\to P_\kappa\kappa^+$ is a function such that for all $\zeta<\xi$ the set $\{\alpha<\kappa\st F^\kappa_\zeta(\alpha)\subseteq F(\alpha)\}$ contains a club, then the set $\{\alpha<\kappa\st F^\kappa_\xi(\alpha)\subseteq F(\alpha)\}$ contains a club subset of $\kappa$.
\end{enumerate}
\end{lemma}

\begin{proof}

For (1), suppose $\zeta<\xi<\kappa^+$ and let $C=\{\alpha<\kappa\st F^\kappa_\zeta(\alpha)\subsetneq F^\kappa_\xi(\alpha)\}$. Let $G\subseteq P(\kappa)/\NS_\kappa$ be generic over $V$ and let $j:V\to V^\kappa/G$ be the corresponding generic ultrapower. By Proposition \ref{proposition_useful_object}, $\kappa\in j(C)$ and thus by Proposition \ref{proposition_framework} we see that $C$ contains a club subset of $\kappa$.

Similarly, for (2), suppose $\xi$ is a limit ordinal, let $C=\{\alpha<\kappa\st F^\kappa_\xi(\alpha)=\bigcup_{\zeta\in F^\kappa_\xi(\alpha)}F^\kappa_\zeta(\alpha)\}$ and let $j:V\to V^\kappa/G$ be the generic ultrapower obtained by forcing with $P(\kappa)/\NS_\kappa$. Working in $V^\kappa/G$, if $\zeta$ is an ordinal less than $j(\kappa)^+$, we let $\overline{F}^{j(\kappa)}_\zeta$ denote the $\zeta$-th cannonical reflection function at $j(\kappa)$. For each $\zeta<\xi$ we let $j(\<F^\kappa_\zeta(\alpha)\st\alpha<\kappa\>)=\<\overline F^{j(\kappa)}_{j(\zeta)}(\alpha)\st\alpha<j(\kappa)\>$. Notice that 
\begin{align*}j(C)&=\{\alpha<\kappa\st j(F^\kappa_\xi)(\alpha)=\bigcup_{\zeta\in j(F^\kappa_\xi)(\alpha)}\overline F^{j(\kappa)}_\zeta(\alpha)\}\\
	&=\{\alpha<\kappa\st j"\xi=\bigcup_{\zeta\in j"\xi}\overline{F}^{j(\kappa)}_\zeta(\alpha)\}
\end{align*}
For each $\zeta\in j"\xi$ we have $\overline F^{j(\kappa)}_\zeta(\kappa)=\overline F^{j(\kappa)}_{j(j^{-1}(\zeta))}(\kappa)=j(F^\kappa_{j^{-1}(\zeta)})(\kappa)=j"(j^{-1}(\zeta))=(j"\xi)\cap\zeta$, it follows that $\kappa\in j(C)$.

For (3), suppose $\xi<\kappa^+$ is a limit and let $F$ be as in the statement of the lemma. By assumption, if $j:V\to V^\kappa/G$ is any generic ultrapower obtained by forcing with $P(\kappa)/\NS_\kappa$, then $j(F^\kappa_\zeta)(\kappa)=j"\zeta\subseteq j(F)(\kappa)$ for all $\zeta<\xi$. By (2), we know that $j(F^\kappa_\xi)(\kappa)=j"\xi=\bigcup_{\zeta<\xi}j"\zeta$ and hence $j(F^\kappa_\xi)(\kappa)\subseteq j(F)(\kappa)$.
\end{proof}

\begin{remark}\label{remark_not_canonical}
Let us point out that Lemma \ref{lemma_canonicity}(3) does not hold if $\xi<\kappa^+$ is a successor ordinal. For example, supose $\xi=\kappa+1$ and $F:\kappa\to P_\kappa\kappa^+$ is defined by $F(\alpha)=\alpha$. Let $j:V\to V^\kappa/G$ be any generic ultrapower obtained by forcing with $P(\kappa)/\NS_\kappa$. Since $j(F)(\kappa)=\kappa$ and $j"(\kappa+1)=\kappa\cup\{j(\kappa)\}$ we see that $\{\alpha<\kappa\st F^\kappa_\kappa(\alpha)\subseteq F(\alpha)\}$ contains a club in $\kappa$ and $\{\alpha<\kappa\st F^\kappa_{\kappa+1}(\alpha)\subseteq F(\alpha)\}$ is nonstationary in $\kappa$.


\end{remark}

The following lemma shows that the canonical reflection functions at a regular cardinal satisfy a natural kind of coherence property.

\begin{lemma}\label{lemma_coherence}
Suppose $\kappa$ is a regular cardinal and $\xi<\kappa^+$ is a limit ordinal. Let $\pi^\kappa_{\xi,\alpha}:F^\kappa_\xi(\alpha)\to f^\kappa_\xi(\alpha)$ be the transitive collapse of $F^\kappa_\xi(\alpha)$ for each $\alpha<\kappa$. Then the set
\[C=\{\alpha<\kappa\st(\forall\zeta\in F^\kappa_\xi(\alpha))\ F^\kappa_\xi(\alpha)\cap\zeta=F^\kappa_\zeta(\alpha)\}\]
contains a club subset of $\kappa$.
\end{lemma}

\begin{proof}
Let $G\subseteq P(\kappa)/\NS_\kappa$ be generic over $V$ and let $j:V\to V^\kappa/G$ be the corresponding generic ultrapower embedding. Let $\vec{F}=\<F^\kappa_\zeta\st\zeta<\kappa^+\>$ and notice that $j(\vec{F})=\<\overline F^{j(\kappa)}_\zeta\st\zeta<j(\kappa^+)\>$ where $\overline F^{j(\kappa)}_\zeta$ is the $\zeta$-th canonical reflection function at $j(\kappa)$ in $V^\kappa/G$. We have
\[j(C)=\{\alpha<j(\kappa)\st(\forall\zeta\in j(F^\kappa_\xi)(\alpha))\ j(F^\kappa_\xi)(\alpha)\cap\zeta = \overline F^{j(\kappa)}_\zeta(\alpha)\}.\]
Since $j(F^\kappa_\xi)(\kappa)=j"\xi$ and for each $\zeta\in j"\xi$ we have $\overline F^{j(\kappa)}_\zeta(\kappa)=\overline F^{j(\kappa)}_{j(j^{-1}(\zeta))}(\kappa)=j(F^\kappa_{j^{-1}(\zeta)})(\kappa)=j"\zeta$, it follows that $\kappa\in j(C)$.
\end{proof}

Next we will show that for all limit ordinals $\xi<\kappa^+$, for club many $\alpha<\kappa$, the value of $f^\kappa_\xi(\alpha)$ is determined by the values of $f^\kappa_\zeta(\alpha)$ for $\zeta\in F^\kappa_\xi(\alpha)$.

\begin{lemma}\label{lemma_canonical_functions_at_limits}
Suppose $\kappa$ is regular and $\xi<\kappa^+$ is a limit ordinal. Then the set
\[D=\{\alpha<\kappa\st f^\kappa_\xi(\alpha)=\bigcup_{\zeta\in F^\kappa_\xi(\alpha)}f^\kappa_\zeta(\alpha)\}\]
contains a club subset of $\kappa$.
\end{lemma}

\begin{proof}
Let $G\subseteq P(\kappa)/\NS_\kappa$ be generic over $V$ and let $j:V\to V^\kappa/G$ be the corresponding generic ultrapower embedding. Let $j(\<f^\kappa_\zeta\st\zeta<\kappa^+\>)=\<\overline f^{j(\kappa)}_\zeta\st\zeta<j(\kappa^+)\>$. We have
\[j(D)=\{\alpha<j(\kappa)\st j(f^\kappa_\xi)(\alpha)=\bigcup_{\zeta\in j(F^\kappa_\xi)(\alpha)} \overline f^{j(\kappa)}_\zeta(\alpha)\}.\]
Since $\xi=\bigcup_{\zeta\in j"\xi} j^{-1}(\zeta)$, it follows that $\kappa\in j(D)$.
\end{proof}

The next two lemmas follow easily from Proposition \ref{proposition_framework} and confirm our intuition that for a regular cardinal $\kappa$ and ordinal $\xi<\kappa^+$, for club-many $\alpha<\kappa$ the value $f^\kappa_\xi(\alpha)$ behaves like $\alpha$'s version of $\xi$.

\begin{lemma}\label{lemma_limits}
Suppose $\kappa$ is regular and $\xi<\kappa^+$ is a limit ordinal. Then the set
\[D=\{\alpha<\kappa\st \text{$f^\kappa_\xi(\alpha)$ is a limit ordinal}\}\]
contains a club subset of $\kappa$.
\end{lemma}

\begin{lemma}\label{lemma_successor}
Suppose $\kappa$ is regular. For all $\zeta<\kappa^+$ the following sets are closed unbounded in $\kappa$.
\begin{align*}
D_0&=\{\alpha<\kappa\st F^\kappa_{\zeta+1}(\alpha)\cap\zeta=F^\kappa_\zeta(\alpha)\}\\
D_1&=\{\alpha<\kappa\st F^\kappa_{\zeta+1}(\alpha)=F^\kappa_\zeta(\alpha)\cup\{\zeta\}\}\\
D_2&=\{\alpha<\kappa\st f^\kappa_{\zeta+1}(\alpha)=f^\kappa_\zeta(\alpha)+1\}
\end{align*}
\end{lemma}

Next we prove a proposition which generalizes a folklore result concerning canonical functions (see Corollary \ref{corollary_crazy}) to canonical reflection functions, and which draws a connection between the canonical reflection functions at a regular cardinal $\kappa$ and the canonical reflection functions at regular $\alpha<\kappa$. The following proposition was originally established in a previous version of this article using a more complicated proof; the proof below is due to Cody and Holy and appears in \cite{cody_holy_2022}.

\begin{proposition}\label{proposition_crazy}
Suppose $\kappa$ is regular and $\xi<\kappa^+$. For each $\alpha<\kappa$ let
\[\pi^\kappa_{\xi,\alpha}:F^\kappa_\xi(\alpha)\to f^\kappa_\xi(\alpha)\]
be the transitive collapse of $F^\kappa_\xi(\alpha)$. Then there is a club $C^\kappa_\xi\subseteq\kappa$ such that for all regular uncountable $\alpha\in C^\kappa_\xi$ the set
\[D^\alpha_\xi=\{\beta<\alpha \st \pi^\kappa_{\xi,\alpha}[F^\kappa_\xi(\beta)]=F^\alpha_{f^\kappa_\xi(\alpha)}(\beta)\}\]
is in the club filter on $\alpha$.
\end{proposition}

\begin{proof}
In order to prove the existence of such a club, we will use Proposition \ref{proposition_framework2}. Suppose $G\subseteq P(\kappa)/\NS_\kappa$ is generic over $V$ such that $\kappa$ is regular in $V^\kappa/G$ and let $j:V\to V^\kappa/G$ be the corresponding generic ultrapower. 
For each regular uncountable $\alpha<\kappa$ let $D^\alpha_\xi=\{\beta<\alpha \st \pi^\kappa_{\xi,\alpha}[F^\kappa_\xi(\beta)]=F^\alpha_{f^\kappa_\xi(\alpha)}(\beta)\}$. We must show that
\[\kappa\in j(\{\alpha\in\REG\cap\kappa\st D^\alpha_\xi\text{ contains a club subset of $\kappa$}\}).\] 
Let $\vec{D}=\<D^\alpha_\xi\st\alpha\in\REG\cap\kappa\>$,
$\vec{\pi}=\<\pi^\kappa_{\xi,\alpha}\st\alpha<\kappa\>$ and $\vec{F}=\<F^\alpha_{f^\kappa_\xi(\alpha)}\st\alpha\in\REG\cap\kappa\>$. By elementarity it follows that in $V^\kappa/G$, $j(\vec{\pi})_\kappa$ is a bijection from $j(F^\kappa_\xi)(\kappa)=j"\xi$ to $j(f^\kappa_\xi)(\kappa)=\xi$. Thus the set $\{j(\vec{\pi})_\kappa[j(F^\kappa_\xi)(\beta)]\st\beta<\kappa\}$ is cofinal in $[\xi]^{<\kappa}$. Also by elementarity, we see that $j(\vec{F})_\kappa$ is the $\xi$-th canonical reflection function at $\kappa$ in $V^\kappa/G$ and hence the set $\{j(\vec{F})_\kappa(\beta)\st\beta<\kappa\}$ is cofinal in $[\xi]^{<\kappa}$. By the usual catching up argument, in $V^\kappa/G$ the set $j(\vec{D})_\kappa$ contains a club subset of $\kappa$.
\end{proof}

The following folklore result (see \cite[Section 5]{MR1077260}) easily follows from Proposition \ref{proposition_crazy}, or can be established directly using an argument which is easier than that of Proposition \ref{proposition_crazy}.

\begin{corollary}\label{corollary_crazy}
Suppose $\kappa$ is regular and $\xi<\kappa^+$. Then there is a club $C^\kappa_\xi\subseteq\kappa$ such that for all regular uncountable $\alpha\in C^\kappa_\xi$ the set
\[D^\alpha_\xi=\{\beta<\alpha \st f^\kappa_\xi(\beta)=f^\alpha_{f^\kappa_\xi(\alpha)}(\beta)\}\]
is in the club filter on $\alpha$.
\end{corollary}

\section{Restricting $\Pi^1_\xi$ formulas and consistency of higher $\Pi^1_\xi$-indescribability} \label{section_definition_pi1xi}

We begin this section with a precise definition of $\Pi^1_\xi$ and $\Sigma^1_\xi$ formulas \emph{over $V_\kappa$}, where $\kappa$ is a regular cardinal and $\xi$ is an ordinal. The following definition is similar to \cite[Definition 4.1]{MR3894041}, the only difference being that we allow for first and second order parameters from $V_\kappa$. Recall that throughout the article we use capital letters to denote second-order variables and lower case letters to denote first-order variables.

\begin{definition}\label{definition_over}
Suppose $\kappa$ is a regular cardinal. We define the notions of $\Pi^1_\xi$ and $\Sigma^1_\xi$ formula over $V_\kappa$, for all ordinals $\xi$ as follows. 
\begin{enumerate}
\item A formula $\varphi$ is $\Pi^1_0$, or equivalently $\Sigma^1_0$, over $V_\kappa$ if it is a first order formula in the language of set theory, however we allow for free variables and parameters from $V_\kappa$ of two types, namely of first and of second order. 
\item A formula $\varphi$ is $\Pi^1_{\xi+1}$ over $V_\kappa$ if it is of the form $\forall X_{k_1}\cdots\forall X_{k_m}\psi$ where $\psi$ is $\Sigma^1_\xi$ over $V_\kappa$ and $m\in\omega$. Similarly, $\varphi$ is $\Sigma^1_{\xi+1}$ over $V_\kappa$ if it is of the form $\exists X_{k_1}\cdots\exists X_{k_m}\psi$ where $\psi$ is $\Pi^1_\xi$ over $V_\kappa$ and $m\in\omega$.\footnote{We follow the convention that uppercase letters represent second order variables, while lower case letters represent first order variables. Thus, in the above, all quantifiers displayed are understood to be second order quantifiers, i.e., quantifiers over subsets of $V_\kappa$.}
\item When $\xi$ is a limit ordinal, a formula $\varphi$, with finitely many second-order free variables and finitely many second-order parameters, is $\Pi^1_\xi$ over $V_\kappa$ if it is of the form
\[\bigwedge_{\zeta<\xi}\varphi_\zeta\]
where $\varphi_\zeta$ is $\Pi^1_\zeta$ over $V_\kappa$ for all $\zeta<\xi$. Similarly, $\varphi$ is $\Sigma^1_\xi$ if it is of the form 
\[\bigvee_{\zeta<\xi}\varphi_\zeta\]
where $\varphi_\zeta$ is $\Sigma^1_\zeta$ over $V_\kappa$ for all $\zeta<\xi$.
\end{enumerate}
\end{definition}

\begin{definition}\label{definition_restriction}
By induction on $\xi<\kappa^+$, we define $\varphi\res^\kappa_\alpha$ for all $\Pi^1_\xi$ formulas $\varphi$ over $V_\kappa$ and all regular $\alpha<\kappa$ as follows. First assume that $\xi<\kappa$. If \[\varphi=\varphi(X_1,\ldots,X_m,A_1,\ldots,A_n),\] with free second order variables $X_1,\ldots,X_m$ and second order parameters $A_1,\ldots,A_n$, then we define
\[\varphi\res^\kappa_\alpha=\varphi(X_1,\ldots,X_m,A_1\cap V_\alpha,\ldots,A_n\cap V_\alpha).\]

If $\xi=\zeta+1$ is a successor ordinal and $\varphi=\forall X_{k_1}\ldots\forall X_{k_m}\psi$ is $\Pi^1_{\zeta+1}$ over $V_\kappa$, then we define 
\[\varphi\res^\kappa_\alpha=\forall X_{k_1}\ldots\forall X_{k_m}(\psi\res^\kappa_\alpha).\] 
We define $\varphi\res^\kappa_\alpha$ analogously when $\varphi$ is $\Sigma^1_{\zeta+1}$.

If $\xi\in\kappa^+\setminus\kappa$ is a limit ordinal, and 
\begin{align}\varphi=\bigwedge_{\zeta<\xi}\psi_\zeta\label{equation_defn_restriction}\end{align}
is $\Pi^1_\xi$ over $V_\kappa$, then we define
\[\varphi\res^\kappa_\alpha=\bigwedge_{\zeta\in f^\kappa_\xi(\alpha)}\psi_{(\pi^\kappa_{\xi,\alpha})^{-1}(\zeta)}\res^\kappa_\alpha\]
in case $\psi_{(\pi^\kappa_{\xi,\alpha})^{-1}(\zeta)}\res^\kappa_\alpha$ is a $\Pi^1_\zeta$ formula over $V_\alpha$ for every $\zeta<f^\kappa_\xi(\alpha)$. We leave $\varphi\res^\kappa_\alpha$ undefined otherwise. We define $\varphi\res^\kappa_\alpha$ similarly when $\xi\in\kappa^+\setminus\kappa$ is a limit ordinal and $\varphi$ is $\Sigma^1_\xi$.

\end{definition}

\begin{remark}\label{remark_definition_of_restriction} A few remarks about Definition \ref{definition_restriction} are in order.
\begin{enumerate}
\item An easy inductive argument on $\xi<\kappa^+$ shows that if $\varphi$ is a $\Pi^1_\xi$ or $\Sigma^1_\xi$ formula over $V_\kappa$, and $\alpha<\kappa$ is regular, then whenever $\varphi\res^\kappa_\alpha$ is defined, it is a $\Pi^1_{f^\kappa_\xi(\alpha)}$ or $\Sigma^1_{f^\kappa_\xi(\alpha)}$ formula over $V_\alpha$ respectively.
\item Recall that we defined the sequence of canonical reflection functions $\<F^\kappa_\xi\st\xi<\kappa^+\>$, the sequence of canonical functions $\<f^\kappa_\xi\st\xi<\kappa^+\>$ and the transitive collapses $\pi^\kappa_{\xi,\alpha}:F^\kappa_\xi(\alpha)\to f^\kappa_\xi(\alpha)$ in a particular way making use of fixed sequence of bijections $\<b_{\kappa,\xi}\st\xi\in\kappa^+\setminus\kappa\>$. Thus, the definition of $\varphi\res^\kappa_\alpha$ given above clearly depends on our choice of bijections $\<b_{\kappa,\xi}\st\xi\in\kappa^+\setminus\kappa\>$. Below we will see that our definition of $\varphi\res^\kappa_\alpha$ is independent of this choice of bijections modulo the nonstationary ideal. See the paragraph after Definition \ref{definition_indescribability} for details.
\end{enumerate}
\end{remark}

In order to establish some basic properties of the restriction operation from Definition \ref{definition_restriction}, let us consider how it behaves with respect to generic ultrapowers. We will want to apply elementary embeddings to $\Pi^1_\xi$ and $\Sigma^1_\xi$ formulas, which will be viewed as set theoretic objects.

\begin{remark}\label{remark_coding}  Assume that $\varphi$ is either a $\Pi^1_\xi$ or $\Sigma^1_\xi$ formula over $V_\kappa$ for some $\xi<\kappa^+$. Let $j:V\to V^\kappa/G$ be the generic ultrapower embedding obtained by forcing with $P(\kappa)/I$ where $I$ is some normal ideal on $\kappa$. We will leave it to the reader to check that any reasonable coding of formulas has the following properties.

\begin{enumerate}
  \item If $\xi<\kappa$, and $A_1,\ldots,A_n$ are all second order parameters appearing in $\varphi$, then \[j(\varphi(A_1,\ldots,A_n))=\varphi(j(A_1),\ldots,j(A_n)).\]
  \item $j(\forall X\,\varphi)=\forall X\,j(\varphi)$.
  \item If $\xi\ge\kappa$ is a limit ordinal, and $\varphi$ is either of the form $\varphi=\bigwedge_{\zeta<\xi}\psi_\zeta$, or of the form $\bigvee_{\zeta<\xi}\psi_\zeta$, let $\vec\psi=\langle\psi_\zeta\mid\zeta<\xi\rangle$. Then, \[j(\varphi)=\bigwedge_{\zeta<j(\xi)}j(\vec\psi)_\zeta\quad\textrm{or}\quad j(\varphi)=\bigvee_{\zeta<j(\xi)}j(\vec\psi)_\zeta\] respectively.
\end{enumerate}
\end{remark}

Regarding the assumption of the next lemma, and also of some later results, note that $\kappa$ will be regular in a generic ultrapower $V^\kappa/G$ obtained by forcing with a normal ideal on $\kappa$ if and only if $G$ contains the set of regular cardinals below $\kappa$. This is of course only possible if that latter set is a stationary subset of $\kappa$, i.e., if $\kappa$ is weakly Mahlo. Let us note that the assumption that $\kappa$ is regular in the generic ultrapower is needed to ensure that $j(\varphi)\res^{j(\kappa)}_\kappa$ is defined.

\begin{lemma}[{Cody-Holy \cite{cody_holy_2022}}]\label{lemma_j_of}
Suppose $\kappa$ is a regular cardinal and $\varphi$ is a $\Pi^1_\xi$ or $\Sigma^1_\xi$ formula over $V_\kappa$ for some $\xi<\kappa^+$. Whenever $I$ is a normal ideal on $\kappa$, $G\subseteq P(\kappa)/I$ is generic over $V$ and $j:V\to V^\kappa/G$ is the corresponding generic ultrapower such that $\kappa$ is regular in $V^\kappa/G$, it follows that $j(\varphi)\res^{j(\kappa)}_\kappa$ is $\Pi^1_\xi$ in $V^\kappa/G$ and furthermore,
\[j(\varphi)\res^{j(\kappa)}_\kappa=\varphi.\]
\end{lemma}

\begin{proof}
We proceed by induction on $\xi<\kappa^+$. By Remark \ref{remark_coding}(1) and the definition of the restriction operation, the case when $\xi<\kappa$ is easy since \[j(\varphi(A_1,\ldots,A_n))\res^{j(\kappa)}_\kappa=\varphi(j(A_1),\ldots,j(A_n))\res^{j(\kappa)}_\kappa=\varphi(A_1,\ldots,A_n).\] 
 
 Suppose $\xi=\zeta+1$ is a successor ordinal above $\kappa$ and $\varphi=\forall X\psi(X)$ where $\psi(X)$ is a $\Sigma^1_\zeta$ formula over $V_\kappa$. By Remark \ref{remark_coding}(2), 
 \[j(\varphi)\res^{j(\kappa)}_\kappa=j(\forall X\psi(X))\res^{j(\kappa)}_\kappa=\forall X j(\psi(X))\res^{j(\kappa)}_\kappa=\forall X\psi(X).\]
Essentially the same argument works when $\varphi=\exists X\psi(X)$ and $\psi(X)$ is a $\Pi^1_\zeta$ formula over $V_\kappa$.
  
Suppose $\xi\in\kappa^+\setminus\kappa$ is a limit and $\varphi=\bigwedge_{\zeta<\xi}\psi_\zeta$ is a $\Pi^1_\xi$ formula over $V_\kappa$. Let $\vec\psi=\langle\psi_\zeta\mid\zeta<\xi\rangle$, and let $\vec\pi=\langle\pi^\kappa_{\xi,\alpha}\mid\alpha<\kappa\rangle$. By elementarity $j(\vec{\pi})_\kappa$ is the transitive collapse of $j(F^\kappa_\xi)(\kappa)=j"\xi$ to $j(f^\kappa_\xi)(\kappa)=\xi$ and hence $j(\vec{\pi})_\kappa\restrict j"\xi=j^{-1}\restrict j"\xi$. Furthermore, For each $\zeta<\xi$ we have $j(\vec{\psi})_{j(\vec{\pi})_\kappa^{-1}(\zeta)}\res^{j(\kappa)}_\kappa=j(\vec{\psi})_{j(\zeta)}\res^{j(\kappa)}_\kappa=j(\psi_\zeta)\res^{j(\kappa)}_\kappa$, which is $\Pi^1_\zeta$ in $V^\kappa/G$ by our inductive hypothesis. Thus we have
  
\begin{align*}
j(\varphi)\res^{j(\kappa)}_\kappa&=\bigwedge_{\zeta<f^{j(\kappa)}_{j(\xi)}(\kappa)}j(\vec{\psi})_{j(\vec{\pi})_\kappa^{-1}(\zeta)}\res^{j(\kappa)}_\kappa\\
	&=\bigwedge_{\zeta<\xi}j(\psi_\zeta)\res^{j(\kappa)}_\kappa\\
	&=\varphi.
\end{align*}
The case when $\varphi$ is a $\Sigma^1_\xi$ formula is treated in exactly the same way.
\end{proof}

A nice feature of our definition of restriction is that it provides a convenient way to represent $\Pi^1_\xi$ and $\Sigma^1_\xi$ formulas in generic ultrapowers.

\begin{lemma}[{Cody-Holy \cite{cody_holy_2022}}]\label{lemma_represent}
Suppose $I$ is a normal ideal on $\kappa$, $G\subseteq P(\kappa)/I$ is generic over $V$, $j:V\to V^\kappa/G$ is the corresponding generic ultrapower and $\kappa$ is regular in $V^\kappa/G$. Suppose $\varphi$ is a $\Pi^1_\xi$ or $\Sigma^1_\xi$ formula over $V_\kappa$ for some $\xi<\kappa^+$ and let $\Phi:\kappa\to V_\kappa$ be such that $\Phi(\alpha)=\varphi\res^\kappa_\alpha$ for every regular $\alpha<\kappa$. Then, $\Phi$ represents $\varphi$ in $V^\kappa/G$. That is, $j(\Phi)(\kappa)=\varphi$.
\end{lemma}

\begin{proof}
This is an easy consequence of Lemma \ref{lemma_j_of} since \[j(\Phi)(\kappa)=j(\<\varphi\res^\kappa_\alpha\st\alpha\in\kappa\cap\REG\>)_\kappa=j(\varphi)\res^{j(\kappa)}_\kappa=\varphi.\]
\end{proof}


As an easy consequence of Lemma \ref{lemma_j_of}, we see that for each $\xi\in\kappa^+\setminus\kappa$, the definition of $\varphi\res^\kappa_\alpha$ where $\varphi$ is $\Pi^1_\xi$ or $\Sigma^1_\xi$ over $V_\kappa$ is independent of which bijection $b_{\kappa,\xi}$ is used in its computation, modulo the nonstationary ideal on $\kappa$.

\begin{corollary}\label{corollary_restriction_is_well_defined}

Suppose $\kappa$ is a regular cardinal, $\xi\in\kappa^+\setminus\kappa$ and $\varphi$ is $\Pi^1_\xi$ or $\Sigma^1_\xi$ formula over $V_\kappa$. Let $b_{\kappa,\xi}$ and $\bar{b}_{\kappa,\xi}$ be bijections from $\kappa$ to $\xi$, and let $\varphi\res^\kappa_\alpha$ and $\varphi\bar{\res}^\kappa_\alpha$ denote the restriction of $\varphi$ to a regular $\alpha<\kappa$ defined using $b_{\kappa,\xi}$ and $\bar{b}_{\kappa,\xi}$ respectively. Then there is a club $C\subseteq\kappa$ such that for all regular $\alpha\in C$ we have $\varphi\res^\kappa_\alpha=\varphi\bar\res^\kappa_\alpha$
\end{corollary}

\begin{proof}
To prove the existence of such a club we use Proposition \ref{proposition_framework2}. Let $G\subseteq P(\kappa)/\NS_\kappa$ be generic over $V$ such that $\kappa$ is regular in $V^\kappa/G$ and let $j:V\to V^\kappa/G$ be the corresponding generic ultrapower. Lemma \ref{lemma_j_of} implies that $j(\varphi)\res^{j(\kappa)}_\kappa=\varphi=j(\varphi)\bar{\res}^{j(\kappa)}_\kappa$.
\end{proof}

The following lemma was essentially established in an earlier version of the current article using a more complicated proof. The following simplified proof is due to Cody-Holy and appears in \cite{cody_holy_2022}.

\begin{lemma}\label{lemma_restriction_is_nice}
  Suppose $\kappa$ is weakly Mahlo. For any $\xi<\kappa^+$, if $\varphi$ is a $\Pi^1_\xi$ or $\Sigma^1_\xi$ formula over $V_\kappa$, then there is a club subset $C_\varphi$ of $\kappa$ such that for any regular $\alpha\in C_\varphi$, $\varphi\res^\kappa_\alpha$ is defined, and therefore a $\Pi^1_{f^\kappa_\xi(\alpha)}$  or $\Sigma^1_{f^\kappa_\xi(\alpha)}$ formula over $V_\alpha$ respectively by Remark \ref{remark_definition_of_restriction}. 
\end{lemma}
\begin{proof}
Suppose $\varphi$ is a $\Pi^1_\xi$ formula over $V_\kappa$. To prove the existence of $C_\varphi$ we use Proposition \ref{proposition_framework2}. Suppose $G\subseteq P(\kappa)/\NS_\kappa$ is generic over $V$ such that $\kappa$ is regular in $V^\kappa/G$ and let $j:V\to V^\kappa/G$ be the corresponding generic ultrapower. By Lemma~\ref{lemma_j_of}, $j(\varphi)\res^{j(\kappa)}_\kappa=\varphi$ is clearly defined and is a $\Pi^1_\xi$ formula in $V^\kappa/G$.
\end{proof}

\begin{definition}\label{definition_indescribability}
Suppose $\kappa$ is a cardinal and $\xi<\kappa^+$. A set $S\subseteq\kappa$ is \emph{$\Pi^1_\xi$-indescribable} if for every $\Pi^1_\xi$ sentence $\varphi$ over $V_\kappa$, if $V_\kappa\models\varphi$ then there is some $\alpha\in S$ such that $V_\alpha\models\varphi|^\kappa_\alpha.$
\end{definition}

It easily follows from Corollary \ref{corollary_restriction_is_well_defined} that the above notion of indescribability does not depend on which sequence $\<b_{\kappa,\xi}\st\xi\in\kappa^+\setminus\kappa\>$ is used to compute restrictions of formulas.

In Proposition \ref{proposition_measurable} below, we establish that the notion of indescribability given in Definition \ref{definition_indescribability} is relatively consistent by showing that every measurable cardinal $\kappa$ is $\Pi^1_\xi$-indescribable for all $\xi<\kappa^+$ and, in terms of consistency strength, the existence of a cardinal $\kappa$ which is $\Pi^1_\xi$-indescribable for all $\xi<\kappa^+$ is strictly weaker than the existence of a measurable cardinal.

\begin{proposition}\label{proposition_measurable}
Suppose $U$ is a normal measure on a measurable cardinal $\kappa$. Then $\kappa$ is $\Pi^1_\xi$-indescribable for all $\xi<\kappa^+$ and the set
\[X=\{\alpha<\kappa\st\text{$\alpha$ is $\Pi^1_\xi$-indescribable for all $\xi<\alpha^+$}\}\]
is in $U$.
\end{proposition}

\begin{proof}
Let $j:V\to M$ be the usual ultrapower embedding obtained from $U$ where $M$ is transitive and $j$ has critical point $\kappa$. Let us show that the set $X$ is in $U$; the fact that $\kappa$ is $\Pi^1_\xi$-indescribable for all $\xi<\kappa^+$ follows by a similar argument. Notice that it follows directly from Lemma \ref{lemma_restriction_is_nice} that for any $\xi<\kappa^+$ if $\varphi$ is any $\Pi^1_\xi$ formula over $V_\kappa$ then, in $M$, the formula $j(\varphi)\res^{j(\kappa)}_\kappa$ is $\Pi^1_\xi$ over $V_\kappa$ (because $\kappa\in j(C_\varphi)$). Furthermore, by Lemma \ref{lemma_j_of} we have 
\begin{align}
j(\varphi)\res^{j(\kappa)}_\kappa=\varphi.\label{equation_measurable}
\end{align}
It will suffice to show that, in $M$, $\kappa$ is $\Pi^1_\xi$-indescribable for all limit ordinals $\xi<\kappa^+$. Fix a limit ordinal $\xi<\kappa^+$ and suppose
\[(V_\kappa\models\varphi)^M\]
where $\varphi$ is $\Pi^1_\xi$ over $V_\kappa$ in $M$. Since $H(\kappa^+)^V= H(\kappa^+)^M$, we have $\varphi\in V$ and $\varphi$ is $\Pi^1_\xi$ over $V_\kappa$ in $V$. It follows by (\ref{equation_measurable}), that
\[\left((\exists\alpha<j(\kappa))\ V_\alpha\models j(\varphi)\res^{j(\kappa)}_\alpha\right)^M\]
and thus, by elementarity, there is some $\alpha<\kappa=\crit(j)$ such that
\[V_\alpha\models\varphi\res^\kappa_\alpha.\]
Thus, $(V_\alpha\models\varphi\res^\kappa_\alpha)^M$ and hence $\kappa$ is $\Pi^1_\xi$-indescribable in $M$.
\end{proof}

\section{Restricting $L_{\kappa^+,\kappa^+}$ formulas and consistency of $L_{\kappa^+,\kappa^+}$-indescribability}\label{section_other}


We will need to apply elementary embeddings to formulas of $L_{\kappa^+,\kappa^+}$, so let us consider some assumptions regarding the set-theoretic nature of these formulas. For example, we assume that if $j:V\to M$ is an elementary embedding with critical point $\kappa$ then $j(\lnot\varphi)=\lnot j(\varphi)$ and $j(\exists x\psi)=\exists j(\vec{x})\psi$; we also make additional assumption as in Remark \ref{remark_coding}(3) above, but we will not discuss this further.

Typically, when defining $L_{\kappa^+,\kappa^+}$ formulas, one begins by fixing a supply of $\kappa^+$-many variables that can be used to form $L_{\kappa^+,\kappa^+}$ sentences. However, without loss of generality, we will assume that we begin with a supply of $\kappa$-many variables $\{x_\eta\st\eta<\kappa\}$. This assumption does not weaken the expressive power of $L_{\kappa^+,\kappa^+}$ sentences or theories (without parameters) in the language of set theory, because any particular sentence defined using a supply of $\kappa^+$-many variables only actually mentions $\kappa$-many. This assumption allows us to write $L_{\kappa^+,\kappa^+}$ formulas beginning with existential quantifiers in the form $\exists\<x_{\alpha_\eta}\st\eta<\gamma\>\psi$, where the domain of the sequence of variables being quantified over is simply some $\gamma\leq\kappa$, rather than some $\xi<\kappa^+$ and $\<\alpha_\eta\st\eta<\gamma\>$ is some increasing sequence of ordinals less than $\kappa$. Another consequence of this assumption is that we may assume that all variables are elements of $V_\kappa$ when $\kappa=|V_\kappa|$, and hence if $j:V\to M$ is an elementary embedding with critical point $\kappa$, we will have $j(x)=x$ for all variables $x$. We will also assume that for all cardinals $\alpha<\kappa$ the set $\{x_\eta\st\eta<\alpha\}$ constitutes the supply of $\alpha$-many variables used to form all $L_{\alpha^+,\alpha^+}$ sentences.

For a regular cardinal $\kappa$ and an ordinal $\alpha<\kappa$, we define $\varphi\res^\kappa_\alpha$ for all $L_{\kappa^+,\kappa^+}$ formulas $\varphi$ by induction of subformulas. For more on such induction principles, see \cite[Page 64]{MR0539973}.

\begin{definition}\label{definition_restriction_2}
Suppose $\kappa$ is a regular cardinal and $\alpha<\kappa$ is a cardinal. We define $\varphi\res^\kappa_\alpha$ for all formulas $\varphi$ of $L_{\kappa^+,\kappa^+}$ in a given signature by induction on complexity of $\varphi$. 
\begin{enumerate}
\item If $\varphi$ is a term equation $t_1=t_2$ or a relational formula of the form $R(t_1,\ldots,t_k)$ we define $\varphi\res^\kappa_\alpha$ to be $\varphi$.
\item If $\varphi$ is of the form $\lnot\psi$ where $\psi\res^\kappa_\alpha$ has already been defined, we let $\varphi\res^\kappa_\alpha$ be the formula $\lnot(\psi\res^\kappa_\alpha)$.
\item If $\varphi$ is of the form $\bigwedge_{\zeta<\xi}\varphi_\zeta$ where $\xi<\kappa^+$ and $\varphi_\zeta\res^\kappa_\alpha$ has been defined for all $\zeta<\xi$, then we define
\[\varphi\res^\kappa_\alpha=\bigwedge_{\zeta < f^\kappa_\xi(\alpha)}\varphi_{(\pi^\kappa_{\xi,\alpha})^{-1}(\zeta)}\res^\kappa_\alpha,\]
provided that this definition of $\varphi\res^\kappa_\alpha$ is a formula of $L_{\alpha^+,\alpha^+}$; otherwise we leave $\varphi\res^\kappa_\alpha$ undefined.
\item If $\varphi$ is of the form $\exists\<x_{\alpha_\eta}\st\eta<\gamma\>\psi$ where $\gamma\leq\kappa$, $\<\alpha_\eta\st\eta<\gamma\>$ is an increasing sequence of ordinals less than $\kappa$ and $\psi\res^\kappa_\alpha$ has already been defined, we let 
\[\varphi\res^\kappa_\alpha=\exists\<x_{\alpha_\eta}\st\eta<\alpha\cap\gamma\>\ \psi\res^\kappa_\alpha,\]
provided that this definition of $\varphi^\kappa_\alpha$ is a formula of $L_{\alpha^+,\alpha^+}$; otherwise we leave $\varphi\res^\kappa_\alpha$ undefined.
\end{enumerate}
\end{definition}

As for the notion of restriction of $\Pi^1_\xi$ formulas considered in Section \ref{section_definition_pi1xi} above, one can easily show that for $L_{\kappa^+,\kappa^+}$ formulas, the definition of $\varphi\res^\kappa_\alpha$ is independent of our choice of bijections $\<b_{\kappa,\xi}\st\xi\in\kappa^+\setminus\kappa\>$ modulo the nonstationary ideal on $\kappa$. 

Notice that, in Definition \ref{definition_restriction_2}(4), it is at least conceivable that some of the bound variables of $\varphi$ could become free variables of $\varphi\res^\kappa_\alpha$. However, it easily follows from the next lemma that this can happen only for a nonstationary set of $\alpha$, and hence this aspect of the definition can be ignored in all of the cases that we care about.

\begin{lemma}\label{lemma_alternative_indescribability}
Suppose $\kappa$ is a regular cardinal and $\varphi$ is an $L_{\kappa^+,\kappa^+}$ formula in the language of set theory. If $I$ is a normal ideal on $\kappa$, $G\subseteq P(\kappa)/I$ is generic over $V$ and $j:V\to V^\kappa/G$ is the corresponding generic ultrapower such that $\kappa$ is regular in $V^\kappa/G$, then it follows that $j(\varphi)\res^{j(\kappa)}_\kappa=\varphi$ is a formula of $L_{\kappa^+,\kappa^+}$ in $V^\kappa/G$.

\end{lemma}

\begin{proof}
When $\varphi$ is a relational formula, it follows by our assumptions on the set-theoretic nature of such formulas that $j(\varphi)\res^{j(\kappa)}_\kappa=\varphi\res^{j(\kappa)}_\kappa=\varphi$. If the result holds for $\psi$ and $\varphi$ is of the form $\lnot \varphi$, then it clearly holds for $\varphi$ too. 

Now suppose $\varphi$ is of the form
\[\varphi=\bigwedge_{\zeta<\xi}\varphi_\zeta,\]
where $\xi<\kappa^+$. Define sequences $\vec{\varphi}=\<\varphi_\zeta\st\zeta<\xi\>$ and $\vec{\pi}=\<\pi^\kappa_{\xi,\alpha}\st\alpha<\kappa\>$. Recall that $j(\vec{\pi})_\kappa \restrict j"\xi=j^{-1}\restrict j"\xi$. We have
\begin{align*}
j(\varphi)=\bigwedge_{\zeta<j(\xi)} j(\vec{\varphi})_\zeta
\end{align*}
and thus
\begin{align*}
j(\varphi)\res^{j(\kappa)}_\kappa&=\bigwedge_{\zeta<j(f^\kappa_\xi)(\kappa)} j(\vec{\varphi})_{j(\vec{\pi})_\kappa^{-1}(\zeta)}\res^{j(\kappa)}_\kappa\\
  &=\bigwedge_{\zeta<\xi} j(\vec{\varphi})_{j(\zeta)}\res^{j(\kappa)}_\kappa\\
  &=\bigwedge_{\zeta<\xi} j(\varphi_\zeta)\res^{j(\kappa)}_\kappa\\
  &=\bigwedge_{\zeta<\xi}\varphi_\zeta\\
  &=\varphi.
\end{align*}
Now suppose $\varphi$ is of the form $\exists\<x_{\alpha_\eta}\st\eta<\gamma\>\psi$ where $\gamma\leq\kappa$ and $\<\alpha_\eta\st\eta<\gamma\>$ is an increasing sequence of ordinals less than $\kappa$. Let $\vec{x}=\<x_{\alpha_\eta}\st\eta<\gamma\>$. We have
\[j(\varphi)=\exists j(\vec{x}) j(\psi)\]
and thus
\begin{align*}
j(\varphi)\res^{j(\kappa)}_\kappa&=\exists j(\vec{x})\restrict(\kappa\cap j(\gamma)) \ j(\psi)\res^{j(\kappa)}_\kappa\\
  &=\exists\vec{x}\psi\\
  &=\varphi.
\end{align*}

\end{proof}

One can easily show that the following definition of $L_{\kappa^+,\kappa^+}$-indescribability is not dependent on which sequence of bijections is used to compute restrictions of $L_{\kappa^+,\kappa^+}$ formulas.

\begin{definition}
Suppose $\kappa$ is a regular cardinal. A set $S\subseteq\kappa$ is  \emph{$L_{\kappa^+,\kappa^+}$-indescribable} if for all sentences $\varphi$ of $L_{\kappa^+,\kappa^+}$ in the language of set theory, if $V_\kappa\models\varphi$ then there is some $\alpha<\kappa$ such that $V_\alpha\models\varphi\res^\kappa_\alpha$.
\end{definition}

From Lemma \ref{lemma_alternative_indescribability} and an argument similar to that given above for Proposition \ref{proposition_measurable} we obtain the following, which shows that the existence of a cardinal $\kappa$ which is $L_{\kappa^+,\kappa^+}$-indescribable is strictly weaker than the existence of a measurable cardinal.

\begin{proposition}\label{proposition_Lkappa+}
Suppose $U$ is a normal measure on a measurable cardinal $\kappa$. Then $\kappa$ is $L_{\kappa^+,\kappa^+}$-indescribable and the set
\[\{\alpha<\kappa\st\text{$\alpha$ is $L_{\alpha^+,\alpha^+}$-indescribable}\}\]
is in $U$.
\end{proposition}

\section{Higher $\Pi^1_\xi$-indescribability ideals}

In this section, given a regular cardinal $\kappa$, we prove the existence of universal $\Pi^1_\xi$ formulas for all $\xi<\kappa^+$ and use such formulas to show that the natural ideal on $\kappa$ associated to $\Pi^1_\xi$-indescribability is normal. We then use universal formulas to show that $\Pi^1_\xi$-indescribability is, in a sense, expressible by a $\Pi^1_{\xi+1}$ formula. This leads to several hierarchy results and a characterization of $\Pi^1_\xi$-indescribability in terms of the natural filter base consisting of the $\Pi^1_\xi$-club subsets of $\kappa$.

\begin{remark}\label{remark_diagonal}
Let us make a brief remark about normal ideals and a notion of diagonal intersection we will use in several places below. Recall that an ideal $I$ on a regular cardinal $\kappa$ is \emph{normal} if and only if for any positive set $S\in I^+=\{X\subseteq\kappa\st X\notin I\}$ and every function $f:S\to\kappa$ with $f(\alpha)<\alpha$ for all $\alpha\in S$, there is a positive set $T\in P(S)\cap I^+$ such that $f$ is constant on $T$. Equivalently, $I$ is normal if and only if the filter $I^*$ dual to $I$ is closed under diagonal intersection; that is, whenever $\vec{C}=\<C_\alpha\st\alpha<\kappa\>$ is a sequences of sets in $I^*$ then $\di\vec{C}=\di_{\alpha<\kappa}C_\alpha=\{\alpha<\kappa\st\alpha\in\bigcap_{\beta<\alpha}C_\beta\}$ is in $I^*$. Since diagonal intersections are independent, modulo the nonstationary ideal, of the particular enumeration of the sets involved, it follows that an ideal $I$ on $\kappa$ is normal if and only if for all $\xi\in\kappa^+\setminus\kappa$ whenever $\vec{C}=\<C_\zeta\st\zeta<\xi\>$ is a sequence of sets in $I^*$, the set
\[\di\vec{C}=\di_{\zeta<\xi}C_\zeta=\{\alpha<\kappa\st \alpha\in \bigcap_{\zeta\in b_{\kappa,\xi}[\alpha]}C_\zeta\}\]
is in $I^*$, where $b_{\kappa,\xi}:\kappa\to\xi$ is a bijection. In what follows, we will often make use of the fact that the club filter on a regular $\kappa$ is closed under such diagonal intersections.
\end{remark}

\subsection{Universal $\Pi^1_\xi$ formulas and normal ideals}\label{section_universal}

For a regular cardinal $\kappa>\omega$, we define the notion of universal $\Pi^1_\xi$ formula, where $\xi<\kappa^+$, as follows. If $\xi<\kappa$ then we adopt a definition of universal $\Pi^1_\xi$ formula over $V_\kappa$, which is similar to that of \cite{MR3894041}, but we need a different notion for $\xi\in\kappa^+\setminus\kappa$.

\begin{definition}\label{definition_universal}
Suppose $\kappa$ is a regular cardinal and $\xi<\kappa$. We say that a $\Pi^1_\xi$ formula $\Psi(X_1,\ldots,X_n,Y_\xi)$ over $V_\kappa$, where $X_1,\ldots,X_n,Y_\xi$ are second-order variables, is a \emph{universal $\Pi^1_\xi$ formula at $\kappa$ for formulas with $n$ free variables} if for all $\Pi^1_\xi$ formulas $\varphi(X_1,\ldots,X_n)$ over $V_\kappa$, with all free variables displayed, there is a $K_\varphi\in V_\kappa$, referred to as a \emph{code for $\varphi$} such that for all $A_1,\ldots,A_n\subseteq V_\kappa$ and all regular $\alpha\in \kappa\setminus\xi$ we have
\[V_\alpha\models\varphi(A_1,\ldots,A_n)\text{ if and only if }V_\alpha\models\Psi(A_1,\ldots,A_n,K_\varphi).\]

On the other hand, suppose $\xi\in\kappa^+\setminus\kappa$. We say that a $\Pi^1_\xi$ formula $\Psi(X_1,\ldots,X_n,Y_\xi)$ over $V_\kappa$, where $X_1,\ldots,X_n,Y_\xi$ are second-order variables, is a \emph{universal $\Pi^1_\xi$ formula at $\kappa$ for formulas with $n$ free second order variables} if for all $\Pi^1_\xi$ formulas $\varphi(X_1,\ldots,X_n)$ over $V_\kappa$, with all free variables displayed, there is a $K_\varphi\subseteq\kappa$ and there is a club $C_\varphi\subseteq\kappa$ such that for all $A_1,\ldots,A_n\subseteq V_\kappa$ and all regular $\alpha\in C_\varphi\cup\{\kappa\}$ we have
\[V_\alpha\models\varphi(A_1,\ldots,A_n)\res^\kappa_\alpha\text{ if and only if }V_\alpha\models\Psi(A_1,\ldots,A_n,K_\varphi)\res^\kappa_\alpha.\]
When $n=0$, the intended meaning is that $\varphi$ is a $\Pi^1_\xi$ sentence over $V_\kappa$ and $\Psi_{\xi,0}(Y)$ has one free-variable. 
The notion of \emph{universal $\Sigma^1_\xi$ formula at $\kappa$ for formulas with $n$ free second-order variables} is defined similarly.
\end{definition}

We will use the following lemma to prove that universal $\Pi^1_\xi$ formulas exist at regular $\kappa$ where $\kappa\leq\xi<\kappa^+$.

\begin{lemma}\label{lemma_no_increase}
Suppose $\kappa$ is regular and $1\leq\zeta<\kappa^+$. Suppose $\psi_\zeta(W_1,\ldots,W_n,Y,Z)$ is a $\Pi^1_\zeta$ formula over $V_\kappa$ and $\varphi(X,Y)$ is a $\Pi^1_0$ formula over $V_\kappa$ where all free second-order variables are displayed. Then there is a $\Pi^1_\zeta$ formula $\varphi_\zeta(X,Z)$ over $V_\kappa$ and a club $C_\zeta$ in $\kappa$ such that for all $A,B\subseteq V_\kappa$ and for all regular $\alpha\in C_\zeta\cup\{\kappa\}$ we have
\[V_\alpha\models\forall Y\forall W_1\cdots\forall W_n(\varphi(A\cap V_\alpha,Y)\lor\psi_\zeta(W_1,\ldots, W_n,Y,B)\res^\kappa_\alpha)\]
if and only if 
\[V_\alpha\models \varphi_\zeta(A,B)\res^\kappa_\alpha.\]
Furthermore, a similar statement holds for $\Sigma^1_\zeta$ formulas $\psi_\zeta'$ over $V_\kappa$. 
\end{lemma}

\begin{proof}
We provide a proof for the cases in which $\psi_\zeta$ is a $\Pi^1_\zeta$ formula over $V_\kappa$. The other case in which the formulas are $\Sigma^1_\zeta$ is similar. We proceed by induction on $\zeta$. If $\zeta=1$ then $\psi_1(W_1,\ldots, W_n,Y,Z)$ is of the form $\forall W\psi_0(W_1,\ldots, W_n,W,Y,Z)$ where $\psi_0(W_1,\ldots,W_n,W,Y,Z)$ is $\Pi^1_0$ over $V_\kappa$ and we see that 
\[\forall Y\forall W_1\cdots W_n(\varphi(X,Z)\lor\psi_1(W_1,\ldots,W_n,Y,Z))\]
is equivalent over $V_\kappa$ to the $\Pi^1_1$ formula 
\[\varphi_1(X,Z)=\forall Y\forall W_1\cdots\forall W_n\forall W(\varphi(X,Y)\lor\psi_0(W_1,\ldots,W_n,W,Y,Z)).\] 
Since restrictions of $\Pi^1_1$ formulas are trivial, this establishes the base case taking $C_1=\kappa$.

If $\zeta=\eta+1<\kappa^+$ is a successor ordinal, then $\psi_{\eta+1}(W_1,\ldots,W_n,Y,Z)$ is of the form $\forall W \psi_\eta'(W_1,\ldots,W_n,W,Y,Z)$ where $\psi_\eta'$ is $\Sigma^1_\eta$ over $V_\kappa$. Clearly the formula
\[\varphi_{\eta+1}:=\forall Y\forall W_1\cdots\forall W_n\forall W(\varphi(X,Y)\lor \psi_\eta'(W_1,\ldots,W_n,W,Y,Z))\]
is $\Pi^1_{\eta+1}$ over $V_\kappa$ and satisfies the desired property together with the club $C_{\eta+1}=C_\eta$ obtained from the inductive hypothesis.

If $\zeta<\kappa^+$ is a limit ordinal, then 
\[\psi_\zeta(W_1,\ldots,W_n,Y,Z)=\bigwedge_{\eta<\zeta}\psi_\eta(W_1,\ldots,W_n,Y,Z)\]
where $\psi_\eta$ is $\Pi^1_\eta$ over $V_\kappa$ for all $\eta<\zeta$. In this case, the formula 
\[\forall Y\forall W_1\cdots\forall W_n(\varphi(X,Y)\lor\psi_\zeta(W_1,\ldots,W_n,Y,Z))\]
is equivalent over $V_\kappa$ to
\[\bigwedge_{\eta<\zeta}\forall Y\forall W_1\cdots\forall W_n(\varphi(X,Y)\lor\psi_\eta(W_1,\ldots,W_n,Y,Z)),\]
and by our inductive hypothesis, for each $\eta<\zeta$, there is a $\Pi^1_\eta$ formula $\varphi_\eta(X,Z)$ over $V_\kappa$ and a club $C_\eta$ in $\kappa$ such that for all $A,B\subseteq V_\kappa$ and all regular $\alpha\in C_\eta\cup\{\kappa\}$ we have
\[V_\alpha\models\forall Y\forall W_1\cdots\forall W_n(\varphi(A,Y)\lor\psi_\eta(W_1,\ldots, W_n,Y,B)\res^\kappa_\alpha)\]
if and only if
\[V_\alpha\models\varphi_\eta(A,B)\res^\kappa_\alpha.\]
It is easy to verify that the formula
\[\varphi_\zeta(X,Y)=\bigwedge_{\eta<\zeta}\varphi_\eta(X,Y)\]
and a club subset $C_\zeta$ of $\di_{\eta<\zeta}C_\eta=\{\alpha<\kappa\st\alpha\in\bigcap_{\eta\in F^\kappa_\zeta(\alpha)}C_\eta\}$ are as desired.\footnote{Recall that $C_\zeta$ is in fact in the club filter on $\mu$ by Remark \ref{remark_diagonal}.}
\end{proof}

The following proposition generalizes results of L\'{e}vy \cite{MR0281606} and Bagaria \cite{MR3894041}; Levy proved the case in which $\xi<\omega$ and Bagaria proved the case in which $\xi<\kappa$.

\begin{theorem}\label{theorem_universal}
Suppose $\kappa>\omega$ is a regular cardinal and $\xi$ is an ordinal with $\xi<\kappa^+$. For each $n<\omega$ there is a universal $\Pi^1_\xi$ formula $\Psi^\kappa_{\xi,n}(X_1,\ldots,X_n,Y_\xi)$ and a universal $\Sigma^1_\xi$ formula $\bar{\Psi}^\kappa_{\xi,n}(X_1,\ldots,X_n,Y_\xi)$ at $\kappa$ for formulas with $n$ free second-order variables.
\end{theorem}

\begin{proof}
The case in which $\xi<\kappa$ follows directly from the proof of \cite[Proposition 4.4]{MR3894041}. Suppose $\xi=\zeta+1$ is a successor ordinal with $\kappa<\zeta+1<\kappa^+$ and the result holds for all $\eta\leq\zeta$. Let us show that there is a universal $\Pi^1_{\zeta+1}$ formula at $\kappa$ for formulas with $n$ free second-order variables; a similar argument works for $\Sigma^1_{\zeta+1}$ formulas, which we leave to the reader. Let $\bar{\Psi}_{\zeta,n+1}^\kappa(X_1,\ldots,X_n,X_{n+1},Y_\xi)$ be a universal $\Sigma^1_\zeta$ formula at $\kappa$ for formulas with $n+1$ free second-order variables obtained from the induction hypothesis. We will show that
\[\Psi_{\zeta+1,n}^\kappa(X_1,\ldots,X_n,Y_\xi)=\forall W\bar{\Psi}_{\zeta,n+1}^\kappa(X_1,\ldots,X_n,W,Y_\xi)\]
is the desired formula. Suppose $\varphi(X_1,\ldots,X_n)=\forall W\varphi_\zeta(X_1,\ldots,X_n,W)$ is any $\Pi^1_{\zeta+1}$ formula with $n$ free second-order variables, where $\varphi_\zeta$ is $\Sigma^1_\zeta$ with $n+1$ free second-order variables.\footnote{Note that if we had here a block of quantifiers $\forall W_1\cdots\forall W_k$, they could be collapsed to a single one by modifying $\varphi_\zeta$ without changing the fact that $\varphi_\zeta$ is $\Sigma^1_\zeta$.} Let $C_{\varphi_\zeta}$ and $K_{\varphi_\zeta}$ be as obtained from the inductive hypothesis. Fix $A_1,\ldots,A_n\subseteq V_\kappa$. Then for all regular $\alpha\in C_{\varphi_\zeta}\cup\{\kappa\}$ we have
\begin{align*}
V_\alpha\models\varphi(A_1,\ldots,A_n)\res^\kappa_\alpha&\iff (\forall W\subseteq V_\alpha) V_\alpha\models \varphi_\zeta(A_1,\ldots,A_n,W)\res^\kappa_\alpha\\
	&\iff (\forall W\subseteq V_\alpha) V_\alpha\models \bar{\Psi}_{\zeta,n+1}^\kappa(A_1,\ldots,A_n,W,K_{\varphi_\zeta})\res^\kappa_\alpha\\
	&\iff V_\alpha\models \forall W\bar{\Psi}_{\zeta,n+1}^\kappa(A_1,\ldots,A_n,W,K_{\varphi_\zeta})\res^\kappa_\alpha\\
	&\iff V_\alpha\models\Psi_{\zeta+1,n}^\kappa(A_1,\ldots,A_n, K_{\varphi_\zeta})\res^\kappa_\alpha,
\end{align*}
which establishes the successor case of the induction.

Suppose $\xi$ is a limit ordinal with $\kappa\leq\xi<\kappa^+$ and the result holds for all $\zeta<\xi$. We will show that there is a universal $\Pi^1_\xi$ formula at $\kappa$ for formulas with $1$ free second-order variable; the proof for $n$ free second-order variables is essentially the same but one must replace the single variable $X$ with a tuple $X_1,\ldots,X_n$ in the appropriate places. We let $\Gamma:\kappa\times\kappa\to\kappa$ be the usual definable pairing function and for $A\subseteq\kappa$ and $\eta<\kappa$ we let 
\[(A)_\eta=\{\beta<\kappa\st\Gamma(\eta,\beta)\in A\}\]
be the ``$\eta^{th}$ slice'' of $A$.

Suppose $\zeta<\xi$. Using the inductive hypothesis, we let $\Psi^\kappa_{\zeta,1}$ be a universal $\Pi^1_\zeta$ formula at $\kappa$ for formulas with $1$ free variable. We will define the desired universal formula $\Psi_{\xi,1}^\kappa(X,Y_\xi)$ by simply taking the conjunction of the $\Psi^\kappa_{\zeta,1}$'s for $\zeta<\xi$, with the proviso that we must take care to use the right slice of the code $K_\varphi\subseteq\kappa$ we will define for an arbitrary $\Pi^1_\xi$ formula $\varphi=\bigwedge_{\zeta<\xi}\varphi_\zeta$, where $(K_\varphi)_{b_{\kappa,\xi}^{-1}(\zeta)}=K_{\varphi_\zeta}$. With this in mind, note that we will define $\Psi_{\xi,1}^\kappa(X,Y_\xi)$ in such a way that it is equivalent to
\[\bigwedge_{\zeta<\xi}\Psi^\kappa_{\zeta,1}(X,(Y_\xi)_{b_{\kappa,\xi}^{-1}(\zeta)})\]
over $V_\kappa$. In order to verify that the following definition of $\Psi^\kappa_{\xi,1}$ produces a $\Pi^1_\xi$ formula over $V_\kappa$, we must check that $\Psi^\kappa_{\zeta,1}(X,(Y_\xi)_{b_{\kappa,\xi}^{-1}(\zeta)})$ is expressible by a $\Pi^1_\zeta$ formula over $V_\kappa$.

Suppose $\zeta<\xi$. Notice that for $A\subseteq V_\kappa$ and $B\subseteq\kappa$ the sentence $\Psi^\kappa_{\zeta,1}(A,(B)_{b_{\kappa,\xi}^{-1}(\zeta)})$ is equivalent to
\[\forall Y(Y=(B)_{b_{\kappa,\xi}^{-1}(\zeta)}\rightarrow \Psi^\kappa_{\zeta,1}(A,(B)_{b_{\kappa,\xi}^{-1}(\zeta)})),\]
over $V_\kappa$ where $Y=(B)_{b_{\kappa,\xi}^{-1}(\zeta)}$ is expressible as a $\Pi^1_0$ formula over $V_\kappa$ using $b_{\kappa,\xi}^{-1}(\zeta)$ as a parameter and using a first order quantifier over ordinals. Thus, by Lemma \ref{lemma_no_increase}, we may let $\Theta^\kappa_\zeta(X,Y_\xi)$ be a $\Pi^1_\xi$ formula over $V_\kappa$ such that for all $A\subseteq V_\kappa$ and $B\subseteq\kappa$, $V_\kappa\models\Theta_\zeta^\kappa(A,B)$ if and only if $V_\kappa\models \Psi^\kappa_{\zeta,1}(A,(B)_{b_{\kappa,\xi}^{-1}(\zeta)})$, and furthermore, we can assume $\Theta_\zeta^\kappa$ has the property that there is a club $D_\zeta$ in $\kappa$ such that for all regular $\alpha\in D_\zeta$ and all $A,B\subseteq V_\kappa$ we have
\[V_\alpha\models \Psi^\kappa_{\zeta,1}(A,(B)_{b_{\kappa,\xi}^{-1}(\zeta)})\res^\kappa_\alpha\text{ if and only if }V_\kappa\models\Theta_\zeta^\kappa(A,B)\res^\kappa_\alpha.\]

Now let us check that 
\[\Psi_{\xi,1}^\kappa(X,Y_\xi)=\bigwedge_{\zeta<\xi}\Theta_\zeta^\kappa(X,Y_\xi)\]
is a universal $\Pi^1_\xi$ formula at $\kappa$ for formulas with $1$ free variable. Suppose $\varphi(X)=\bigwedge_{\zeta<\xi}\varphi_\zeta(X)$ is any $\Pi^1_\xi$ formula over $V_\kappa$ with one free second-order variable, where each $\varphi_\zeta$ is $\Pi^1_\zeta$ over $V_\kappa$. Let $K_\varphi=\{\Gamma(b_{\kappa,\xi}^{-1}(\zeta),\beta)\st\zeta<\xi\land\beta\in K_{\varphi_\zeta}\}$ code the sequence $\<K_{\varphi_\zeta}\st\zeta<\xi\>$, where the codes $K_{\varphi_\zeta}$ are obtained by the inductive hypothesis. Notice that for each $\zeta<\xi$ we have $(K_\varphi)_{b_{\kappa,\xi}^{-1}(\zeta)}=K_{\varphi_\zeta}$.


Fix $A\subseteq V_\kappa$. It follows easily from the definitions of $K_\varphi$ and $\Psi_{\xi,1}^\kappa(X,Y_\xi)$ that 
\begin{align*}
V_\kappa\models\Psi_{\xi,1}^\kappa(A,K_\varphi)&\iff V_\kappa\models\bigwedge_{\zeta<\xi}\Psi^\kappa_{\zeta,1}(A,(K_{\varphi})_{b_{\kappa,\xi}^{-1}(\zeta)})\\
  &\iff V_\kappa\models\bigwedge_{\zeta<\xi}\Psi^\kappa_{\zeta,1}(A,(K_{\varphi_\zeta}))\\
&\iff V_\kappa\models\varphi(A).
\end{align*}

Next let us show that there is a club $C$ in $\kappa$ such that for all regular $\alpha\in C$ we have $V_\alpha\models\varphi(A)\res^\kappa_\alpha$ if and only if $V_\alpha\models\Psi_{\xi,1}^\kappa(A,K_\varphi)\res^\kappa_\alpha$. To prove that such a club exists we use Proposition \ref{proposition_framework2}. Let $G\subseteq P(\kappa)/\NS_\kappa$ be generic over $V$ such that $\kappa$ is regular in $V^\kappa/G$ and let $j:V\to V^\kappa/G$ be the corresponding generic ultrapower embedding. We must show that $\kappa\in j(T)$ where
\begin{align}
T=\{\alpha\in\REG\st V_\alpha\models\varphi(A)\res^\kappa_\alpha\iff V_\alpha\models\Psi_{\xi,1}^\kappa(A,K_\varphi)\res^\kappa_\alpha\}.
\end{align}
By Lemma \ref{lemma_j_of} we have
\[j(\varphi(A))\res^{j(\kappa)}_\kappa=\varphi(A)=\bigwedge_{\zeta<\xi}\varphi_\zeta(A)\]
and
\[j(\Psi_{\xi,1}^\kappa(A,K_\varphi))\res^{j(\kappa)}_\kappa=\Psi_{\xi,1}^\kappa(A,K_\varphi)=\bigwedge_{\zeta<\xi}\Theta_\zeta^\kappa(A,K_\varphi).\]
By our inductive hypothesis, for each $\zeta<\xi$ there is a club $C_{\varphi_\zeta}$ in $\kappa$ such that for all regular $\alpha\in C_{\varphi_\zeta}$ we have $V_\alpha\models\varphi_\zeta(A)\res^\kappa_\alpha$ if and only if $V_\alpha\models\Psi_{\zeta,1}^\kappa(A,K_{\varphi_\zeta})\res^\kappa_\alpha$. Since for all $\zeta<\xi$ we have $\kappa\in j(C_{\varphi_\zeta}\cap D_\zeta)$ and $(K_\varphi)_{b_{\kappa,\xi}^{-1}(\zeta)}=K_{\varphi_\zeta}$, it follows that in $V^\kappa/G$ we have
\begin{align*}
V_\kappa\models\varphi_\zeta(A)&\iff
V_\kappa\models\Psi_{\zeta,1}^\kappa(A,(K_\varphi)_{b_{\kappa,\xi}^{-1}(\zeta)})\\
  &\iff V_\kappa\models\Theta_\zeta^\kappa(A,K_\varphi).
\end{align*}
Hence $\kappa\in j(T)$.
\end{proof}

\begin{remark}\label{remark_universal_formula_depends_on_bijections}
Notice that contained within the proof of Theorem \ref{theorem_universal} is a construction via transfinite recursion on $\xi$ of a universal $\Pi^1_\xi$ formula at $\kappa$ for formulas with $n$ free second-order variables. Furthermore, when $\xi$ is a limit, let us emphasize that the definition of $\Psi^\kappa_{\xi,n}(X_1,\ldots,X_n,Y_\xi)$ depends not only on the chosen bijection $b_{\kappa,\xi}:\kappa\to\xi$, but on the entire history of bijections $b_{\kappa,\zeta}:\kappa\to\zeta$ chosen at previous limit steps $\zeta<\xi$ in the construction.
\end{remark}

Generalizing work of Bagaria \cite{MR3894041}, as our first application of the existence of universal formulas, we show that there are natural normal ideals on $\kappa$ associated to $\Pi^1_\xi$-indescribability for all $\xi<\kappa^+$.

\begin{theorem}\label{theorem_normal_ideal}
If a cardinal $\kappa$ is $\Pi^1_\xi$-indescribable where $\xi<\kappa^+$, then the collection
\[\Pi^1_\xi(\kappa)=\{X\subseteq\kappa\st\text{$X$ is not $\Pi^1_\xi$-indescribable}\}\]
is a nontrivial normal ideal on $\kappa$.
\end{theorem}

\begin{proof}
Suppose $\kappa$ is $\Pi^1_\xi$-indescribable where $\xi<\kappa^+$. It is easy to see that 
\[\Pi^1_\xi(\kappa)=\{X\subseteq\kappa\st\text{$X$ is not $\Pi^1_\xi$-indescribable}\}\]
is a nontrivial ideal on $\kappa$, so we just need to prove it is normal. Suppose $S\in\Pi^1_\xi(\kappa)^+$ and fix a regressive function $f:S\to\kappa$. For the sake of contradiction, assume that for all $\eta<\kappa$ the set $f^{-1}(\{\eta\})=\{\alpha\in S\st f(\alpha)=\eta\}$ is not in $\Pi^1_\xi(\kappa)^+$. Then, for each $\eta<\kappa$ there is some $\Pi^1_\xi$ formula $\varphi_\eta(X)$ over $V_\kappa$ and some $A_\eta\subseteq V_\kappa$ such that $V_\kappa\models\varphi_\eta(A_\eta)$ but 
\begin{align}V_\alpha\models\lnot\varphi_\eta(A_\eta)\res^\kappa_\alpha\text{ for all }\alpha\in S \text{ such that }f(\alpha)=\eta.\label{equation_will_contradict}
\end{align}
Let $\Psi^\kappa_{\xi,1}(X,Y_\xi)$ be the universal $\Pi^1_\xi$ formula at $\kappa$ for formulas with one free second-order variable, let $K_{\varphi_\eta}\subseteq\kappa$ be the code for $\varphi_\eta$ and let $C_{\varphi_\eta}$ be the club subset of $\kappa$ as in Definition \ref{definition_universal}. Then for all $\eta<\kappa$ we have
\[V_\kappa\models\Psi^\kappa_{\xi,1}(A_\eta,K_{\varphi_\eta}).\]

We would like to show that the formula $\bigwedge_{\eta<\kappa}\Psi^\kappa_{\xi,1}(A_\eta,K_{\varphi_\eta})$ is equivalent to a single $\Pi^1_\xi$ formula over $V_\kappa$. Let $A=\{\Gamma(\eta,\beta)\st \eta<\kappa\land\beta\in A_\eta\}\subseteq\kappa$ and $K=\{\Gamma(\eta,\beta)\st\eta<\kappa\land\beta\in K_{\varphi_\eta}\}\subseteq\kappa$ code the sequences $\<A_\eta\st\eta<\kappa\>$ and $\<K_{\varphi_\eta}\st\eta<\kappa\>$ respectively. Let
\[C=\di_{\eta<\kappa}C_{\varphi_\eta}=\{\zeta<\kappa\st\zeta\in\bigcap_{\eta<\zeta}C_{\varphi_\eta}\}\]
and notice that $C$ is in the club filter on $\kappa$. By a straightforward application of Lemma \ref{lemma_no_increase}, there is a $\Pi^1_\xi$ sentence $\varphi(A,K,C)$ such that 
\[V_\kappa\models\varphi(A,K,C) \text{ if and only if } V_\kappa\models\bigwedge_{\eta<\kappa}\Psi^\kappa_{\xi,1}(A_\eta,K_{\varphi_\eta}),\]
and furthermore, there is a club $D\subseteq\kappa$ such that for all regular $\alpha\in D$ we have
\[V_\alpha\models\varphi(A,K,C)\res^\kappa_\alpha\text{ if and only if }V_\alpha\models\bigwedge_{\eta<\alpha}\Psi^\kappa_{\xi,1}(A_\eta,K_{\varphi_\eta})\res^\kappa_\alpha.\]
Since $S$ is $\Pi^1_\xi$-indescribable in $\kappa$, there is some regular $\alpha\in S\cap C\cap D$ such that $V_\alpha\models\varphi(A,K,C)\res^\kappa_\alpha$. Since $\alpha\in D$ we have $V_\alpha\models \bigwedge_{\eta<\alpha}\Psi_{\xi,1}^\kappa(A_\eta,K_{\varphi_\eta})\res^\kappa_\alpha$ and since $\alpha\in C$ we have $V_\alpha\models\bigwedge_{\eta<\alpha} \varphi_\eta(A_\eta)$, which contradicts (\ref{equation_will_contradict}) since $f(\alpha)<\alpha$.
\end{proof}

As an easy consequence of Theorem \ref{theorem_normal_ideal} we obtain the following a characterization of $\Pi^1_\xi$-indescribable subsets of a cardinal in terms of generic elementary embeddings, which we will use below to characterize $\Pi^1_\xi$-indescribable sets in terms of a natural filter base (see Theorem \ref{theorem_xi_clubs}(1)).

\begin{proposition}\label{proposition_generic_characterization}
Suppose $\kappa$ is a regular cardinal, $\xi<\kappa^+$ and $S\subseteq\kappa$. The following are equivalent.
\begin{enumerate}
\item The set $S\subseteq\kappa$ is $\Pi^1_\xi$-indescribable in $\kappa$.
\item There is some poset $\P$ such that whenever $G\subseteq \P$ is generic over $V$, there is an elementary embedding $j:V\to M\subseteq V[G]$ in $V[G]$ with critical point $\kappa$ such that 
\begin{enumerate}
\item $\kappa\in j(S)$ and 
\item for all $\Pi^1_\xi$ sentences $\varphi$ over $V_\kappa$ in $V$ we have $j(\varphi)\res^{j(\kappa)}_\kappa=\varphi$ and 
\[(V_\kappa\models\varphi)^V \implies (V_\kappa\models\varphi)^M.\]
\end{enumerate}
\end{enumerate}
\end{proposition}

\begin{proof}
Suppose $S\subseteq\kappa$ is $\Pi^1_\xi$-indescribable in $\kappa$. Let $G\subseteq P(\kappa)/\Pi^1_\xi(\kappa)$ be generic over $V$ with $S\in G$ and let $j:V\to V^\kappa/G$ be the corresponding generic ultrapower embedding. Note that the normality of the ideal $\Pi^1_\xi(\kappa)$ implies that the critical point of $j$ is $\kappa$ and $\kappa\in j(S)$. Suppose $\varphi$ is a $\Pi^1_\xi$ sentence over $V_\kappa$ and $V_\kappa\models\varphi$. Then the set 
\[C=\{\alpha<\kappa\st\text{$\varphi\res^\kappa_\alpha$ is defined and $V_\alpha\models\varphi\res^\kappa_\alpha$}\}\]
is in the filter dual to $\Pi^1_\xi(\kappa)$. Thus $\kappa\in j(C)$ and since $j(\varphi)\res^{j(\kappa)}_\kappa=\varphi$ by Lemma \ref{lemma_j_of}, we see that (2b) holds.

Conversely, let $j$ be as in (2). Fix a $\Pi^1_\xi$ sentence $\varphi$ over $V_\kappa$ with $V_\kappa\models\varphi$. Then it follows by (2) that, in $M$, there is some $\alpha\in j(S)$ such that $V_\alpha\models j(\varphi)\res^{j(\kappa)}_\alpha$. Hence by elementarity, there is an $\alpha\in S$ such that $V_\alpha\models\varphi\res^\kappa_\alpha$.
\end{proof}

\subsection{A hierarchy result}\label{section_hierarchy}

In order to prove the hierarchy results below (Corollary \ref{corollary_hierarchy} and Corollary \ref{corollary_proper}), we first need to establish a connection between universal formulas at $\kappa$ and universal formulas at regular $\alpha<\kappa$.

\begin{lemma}\label{lemma_restricting_universal_formulas}
Suppose $\kappa>\omega$ is regular. Fix any $\xi<\kappa^+$ and $n<\omega$, let $\Psi^\kappa_{\xi,n}(X_1,\ldots,X_n,Y_\xi)$ and $\bar{\Psi}^\kappa_{\xi,n}(X_1,\ldots,X_n,Y_\xi)$ be, respectively, universal $\Pi^1_\xi$ and $\Sigma^1_\xi$ formulas at $\kappa$ for formulas with $n$ free second-order variables, which were defined by transfinite recursion in the proof of Theorem \ref{theorem_universal}. There are clubs $C_{\xi,n}$ and $D_{\xi,n}$ in $\kappa$ such that the following hold.
\begin{enumerate}
\item For all regular $\alpha\in C_{\xi,n}$ the formula $\Psi^\kappa_{\xi,n}(X_1,\ldots,X_n,Y_\xi)\res^\kappa_\alpha$ is a universal $\Pi^1_{f^\kappa_\xi(\alpha)}$ formula at $\alpha$ for formulas with $n$ free second-order variables.
\item For all regular $\alpha\in D_{\xi,n}$ the formula $\bar{\Psi}^\kappa_{\xi,n}(X_1,\ldots,X_n,Y_\xi)\res^\kappa_\alpha$ is a universal $\Sigma^1_{f^\kappa_\xi(\alpha)}$ formula at $\alpha$ for formulas with $n$ free second-order variables.
\end{enumerate}
\end{lemma}

\begin{proof}
We proceed by induction on $\xi$. Suppose $\xi<\kappa^+$ is a limit ordinal. The case in which $\xi$ is a successor ordinal is easier and is left to the reader. We will now prove (1); the proof of (2) is similar. Recall that 
\[\Psi^\kappa_{\xi,1}(X,Y_\xi)=\bigwedge_{\zeta<\xi}\Theta^\kappa_{\zeta,1}(X,Y_\xi)\]
where each $\Theta^\kappa_{\zeta,1}(X,Y_\xi)$ is a $\Pi^1_\zeta$ formula over $V_\kappa$ equivalent to $\Psi^\kappa_{\zeta,1}(X,(Y_\xi)_{b_{\kappa,\xi}^{-1}(\zeta)})$ in the sense that there is a club $D_\zeta$ in $\kappa$ such that for all regular $\alpha\in D_\zeta\cup\{\kappa\}$ we have $V_\alpha\models\Theta^\kappa_{\zeta,1}(A,B)\res^\kappa_\alpha$ if and only if $V_\alpha\models \Psi^\kappa_{\zeta,1}(A,(B)_{b_{\kappa,\xi}^{-1}(\zeta)})\res^\kappa_\alpha$ for all $A\subseteq V_\kappa$ and $B\subseteq\kappa$. To prove (1), we will use Proposition \ref{proposition_framework2}. Suppose $G\subseteq P(\kappa)/\NS_\kappa$ is generic such that $\kappa$ is regular in $V^\kappa/G$ and let $j:V\to V^\kappa/G$ be the corresponding generic ultrapower embedding. To show that the desired club $C_{\xi,1}$ exists, we must show that $\kappa\in j(T)$ where $T$ is the set of regular cardinals $\alpha<\kappa$ such that $\Psi^\kappa_{\xi,1}(X,Y_\xi)\res^\kappa_\alpha$ is a universal $\Pi^1_{f^\kappa_\xi(\alpha)}$ formula at $\alpha$ for formulas with $1$ free variable. By Lemma \ref{lemma_j_of} we have $j(\Psi^\kappa_{\xi,1}(X,Y_\xi))\res^{j(\kappa)}_\kappa=\Psi^\kappa_{\xi,1}(X,Y_\xi)$ and $\kappa\in j(D_\zeta)$ for all $\zeta<\xi$. Thus, working in $V^\kappa/G$, if $A\subseteq V_\kappa$ and $B\subseteq\kappa$ then
\begin{align*}
V_\kappa\models\Psi^\kappa_{\xi,1}(A,B)&\iff V_\kappa\models \bigwedge_{\zeta<\xi}j(\Psi^\kappa_{\zeta,1}(A,(B)_{b_{\kappa,\xi}^{-1}(\zeta)}))\res^{j(\kappa)}_\kappa\\
  &\iff V_\kappa\models \bigwedge_{\zeta<\xi}\Psi^\kappa_{\zeta,1}(A,(B)_{b_{\kappa,\xi}^{-1}(\zeta)}).\\
\end{align*}
Now it is straightforward to verify $\kappa\in j(T)$, that is, $\Psi^\kappa_{\xi,1}(X,Y_\xi)$ is a universal $\Pi^1_\xi$ formula at $\kappa$ for formulas with $1$ free variable in $V^\kappa/G$; we give a brief outline of how to do this here. Still working in $V^\kappa/G$, fix a $\Pi^1_\xi$ formula $\varphi_\xi(X)=\bigwedge_{\zeta<\xi}\varphi_\zeta(X)$. From our inductive assumption, working in $V^\kappa/G$, we may fix codes $K_{\varphi_\zeta}\subseteq\kappa$ such that $V_\kappa\models\varphi_\zeta(A)$ if and only if $V_\kappa\models\Psi^\kappa_{\zeta,1}(A,K_{\varphi_\zeta})$. Then we let $K_\varphi=\{\Gamma(b_{\kappa,\xi}^{-1}(\zeta),\beta)\st\zeta<\xi\land\beta\in K_{\varphi_\zeta}\}$ and proceed exactly as in the proof of Theorem \ref{theorem_universal}, except that here we work in $V^\kappa/G$. Thus we conclude $\kappa\in j(T)$.
\end{proof}

Next we show that for $\xi<\kappa^+$, the $\Pi^1_\xi$-indescribability of a set $S\subseteq\kappa$, is expressible by a $\Pi^1_{\xi+1}$ formula over $V_\kappa$ in the following sense.

\begin{theorem}\label{theorem_expressing_indescribability}
Suppose $\kappa>\omega$ is inaccessible and $\xi<\kappa^+$. There is a $\Pi^1_{\xi+1}$ formula $\Phi^\kappa_\xi(Z)$ over $V_\kappa$ and a club $C\subseteq\kappa$ such that for all $S\subseteq\kappa$ we have
\[\text{$S$ is a $\Pi^1_\xi$-indescribable subset of $\kappa$ if and only if $V_\kappa\models\Phi^\kappa_\xi(S)$}\]
and for all regular $\alpha\in C$ we have
\[\text{$S\cap\alpha$ is a $\Pi^1_{f^\kappa_\xi(\alpha)}$-indescribable subset of $\alpha$ if and only if $V_\alpha\models\Phi^\kappa_\xi(S)\res^\kappa_\alpha$}.\]
\end{theorem}
\begin{proof}
We let $R\subseteq\kappa$ be a set, defined as follows, coding information about which $\alpha<\kappa$ and which $a\subseteq\alpha$ satisfy $V_\alpha\models\Psi^\kappa_{\xi,0}(a)\res^\kappa_\alpha$. For each regular $\alpha<\kappa$ let $\<a^\alpha_\beta\st\beta<\delta_\alpha\>$ be a sequence of subsets of $\alpha$ such that for all $a\subseteq\alpha$ we have $V_\alpha\models\Psi^\kappa_{\xi,0}(a)\res^\kappa_\alpha$ if and only if $a=a^\alpha_\beta$ for some $\beta<\delta_\alpha$. Let $\Gamma:\kappa\times\kappa\times\kappa\to\kappa$ be the usual definable bijection. We let 
\[R=\{\Gamma(\alpha,\beta,\gamma)\st \text{($\alpha$ is regular)}\land \beta<\delta_\alpha\land \gamma\in a^\alpha_\beta\}.\] 
For $\alpha,\beta<\kappa$ we define
\[R_{(\alpha,\beta)}=\{\gamma\st\Gamma(\alpha,\beta,\gamma)\in R\}\]
to be the $(\alpha,\beta)^{th}$ slice of $R$ so that when $\alpha$ is regular and $\beta<\delta_\alpha$ we have $R_{(\alpha,\beta)}=a^\alpha_\beta$.
Now we let
\[\Phi^\kappa_\xi(Z)=\forall X[\Psi^\kappa_{\xi,0}(X)\rightarrow(\exists Y\subseteq Z\cap\REG)(Y\in\NS_\kappa^+)\land(\forall\eta\in Y)(\exists\beta)(X\cap\eta=R_{(\eta,\beta)})].\]
Since the part of $\Phi^\kappa_\xi$ to the right of the $\rightarrow$ is $\Sigma^1_2$ over $V_\kappa$, and since $\Psi^\kappa_{\xi,0}(X)$ is $\Pi^1_\xi$ over $V_\kappa$ and appears to the left of the $\rightarrow$ in $\Phi^\kappa_\xi$, it follows that $\Phi^\kappa_\xi$ is expressible by a $\Pi^1_{\xi+1}$ formula over $V_\kappa$. In what follows, we will identify $\Phi^\kappa_\xi$ with this $\Pi^1_{\xi+1}$ formula.

First let us show that $S\subseteq\kappa$ is $\Pi^1_\xi$-indescribable in $\kappa$ if and only if $V_\kappa\models\Phi^\kappa_\xi(S)$. Suppose $S$ is $\Pi^1_\xi$-indescribable in $\kappa$. To see that $V_\kappa\models\Phi^\kappa_\xi(S)$, fix $K\subseteq\kappa$ such that $V_\kappa\models\Psi^\kappa_{\xi,0}(K)$. Then $D_0=\{\alpha<\kappa\st V_\alpha\models\Psi^\kappa_{\xi,0}(K)\res^\kappa_\alpha\}$ is in the filter $\Pi^1_\xi(\kappa)^*$ and thus $Y=S\cap D_0\cap\REG$ is, in particular, stationary in $\kappa$. If $\alpha\in Y$ then we have $V_\alpha\models\Psi^\kappa_{\xi,0}(K)\res^\kappa_\alpha$ and hence $V_\alpha\models\Psi^\kappa_{\xi,0}(K\cap\alpha)\res^\kappa_\alpha$, which implies that $K\cap\alpha=R_{(\alpha,\beta)}$ for some $\beta<\delta_\alpha$. Conversely, suppose $V_\kappa\models \Phi^\kappa_\xi(S)$ and let us show that $S$ is $\Pi^1_\xi$-indescribable in $\kappa$. Fix a $\Pi^1_\xi$ sentence $\varphi$ such that $V_\kappa\models\varphi$. Then, by Theorem \ref{theorem_universal}, $V_\kappa\models\Psi^\kappa_{\xi,0}(K_\varphi)$ and thus there is a $Y\subseteq S\cap\REG$ stationary in $\kappa$ such that for all $\alpha\in Y$ we have $V_\alpha\models\Psi^\kappa_{\xi,0}(K_\varphi)\res^\kappa_\alpha$. By Theorem \ref{theorem_universal} there is a club $D_\varphi\subseteq\kappa$ such that for all regular $\alpha\in D_\varphi$ we have $V_\alpha\models\varphi\res^\kappa_\alpha$ if and only if $V_\alpha\models\Psi_{\xi,0}(K_\varphi)\res^\kappa_\alpha$. Thus we may choose a regular $\alpha\in Y\cap D_\varphi\cap\REG\subseteq S$ and observe that $V_\alpha\models\varphi\res^\kappa_\alpha$. Hence $S$ is $\Pi^1_\xi$-indescribable in $\kappa$.

To prove the second part of the statement we will use Proposition \ref{proposition_framework2}. Fix $S\subseteq\kappa$. Suppose $G\subseteq P(\kappa)/\NS_\kappa$ is generic, $\kappa$ is regular in $V^\kappa/G$ and $j:V\to V^\kappa/G$ is the corresponding generic ultrapower. Let $E$ be the set of ordinals $\alpha<\kappa$ such that $S\cap\alpha$ is a $\Pi^1_{f^\kappa_\xi(\alpha)}$-indescribable subset of $\alpha$ if and only if $V_\alpha\models\Phi^\kappa_{\xi,0}(S)\res^\kappa_\alpha$. We must show that $\kappa\in j(E)$. By Lemma \ref{lemma_j_of} we have $j(\Phi^\kappa_{\xi,0}(S))\res^{j(\kappa)}_\kappa=\Phi^\kappa_{\xi,0}(S)$, and thus we must show that in $V^\kappa/G$, $S$ is $\Pi^1_\xi$-indescribable in $\kappa$ if and only if $V_\kappa\models \Phi^\kappa_{\xi,0}(S)$. By Lemma \ref{lemma_restricting_universal_formulas}, it follows that in $V^\kappa/G$, $\Psi^\kappa_{\xi,0}(X)$ is a universal $\Pi^1_\xi$ formula at $\kappa$ and therefore we can proceed to verify $\kappa\in j(E)$ by using the argument in the previous paragraph, but working in $V^\kappa/G$.
\end{proof}

We obtain our first hierarchy result as an easy corollary of Theorem \ref{theorem_expressing_indescribability}.

\begin{corollary}\label{corollary_hierarchy}
Suppose $S\subseteq\kappa$ is $\Pi^1_\xi$-indescribable in $\kappa$ where $\xi<\kappa^+$ and let $\zeta<\xi$. Then the set
\[C=\{\alpha<\kappa\st\text{$S\cap\alpha$ is $\Pi^1_{f^\kappa_\zeta(\alpha)}$-indescribable}\}\]
is in the filter $\Pi^1_\xi(\kappa)^*$.
\end{corollary}

\begin{proof}
Since $\zeta<\xi$, it follows that $S$ is $\Pi^1_\zeta$-indescribable in $\kappa$, and thus $V_\kappa\models\Phi^\kappa_\zeta(S)$, where $\Phi^\kappa_\zeta(Z)$ is the $\Pi^1_{\zeta+1}$ formula over $V_\kappa$ obtained from Theorem \ref{theorem_expressing_indescribability}. By Theorem \ref{theorem_expressing_indescribability}, there is a club $D$ in $\kappa$ such that for every regular $\alpha\in D$,
\[\text{$S\cap\alpha$ is $\Pi^1_{f^\kappa_\zeta(\alpha)}$-indescribable if and only if $V_\alpha\models\Phi^\kappa_\zeta(S)\res^\kappa_\alpha$}.\] Since the set
\[D\cap\{\alpha<\kappa\st V_\alpha\models\Phi^\kappa_\zeta(S)\res^\kappa_\alpha\}\]
is in the filter $\Pi^1_\zeta(\kappa)^*$, we see that
\[\{\alpha<\kappa\st\text{$S\cap\alpha$ is $\Pi^1_{f^\kappa_\zeta(\alpha)}$-indescribable}\}\in\Pi^1_\zeta(\kappa)^*\subseteq \Pi^1_\xi(\kappa)^*.\]
\end{proof}

Next, in order to show that when $\kappa$ is $\Pi^1_\xi$-indescribable, we have a proper containment $\Pi^1_\zeta(\kappa)\subsetneq\Pi^1_\xi(\kappa)$ for all $\zeta<\xi$ (see Corollary \ref{corollary_proper}), we need several preliminary results.

Before we show that the restriction of a restriction of a given $\Pi^1_\xi$ formula $\varphi$, is often equal to a single restriction of $\varphi$, we need a lemma, which is established using an argument similar to that of Lemma \ref{lemma_j_of}.

\begin{lemma}[{Cody-Holy \cite{cody_holy_2022}}]\label{lemma_double_restriction}
Suppose $I$ is a normal ideal on $\kappa$ and $G\subseteq P(\kappa)/I$ is generic such that $\kappa$ is regular in $V^\kappa/G$ and let $j:V\to V^\kappa/G$ be the corresponding generic ultrapower. If $\varphi$ is either a $\Pi^1_\xi$ or $\Sigma^1_\xi$ formula over $V_\kappa$ for some $\xi<\kappa^+$, and $\alpha<\kappa$ is regular such that $\varphi\res^\kappa_\alpha$ is defined, then \[j(\varphi)\res^{j(\kappa)}_\alpha=\varphi\res^\kappa_\alpha,\] with the former being calculated in $V^\kappa/G$, and the latter being calculated in $V$.
\end{lemma}
\begin{proof}
  By induction on $\xi<\kappa^+$. This is immediate in case $\xi<\kappa$, for then by Remark \ref{remark_coding}(1), $j(\varphi(A_1,\ldots,A_n))=\varphi(j(A_1),\ldots,j(A_n))$, and thus $j(\varphi)\res^{j(\kappa)}_\alpha=\varphi\res^\kappa_\alpha$ by the definition of the restriction operation in this case. It is also immediate for successor steps above $\kappa$, for then by Remark \ref{remark_coding}(2), $j(\forall\vec X\psi)=\forall\vec X j(\psi)$.
  
    At limit steps $\xi\ge\kappa$, if $\varphi=\bigwedge_{\zeta<\xi}\psi_\zeta$ is a $\Pi^1_\xi$ formula, let $\vec\psi=\langle\psi_\zeta\mid\zeta<\xi\rangle$, and let $\vec\pi=\langle\pi^\kappa_{\xi,\alpha}\mid\alpha<\kappa\rangle$. Then, by Remark \ref{remark_coding}(3), $j(\varphi)=\bigwedge_{\zeta<j(\xi)}j(\vec\psi)_\zeta$, and therefore, assuming for now that $j(\varphi)\res^{j(\kappa)}_\alpha$ is defined, \[j(\varphi)\res^{j(\kappa)}_\alpha=\bigwedge_{\zeta\in j(f^\kappa_\xi)(\alpha)}j(\vec\psi)_{j(\vec\pi)_\alpha^{-1}(\zeta)}\res^{j(\kappa)}_\alpha=\bigwedge_{\zeta\in j(f^\kappa_\xi)(\alpha)}j(\psi_{j^{-1}(j(\vec\pi)_\alpha^{-1}(\zeta))})\res^{j(\kappa)}_\alpha,\]
using that $j(\vec\pi)_\alpha^{-1}[j(f^\kappa_\xi)(\alpha)]=j(F^\kappa_\xi)(\alpha)\subseteq j(F^\kappa_\xi)(\kappa)=j"\xi$. 
By our inductive hypothesis, for each $\gamma\in\xi$ and every regular $\alpha<\kappa$, $j(\psi_\gamma)\res^{j(\kappa)}_\alpha=\psi_\gamma\res^\kappa_\alpha$. Thus,
\[j(\varphi)\res^{j(\kappa)}_\alpha=\bigwedge_{\zeta\in j(f^\kappa_\xi)(\alpha)}\psi_{j^{-1}(j(\vec\pi)_\alpha^{-1}(\zeta))}\res^\kappa_\alpha.\]
Now, \[\varphi\res^\kappa_\alpha=\bigwedge_{\zeta\in f^\kappa_\xi(\alpha)}\psi_{(\pi^\kappa_{\xi,\alpha})^{-1}(\zeta)}\res^\kappa_\alpha.\] 
Since $\alpha<\kappa$ we have $j(f^\kappa_\xi)(\alpha)=f^\kappa_\xi(\alpha)$, and furthermore \[(\pi^\kappa_{\xi,\alpha})^{-1}[f^\kappa_\xi(\alpha)]=F^\kappa_\xi(\alpha)=(j^{-1}\circ j(\vec\pi)_\alpha^{-1})[j(f^\kappa_\xi)(\alpha)],\]
showing the above restrictions of $\varphi$ and of $j(\varphi)$ to be equal,\footnote{Being somewhat more careful here, this in fact also uses that the maps $\pi^\kappa_{\xi,\alpha}$, $j$, and $j(\vec\pi)_\alpha$ are order-preserving, so that both of the above conjunctions are taken of the same formulas \emph{in the same order}.} and thus in particular also showing that $j(\varphi)\res^{j(\kappa)}_\alpha$ is defined, as desired.

 The case when $\varphi$ is a $\Sigma^1_\xi$ formula is treated in exactly the same way.
\end{proof}

We can now easily deduce the following, which was originally established in an earlier version of this article using a different proof. The proof included below is due to the author and Peter Holy.

\begin{proposition}\label{proposition_double_restriction}
Suppose $\kappa$ is weakly Mahlo, and $\xi<\kappa^+$. For any formula $\varphi$ which is either $\Pi^1_\xi$ or $\Sigma^1_\xi$ over $V_\kappa$, there is a club $D\subseteq\kappa$ such that for all regular uncountable $\alpha\in D$, $\varphi\res^\kappa_\alpha$ is defined, and the set $D_\alpha$ of all ordinals $\beta<\alpha$ such that $(\varphi\res^\kappa_\alpha)\res^\alpha_\beta$ is defined and $(\varphi\res^\kappa_\alpha)\res^\alpha_\beta=\varphi\res^\kappa_\beta$, is in the club filter on $\alpha$. 
\end{proposition}
\begin{proof}
  Assume for a contradiction that the conclusion of the proposition fails. By Lemma \ref{lemma_restriction_is_nice}, this means that there is a stationary set $T$ consisting of regular and uncountable cardinals $\alpha$ such that the set $D_\alpha$ has stationary complement $E_\alpha\subseteq\alpha$. Using Lemma \ref{lemma_restriction_is_nice} once again, we may assume that $(\varphi\res^\kappa_\alpha)\res^\alpha_\beta$ is defined for every $\alpha\in T$ and every $\beta\in E_\alpha$. Let $\vec E$ denote the sequence $\langle E_\alpha\mid\alpha\in T\rangle$. Assume that $G\subseteq P(\kappa)/\NS_\kappa$ is generic over $V$ with $T\in G$ and $j:V\to V^\kappa/G$ is the corresponding generic ultrapower. Then, $\kappa\in j(T)$, and thus $j(\vec E)_\kappa$ is stationary in $V^\kappa/G$. But, \[j(\vec E)_\kappa=\{\beta<\kappa\mid(j(\varphi)\res^{j(\kappa)}_\kappa)\res^\kappa_\beta\ne j(\varphi)\res^{j(\kappa)}_\beta\}.\]
  Note that by Lemma \ref{lemma_double_restriction}, $j(\varphi)\res^{j(\kappa)}_\kappa=\varphi$. But then, by Lemma \ref{lemma_restriction_is_nice} and Lemma \ref{lemma_double_restriction}, $j(\vec E)_\kappa$ is nonstationary in $V^\kappa/G$, which gives our desired contradiction.
\end{proof}

Recall that for an uncountable regular cardinal $\kappa$, if $S\subseteq\kappa$ is stationary in $\kappa$ and for each $\alpha\in S$ we have a set $S_\alpha\subseteq\alpha$ which is stationary in $\alpha$, then it follows that $\bigcup_{\alpha\in S}S_\alpha$ is stationary in $\kappa$. We generalize this to $\Pi^1_\xi$-indescribability for all $\xi<\kappa^+$ as follows (this result was previously known \cite[Lemma 3.1]{MR4206111} for $\xi<\kappa$).

\begin{lemma}\label{lemma_union}
Suppose $S$ is a $\Pi^1_\xi$-indescribable subset of $\kappa$ where $\xi<\kappa^+$. Further suppose that $S_\alpha$ is a $\Pi^1_{f^\kappa_\xi(\alpha)}$-indescribable subset of $\alpha$ for each $\alpha\in S$. Then $\bigcup_{\alpha\in S}S_\alpha$ is a $\Pi^1_\xi$-indescribable subset of $\kappa$.
\end{lemma}

\begin{proof} Suppose $\xi<\kappa^+$ and $\varphi$ is some $\Pi^1_\xi$ sentence over $V_\kappa$ such that $V_\kappa\models\varphi$. By Lemma \ref{lemma_restriction_is_nice},
\[C_\varphi=\{\alpha<\kappa\st\text{$\varphi\res^\kappa_\alpha$ is $\Pi^1_{f^\kappa_\xi(\alpha)}$ over $V_\kappa$}\}\]
is in the club filter on $\kappa$. By Proposition \ref{proposition_double_restriction}, there is a club $D_\varphi\subseteq\kappa$ such that for all regular $\alpha\in D_\varphi$ the set of $\beta<\alpha$ such that $(\varphi\res^\kappa_\alpha)\res^\alpha_\beta=\varphi\res^\kappa_\beta$ is in the club filter on $\alpha$. Thus, $S\cap C_\varphi\cap D_\varphi$ is $\Pi^1_\xi$-indescribable in $\kappa$. Hence there is a regular uncountable $\alpha\in S\cap C_\varphi\cap D_\varphi$ such that $V_\alpha\models\varphi\res^\kappa_\alpha$. Let $E$ be a club subset of $\alpha$ such that for all $\beta\in E$ we have $(\varphi\res^\kappa_\alpha)\res^\alpha_\beta=\varphi\res^\kappa_\beta$. Since $S_\alpha\cap E$ is $\Pi^1_{f^\kappa_\xi(\alpha)}$-indescribable in $\alpha$ and $\varphi\res^\kappa_\alpha$ is $\Pi^1_{f^\kappa_\xi(\alpha)}$ over $V_\alpha$, there is some $\beta\in S_\alpha\cap E$ such that $V_\beta\models(\varphi\res^\kappa_\alpha)\res^\alpha_\beta$. Since $(\varphi\res^\kappa_\alpha)\res^\alpha_\beta=\varphi\res^\kappa_\beta$, it follows that $\bigcup_{\alpha\in S}S_\alpha$ is $\Pi^1_\xi$-indescribable in $\kappa$.
\end{proof}

\begin{lemma}\label{lemma_set_of_non}
For all ordinals $\xi$, if $S\subseteq\kappa$ is $\Pi^1_\xi$-indescribable in $\kappa$ where $\xi<\kappa^+$, then the set
\[T=\{\alpha<\kappa\st\text{$S\cap\alpha$ is not $\Pi^1_{f^\kappa_\xi(\alpha)}$-indescribable in $\alpha$}\}\]
is $\Pi^1_\xi$-indescribable in $\kappa$.
\end{lemma}

\begin{proof}
We proceed by induction on $\xi$. For $\xi<\omega$ this is a well-known result, which follows directly from \cite[Lemma 3.2]{MR4206111}. Suppose $\xi\in\kappa^+\setminus\omega$ and, for the sake of contradiction, suppose $S$ is $\Pi^1_\xi$-indescribable and $T$ is not $\Pi^1_\xi$-indescribable in $\kappa$. Then $\kappa\setminus T$ is in the filter $\Pi^1_\xi(\kappa)^*$ and is thus $\Pi^1_\xi$-indescribable in $\kappa$. By Corollary \ref{corollary_crazy}, there is a club $C\subseteq\kappa$ such that for all regular uncountable $\alpha\in C$, the set 
\[D_\alpha=\{\beta<\alpha\st f^\kappa_\xi(\beta)=f^\alpha_{f^\kappa_\xi(\alpha)}(\beta)\}\] is in the club filter on $\alpha$. Let $D$ be the set of regular uncountable cardinals less than $\kappa$, and note that $D\in \Pi^1_1(\kappa)^*\subseteq\Pi^1_\xi(\kappa)^*$. Notice that $(\kappa\setminus T)\cap C\cap D$ is $\Pi^1_\xi$-indescribable in $\kappa$. For each $\alpha\in (\kappa\setminus T)\cap C\cap D$, it follows by induction that the set
\[T_\alpha=\{\beta<\alpha\st\text{$S\cap\beta$ is not $\Pi^1_{f^\alpha_{f^\kappa_\xi(\alpha)}(\beta)}$-indescribable}\}\]
is $\Pi^1_{f^\kappa_\xi(\alpha)}$-indescribable in $\alpha$. Thus, for each $\alpha\in (\kappa\setminus T)\cap C\cap D$ the set $T_\alpha\cap D_\alpha$ is $\Pi^1_{f^\kappa_\xi(\alpha)}$-indescribable in $\alpha$. Now it follows by Lemma \ref{lemma_union} that the set
\[\bigcup_{\alpha\in (\kappa\setminus T)\cap C\cap D}(T_\alpha\cap D_\alpha)\subseteq T\]
is $\Pi^1_\xi$-indescribable in $\kappa$, a contradiction.
\end{proof}

Now we show that for regular $\kappa$, whenever $\zeta<\xi<\kappa^+$ and the ideals under consideration are nontrivial, we have $\Pi^1_\zeta(\kappa)\subsetneq\Pi^1_\xi(\kappa)$.

\begin{corollary}\label{corollary_proper}
Suppose $\kappa$ is $\Pi^1_\xi$-indescribable where $\xi<\kappa^+$. Then for all $\zeta<\xi$ we have $\Pi^1_\zeta(\kappa)\subsetneq\Pi^1_\xi(\kappa)$.
\end{corollary}

\begin{proof}
The fact that $\Pi^1_\zeta(\kappa)\subseteq\Pi^1_\xi(\kappa)$ follows easily from the fact that the class of $\Pi^1_\xi$ formulas includes the $\Pi^1_\zeta$ formulas. To see that the proper containment holds, consider the set
\[C=\{\alpha<\kappa\st\text{$\alpha$ is $\Pi^1_{f^\kappa_\zeta(\alpha)}$-indescribable}\}.\]
By Corollary \ref{corollary_hierarchy} and Proposition \ref{lemma_set_of_non}, we have $\kappa\setminus C\in \Pi^1_\xi(\kappa)\setminus\Pi^1_\zeta(\kappa)$.
\end{proof}

\subsection{Higher $\Pi^1_\xi$-clubs}\label{section_higher_xi_clubs}

Now we present a characterization of the $\Pi^1_\xi$-indescribability of sets $S\subseteq\kappa$ in terms of a natural base for the filter $\Pi^1_\xi(\kappa)^*$.

\begin{definition}\label{definition_Pi1xi_club}
Suppose $\kappa$ is a regular cardinal. We define the notion of $\Pi^1_\xi$-club subset of $\kappa$ for all $\xi<\kappa^+$ by induction.\begin{enumerate}
\item A set $C\subseteq\kappa$ is \emph{$\Pi^1_0$-club} if it is closed and unbounded in $\kappa$. 
\item We say that $C$ is \emph{$\Pi^1_{\zeta+1}$-club in $\kappa$} if $C$ is $\Pi^1_\zeta$-indescribable in $\kappa$ and $C$ is \emph{$\Pi^1_\zeta$-closed}, in the sense that there is a club $C^*$ in $\kappa$ such that for all $\alpha\in C^*$, whenever $C\cap\alpha$ is $\Pi^1_{f_\zeta^\kappa(\alpha)}$-indescribable in $\alpha$ we must have $\alpha\in C$. 
\item If $\xi$ is a limit, we say that $C\subseteq\kappa$ is \emph{$\Pi^1_\xi$-club in $\kappa$} if $C$ is $\Pi^1_\zeta$-indescribable for all $\zeta<\xi$ and $C$ is \emph{$\Pi^1_\xi$-closed}, in the sense that there is a club $C^*$ in $\kappa$ such that for all $\alpha\in C^*$, whenever $C\cap\alpha$ is $\Pi^1_\zeta$-indescribable for all $\zeta < f^\kappa_\xi(\alpha)$, we must have $\alpha\in C$.
\end{enumerate}
\end{definition}

Let us show that, when the $\Pi^1_\xi$-indescribability ideal $\Pi^1_\xi(\kappa)$ is nontrivial, the $\Pi^1_\xi$-club subsets of $\kappa$ form a filter base for the dual filter $\Pi^1_\xi(\kappa)^*$ and a set being $\Pi^1_\xi$-club in $\kappa$ is expressible by a $\Pi^1_\xi$ sentence.

\begin{theorem}\label{theorem_xi_clubs}
Suppose $\kappa$ is a regular cardinal. For all $\xi<\kappa^+$, if $\kappa$ is $\Pi^1_\xi$-indescribable then the following hold.
\begin{enumerate}
\item A set $S\subseteq\kappa$ is $\Pi^1_\xi$-indescribable if and only if $S\cap C\neq\emptyset$ for all $\Pi^1_\xi$-clubs $C\subseteq\kappa$.
\item There is a $\Pi^1_\xi$ formula $\chi^\kappa_\xi(X)$ over $V_\kappa$ such that for all $C\subseteq \kappa$ we have
\[\text{$C$ is $\Pi^1_\xi$-club in $\kappa$ if and only if } V_\kappa\models \chi^\kappa_\xi(C)\]
and there is a club $D_\xi$ in $\kappa$ such that for all regular $\alpha\in D_\xi$ and all $C\subseteq\kappa$ we have
\[\text{$C\cap\alpha$ is $\Pi^1_{f^\kappa_\xi(\alpha)}$-club in $\alpha$ if and only if } V_\alpha\models\chi^\kappa_\xi(C)\res^\kappa_\alpha.\]
\end{enumerate}
\end{theorem}

\begin{proof}
Sun \cite[Theorem 1.17]{MR1245524} proved that the theorem holds for $\xi=1$, and Hellsten \cite[Theorem 2.4.2]{MR2026390} generalized this to the case in which $\xi<\omega$. We provide a proof of the case in which $\xi<\kappa^+$ is a limit ordinal; the case in which $\xi<\kappa^+$ is a successor is similar, but easier.

Suppose $\xi<\kappa^+$ is a limit ordinal and that both (1) and (2) hold for all ordinals $\zeta<\xi$. For the forward direction of (1), suppose $S\subseteq\kappa$ is $\Pi^1_\xi$-indescribable and fix $C\subseteq\kappa$ a $\Pi^1_\xi$-club subset of $\kappa$. Then, in particular, for each $\zeta<\xi$, $C$ is $\Pi^1_\zeta$-indescribable and, by Theorem \ref{theorem_expressing_indescribability}, we see that
\[V_\kappa\models\bigwedge_{\zeta<\xi}\Phi^\kappa_\zeta(C).\]
Let $G\subseteq P(\kappa)/\Pi^1_\xi(\kappa)$ be generic over $V$ with $S\in G$ and let $j:V\to V^\kappa/G$ be the corresponding generic ultrapower. Then $\kappa\in j(S)$ and by the proof of Proposition \ref{proposition_generic_characterization}, we have $\left(V_\kappa\models\bigwedge_{\zeta<\xi}\Phi^\kappa_\zeta(C)\right)^{V^\kappa/G}.$
For each $\zeta<\xi$, let $C_\zeta$ be the club subset of $\kappa$ obtained from Theorem \ref{theorem_expressing_indescribability} and notice that $\kappa\in j(C_\zeta)$ and hence 
in $V^\kappa/G$ the set $C$ is $\Pi^1_\zeta$-indescribable in $\kappa$. Since $C$ is a $\Pi^1_\xi$-club subset of $\kappa$ there is a club $C^*\subseteq\kappa$ as in Definition \ref{definition_Pi1xi_club}. Since $\kappa\in j(C^*)$ and $j(C)\cap\kappa$ is $\Pi^1_\zeta$-indescribable for all $\zeta<\xi=j(f^\kappa_\xi)(\kappa)$, it follows that $\kappa\in j(C)$. Therefore by elementarity $S\cap C\neq\emptyset$.

For the reverse direction of (1), suppose $\kappa$ is $\Pi^1_\xi$-indescribable and $S\subseteq\kappa$ intersects every $\Pi^1_\xi$-club. It suffices to show that if $\varphi=\bigwedge_{\zeta<\xi}\varphi_\zeta$ is any $\Pi^1_\xi$ sentence over $V_\kappa$ such that $V_\kappa\models\varphi$, then the set
\[C=\{\alpha\in D\st V_\alpha\models\varphi\res^\kappa_\alpha\}\]
contains a $\Pi^1_\xi$-club, where $D\subseteq\kappa$ is a club subset of $\kappa$ such that for all regular $\alpha\in D$, $\varphi\res^\kappa_\alpha$ is defined. 

First, let us argue that $C$ is $\Pi^1_\zeta$-indescribable for all $\zeta<\xi$. Suppose not. Then for some fixed $\zeta<\xi$, $C$ is not $\Pi^1_\zeta$-indescribable in $\kappa$ and hence $\kappa\setminus C$ is in the filter $\Pi^1_\zeta(\kappa)^*$. Since $\kappa$ is $\Pi^1_\xi$-indescribable by assumption, and since $\Pi^1_\zeta(\kappa)^*\subseteq\Pi^1_\xi(\kappa)^*\subseteq\Pi^1_\xi(\kappa)^+$, we see that $\kappa\setminus C$ is $\Pi^1_\xi$-indescribable in $\kappa$. Since $(\kappa\setminus C)\cap D$ is $\Pi^1_\xi$-indescribable and $V_\kappa\models\varphi$ there is an $\alpha\in (\kappa\setminus C)\cap D$ such that $V_\alpha\models\varphi\res^\kappa_\alpha$, a contradiction.

Next we must argue that $C$ is $\Pi^1_\xi$-closed. We must show that there is a club $C^*$ in $\kappa$ such that for all regular $\alpha\in C^*$, if $C\cap\alpha$ is $\Pi^1_\zeta$-indescribable in $\alpha$ for all $\zeta<f^\kappa_\xi(\alpha)$ then $\alpha\in C$. We will use Proposition \ref{proposition_framework2}. Let $G\subseteq P(\kappa)/\NS_\kappa$ be generic with $\kappa$ regular in $V^\kappa/G$ and let $j:V\to V^\kappa/G$ be the corresponding generic ultrapower embedding. It suffices to show that in $V^\kappa/G$, if $C$ is $\Pi^1_\zeta$-indescribable in $\kappa$ for all $\zeta<\xi$ then $\alpha\in j(C)$. Assume that in $V^\kappa/G$, $C$ is $\Pi^1_\zeta$-indescribable in $\kappa$ for all $\zeta<\xi$ but $\kappa\notin j(C)$. Since $j(\varphi)\res^{j(\kappa)}_\kappa=\varphi$, it follows from the definition of $C$ that for some $\zeta<\xi$, $(V_\kappa\models\lnot\varphi_\zeta)^{V^\kappa/G}$. But in $V^\kappa/G$, $C$ is $\Pi^1_\zeta$-indescribable in $\kappa$ and so there is some $\alpha\in C$ such that $(V_\alpha\models\lnot\varphi_\zeta\res^\kappa_\alpha)^{V^\kappa/G}$, which contradicts the definition of $C$.

Now, let us show that (2) holds for the limit ordinal $\xi$. The definition of ``$X$ is $\Pi^1_\xi$-club'' is equivalent over $V_\kappa$ to
\[\left(\bigwedge_{\eta<\xi}\Phi^\kappa_\eta(X)\right)\land(\exists C^*)\left[(\text{$C^*$ is club})\land (\forall\beta\in C^*)\left(\bigwedge_{\zeta< f^\kappa_\xi(\alpha)}(X\cap\beta\in\Pi^1_\zeta(\beta)^+)\rightarrow\beta\in X\right)\right].\]
We define a set $R_\xi\subseteq \kappa$ that codes all relevant information about which subsets of $\alpha$, for $\alpha<\kappa$, are $\Pi^1_\zeta$-indescribable for all $\zeta<f^\kappa_\xi(\alpha)$ as follows. We let $R_\xi\subseteq\kappa$ be such that for each regular $\alpha<\kappa$, if $\alpha$ is $\Pi^1_\zeta$-indescribable for all $\zeta<f^\kappa_\xi(\alpha)$, then the sequence
\[\<(R_\xi)_\eta\st\alpha\leq\eta<2^\alpha\>\]
is an enumeration of the subsets of $\alpha$ that are $\Pi^1_\zeta$-indescribable in $\alpha$ for all $\zeta<f^\kappa_\xi(\alpha)$. Otherwise, we define $(R_\xi)_\eta=\emptyset$. Now we let 
\[\bar\chi^\kappa_\xi(X)=(\exists C^*)\left[(\text{$C^*$ is club})\land(\forall\beta\in C^*)(\exists\eta (X\cap\beta=(R_\xi)_\eta)\rightarrow \beta\in X)\right]\]
and
\[\chi^\kappa_\xi(X)=\left(\bigwedge_{\eta<\xi}\Phi^\kappa_\eta(X)\right)\land\bar\chi^\kappa_\xi(X).\]
Since the second part $\bar\chi^\kappa_\xi(X)$ of the definition of $\chi^\kappa_\xi(X)$ is $\Sigma^1_1$, it is also trivially $\Pi^1_2$, and thus we see that $\chi^\kappa_\xi(X)$ is $\Pi^1_\xi$ over $V_\kappa$. Clearly, for all $C\subseteq\kappa$ we have
\[\text{$C$ is $\Pi^1_\xi$-club in $\kappa$}\iff V_\kappa\models\chi_\xi(C).\]
To complete the proof of (2), one may use Proposition \ref{proposition_framework2}, along with Theorem \ref{theorem_expressing_indescribability} to show that there is a club $D_\xi$ in $\kappa$ such that for all regular $\alpha\in D_\xi$ we have that for all $C\subseteq\kappa$,
\[\text{$C\cap\alpha$ is $\Pi^1_{f^\kappa_\xi(\alpha)}$-club in $\alpha$}\iff V_\alpha\models\chi_\xi(C)\res^\kappa_\alpha.\]
Let us note that the remaining details are similar to the proof of Theorem \ref{theorem_expressing_indescribability}(2), and are therefore left to the reader.
\end{proof}

\section{Higher $\xi$-stationarity, $\xi$-s-stationarity and derived topologies}\label{section_higher_derived_topologies}

In this section we define natural generalizations of Bagaria's notions of $\xi$-stationarity, $\xi$-s-stationarity and derived topologies. Given a regular cardinal $\mu$, we will define a sequence of topologies $\<\tau_\xi\st\xi<\mu^+\>$ on $\mu$ such that the sequence $\<\tau_\xi\st\xi<\mu\>$ is Bagaria's sequence of derived topologies, and Bagaria's characterization of nonisolated points in the spaces $(\mu,\tau_\xi)$ for $\xi<\mu$ has a natural generalization to $\tau_\xi$ for $\xi\in\mu^+\setminus\mu$ (see Theorem \ref{theorem_xi_s_nonisolated}). We also show that Bagaria's result, in which he obtains the nondiscreteness of the topologies $\tau_\xi$ for $\xi<\mu$ from an indescribability hypothesis, can be generalized to $\tau_\xi$ for all $\xi<\mu^+$ using higher indescribability (see Corollary \ref{corollary_nondiscreteness_from_indescribability}).

Let us now discuss a generalization of Bagaria's derived topologies. Recall that, under certain conditions, one can specify a topology on a set $X$ by stating what the limit point operation must be. If $d:P(X)\to P(X)$ is a function satisfying properties (1) and (2) in Definition \ref{definition_cantor_derivative}, then one can define a topology $\tau_d$ on $X$ by demanding that a set $C\subseteq X$ be closed if and only if $d(C)\subseteq C$. Furthermore, if $d$ also satisfies property (3) in Definition \ref{definition_cantor_derivative}, then $d$ equals the limit point operator in the space $(X,\tau_d)$.

\begin{definition}\label{definition_cantor_derivative}
Given a set $X$, we say that a function $d:P(X)\to P(X)$ is a \emph{Cantor derivative} on $X$ provided that the following conditions hold.
\begin{enumerate}
\item $d(\emptyset)=\emptyset$.
\item For all $A,B\in P(X)$ we have
\begin{enumerate}
\item $A\subseteq B$ implies $d(A)\subseteq d(B)$,
\item $d(A\cup B)\subseteq d(A)\cup d(B)$ and
\item for all $x\in X$, $x\in d(A)$ implies $x\in d(A\setminus\{x\})$.
\end{enumerate}
\item $d(d(A))\subseteq d(A)\cup A$.
\end{enumerate}
If $d$ satisfies only (1) and (2) then we say that $d$ is a \emph{pre-Cantor derivative}.
\end{definition}

\begin{fact}\label{fact_cantor_derivatives}
If $d:P(X)\to P(X)$ is a pre-Cantor derivative on $X$ then the collection
\[\tau=\{U\subseteq X\st d(X\setminus U)\subseteq X\setminus U\}\]
is a topology on $X$. Furthermore, if $d$ is a Cantor derivative on $X$ then
\[d(A)=\{x\in X\st\text{$x$ is a limit point of $A$ in $(X,\tau)$}\}\]
for all $A\subseteq X$.
\end{fact}

\begin{proof}
Clearly $\emptyset\in\tau$ since $d(X)\subseteq X$. Furthermore, $X\in\tau$ because $d(\emptyset)=\emptyset$ by assumption. Suppose $I$ is some index set and for each $i\in I$ we have a set $C_i\subseteq X$ with $d(C_i)\subseteq C_i$. By (2a), it follows that for every $j\in I$ we have $d\left(\bigcap_{i\in I}C_i\right)\subseteq d(C_j)$ and hence $d\left(\bigcap_{i\in I}C_i\right)\subseteq \bigcap_{i\in I}C_i.$ Furthermore, if $I=\{0,1\}$ we have $d(C_0\cup C_1)\subseteq d(C_0)\cup d(C_1)\subseteq C_0\cup C_1$. Thus, $\tau$ is a topology on $X$.

For $A\subseteq X$, let $A'$ denote the set of limit points of $A$ in $(X,\tau)$. Let us show that $d(A)=A'$. Suppose $x\in d(A)$ and fix $U\in \tau$ with $x\in U$. For the sake of contradiction, suppose that $(U\cap A)\setminus\{x\}=\emptyset$, and notice that $A\setminus\{x\}\subseteq A\setminus U\subseteq X\setminus U$. Thus $d(A\setminus \{x\})\subseteq d(X\setminus U)\subseteq X\setminus U$, and since $x\in U$, this implies $x\notin d(A\setminus \{x\})$. But this implies $x\notin d(A)$ by (2c), a contradiction. Thus for any $A\subseteq X$ we have $d(A)\subseteq A'$.

For any set $A\subseteq X$, since the closure $\overline{A}=A\cup A'$ is the smallest closed set containing $A$, since $A\cup d(A)$ is closed (by (2b) and (3)) and since $d(A)\subseteq A'$, it follows that $\overline{A}=A\cup A'=A\cup d(A)$.

Now fix $A\subseteq X$. We have 
\begin{align*}
x\in A' &\iff x\in \overline{A\setminus\{x\}}\\ 
	&\iff x\in (A\setminus \{x\})\cup d(A\setminus\{x\})\\
	&\iff x\in d(A\setminus\{x\})\\
	&\iff x\in d(A).
\end{align*}



\end{proof}

Given an ordinal $\delta$, Bagaria defined the sequence of derived topologies $\<\tau_\xi\st\xi<\delta\>$ on $\delta$ as follows. 

\begin{definition}[Bagaria \cite{MR3894041}]\label{definition_bagaria}
Let $\tau_0$ be the interval topology on $\delta$. That is, $\tau_0$ is the topology on $\delta$ generated\footnote{Recall that, given a set $X$ and a collection $\B\subseteq P(X)$, the \emph{topology generated by $\B$} is the smallest topology on $X$ which contains $\B$. That is, the topology generated by $\B$ is the collection of all unions of finite intersections of members of $\B$ together with the set $X$.} by the collection $\B_0$ consisting of $\{0\}$ and all open intervals of the form $(\alpha,\beta)$ where $\alpha<\beta\leq\delta$. We let $d_0:P(\delta)\to P(\delta)$ be the limit point operator of the space $(\delta,\tau_0)$. If $\xi<\delta$ is an ordinal and the sequences $\<B_\zeta\st\zeta\leq\xi\>$, $\<\tau_\zeta\st\zeta\leq\xi\>$ and $\<d_\zeta\st\zeta\leq\xi\>$ have been defined, we let $\tau_{\xi+1}$ be the topology generated by the collection
\[\B_{\xi+1}=\B_\xi\cup\{d_\xi(A)\st A\subseteq\delta\}\]
and we let
\[d_{\xi+1}(A)=\{\alpha<\delta\st\text{$\alpha$ is a limit point of $A$ in the $\tau_{\xi+1}$ topology}\}.\]
When $\xi<\delta$ is a limit ordinal, we define $\B_\xi=\bigcup_{\zeta<\xi}\B_\zeta$, let $\tau_\xi$ be the topology generated by $\B_\xi$ and define $d_\xi$ to be the limit point operator of the space $(\delta,\tau_\xi)$.
\end{definition}

Bagaria proved that a point $\alpha<\delta$ is not isolated in $(\delta,\tau_\xi)$ if and only if it is $\xi$-s-reflecting (see \cite[Definition 2.8]{MR3894041} or Definition \ref{definition_xi_s_stationarity}). Since no ordinal $\alpha<\delta$ can be $\delta$-s-reflecting (see Remark \ref{remark_nontrivial}), it follows that the topology $\tau_\delta$ generated by $\bigcup_{\zeta<\delta}\B_\zeta$ is discrete. In what follows, by using diagonal Cantor derivatives, we extend Bagaria's definition of derived topologies to allow for more nontrivial cases. One may want to review Remark \ref{remark_example} before reading the following.

\begin{definition}\label{definition_tau_xi}
Suppose $\mu$ is a regular cardinal. We define three sequences of functions $\<\B_\xi\st\xi<\mu^+\>$, $\<\T_\xi\st\xi<\mu^+\>$ and $\<d_\xi\st\xi<\mu^+\>$, and one sequence $\<\tau_\xi\st\xi<\mu^+\>$ of topologies on $\mu$ by transfinite induction as follows. For $\xi<\mu$ we let $\tau_\xi$ and $d_\xi$ be defined as Definition \ref{definition_bagaria}, and we let $\B_\xi$ and $\T_\xi$ be functions with domain $\mu$ such that for all $\alpha<\mu$, we have $\T_\xi(\alpha)=\tau_\xi$ and $\B_\xi(\alpha)=\B_\xi$ is the subbasis for $\tau_\xi$ as in Definition \ref{definition_bagaria}.

Suppose $\xi\in\mu^+\setminus\mu$ and we have already defined $\<\B_\zeta\st\zeta<\xi\>$, $\<\T_\zeta\st\zeta<\xi\>$, $\<d_\zeta\st\zeta<\xi\>$ and $\<\tau_\zeta\st\zeta<\xi\>$. We let $\B_\xi$ and $\T_\xi$ be the functions with domain $\mu$ such that for each $\alpha\in \mu$ we have
\[\B_\xi(\alpha)=\B_0\cup\{d_\zeta(A)\st\zeta\in F^\mu_\xi(\alpha)\land A\subseteq\mu\}\]
and $\T_\xi(\alpha)$ is the topology on $\mu$ generated by $\B_\xi(\alpha)$. We define $d_\xi:P(\mu)\to P(\mu)$ by letting
\[d_\xi(A)=\{\alpha<\mu\st\text{$\alpha$ is a limit point of $A$ in the $\T_\xi(\alpha)$ topology}\}\]
for $A\subseteq\mu$. Then we let $\tau_\xi$ be the topology\footnote{It is easily seen that this $d_\xi$ is a pre-Cantor derivative as in Definition \ref{definition_cantor_derivative}, and thus $\tau_\xi$ is in fact a topology on $\mu$.}
\[\tau_\xi=\{U\subseteq\mu\st d_\xi(\mu\setminus U)\subseteq\mu\setminus U\}.\]
\end{definition}

For all $\xi<\mu^+$ and $\alpha<\mu$, since $\T_\xi(\alpha)$ is the topology generated by $\B_\xi(\alpha)$, it follows that the collection of finite intersections of members of $\B_\xi(\alpha)$ is a basis for $\T_\xi(\alpha)$. That is, the collection of sets of the form
\[I\cap d_{\xi_0}(A_0)\cap\cdots\cap d_{\xi_{n-1}}(A_{n-1})\]
where $n<\omega$, $I\in\B_0$ is an interval in $\mu$, the ordinals $\xi_0\leq\cdots\leq\xi_{n-1}$ are in $F^\mu_\xi(\alpha)$ and $A_i\subseteq\mu$ for $i<n$, is a basis for the $\T_\xi(\alpha)$ topology on $\mu$.

Next let us show that the diagonal Cantor derivatives $d_\xi$ in Definition \ref{definition_tau_xi} are in fact Cantor derivatives as in Definition \ref{definition_cantor_derivative}, and thus each $d_\xi$ is the Cantor derivative of the space $(\mu,\tau_\xi)$ for all $\xi<\mu^+$.

\begin{lemma}\label{lemma_d_xi_is_cantor}
Suppose $\mu$ is regular. For all $\xi<\mu^+$ and all $A\subseteq\mu$ we have
\[d_\xi(d_\xi(A))\subseteq d_\xi(A).\]
\end{lemma}

\begin{proof}
For $\xi<\mu$ this follows easily from the fact that $d_\xi$ is defined to be the Cantor derivative of the space $(\mu,\tau_\xi)$.

Suppose $\xi\in\mu^+\setminus\mu$ and $\alpha\in d_\xi(d_\xi(A))$. Then $\alpha$ is a limit point of the set
\[d_\xi(A)=\{\beta<\alpha\st\text{$\beta$ is a limit point of $A$ in $\T_\xi(\beta)$}\}\] 
in the topology $\T_\xi(\alpha)$ on $\mu$ generated by $\B_\xi(\alpha)$. To show $\alpha\in d_\xi(A)$, we must show that $\alpha$ is a limit point of $A$ in the topology $\T_\xi(\alpha)$. Fix a basic open neighborhood $U$ of $\alpha$ in the $\T_\xi(\alpha)$ topology. Then $U$ is of the form
\[I\cap d_{\xi_0}(A_0)\cap\cdots\cap d_{\xi_{n-1}}(A_{n-1})\]
for some $\xi_i\in F^\mu_\xi(\alpha)$ and some $A_i\subseteq\mu$ where $i<n$. Since $\alpha$ is a limit point of $d_\xi(A)$ in the $\T_\xi(\alpha)$ topology on $\mu$ and since $\B_0\subseteq\T_\xi(\alpha)$, it follows that for all $\eta<\alpha$, $\alpha$ is a limit point of the set $d_\xi(A)\setminus\eta$ in the $\T_\xi(\alpha)$ topology. Since $\xi_i\in F^\mu_\xi(\alpha)$ for $i<n$ and since $\alpha$ is a limit ordinal, we can choose a $\beta<\alpha$ such that $\xi_i\in F^\mu_\xi(\beta)$ for all $i<n$. Since $\alpha$ is a limit point of $d_\xi(A)\setminus\beta$ in the $\T_\xi(\alpha)$ topology, we may choose an $\eta\in (d_\xi(A)\setminus\beta)\cap U\cap\alpha$. Since $\eta\geq\beta$ we have $\xi_i\in F^\mu_\xi(\eta)$ for all $i<n$ and thus $U\in \B_\xi(\eta)\subseteq \T_\xi(\eta)$. But since $\eta$ is a limit point of $A$ in the $\T_\xi(\eta)$ topology we have $A\cap U\cap\eta\neq\emptyset$. Thus $\alpha\in d_\xi(A)$.
\end{proof}

The following result is an easy consequence of Fact \ref{fact_cantor_derivatives} and Lemma \ref{lemma_d_xi_is_cantor}.

\begin{corollary}
Suppose $\mu$ is a regular cardinal. For each $\xi<\mu^+$, the function $d_\xi$ is the Cantor derivative of the space $(\mu,\tau_\xi)$.
\end{corollary}

Let us present the following generalizations of Bagaria's notions of $\xi$-stationarity and $\xi$-s-stationarity, which will allow us to characterize the nondiscreteness of points in the spaces $(\mu,\tau_\xi)$ for $\xi<\mu^+$.

\begin{definition}\label{definition_xi_stationary}
Suppose $\mu$ is a regular cardinal. A set $A\subseteq\mu$ is $0$-stationary in $\alpha<\mu$ if and only if $A$ is unbounded in $\alpha$. For $0<\xi<\alpha^+$, where $\alpha$ is regular, we say that $A$ is \emph{$\xi$-stationary in $\alpha$} if and only if for every $\zeta<\xi$, every set $S$ that is $\zeta$-stationary in $\alpha$ \emph{$\zeta$-reflects} to some $\beta\in A$, i.e., $S$ is $f^\alpha_\zeta(\beta)$-stationary in $\beta$. We say that an ordinal $\alpha<\mu$ is \emph{$\xi$-reflecting} if it is $\xi$-stationary in $\alpha$ as a subset of $\mu$.
\end{definition}

\begin{definition}\label{definition_xi_s_stationarity}
Suppose $\mu$ is a regular cardinal. $A$ set $A\subseteq\mu$ is \emph{$0$-simultaneously stationary in $\alpha$} (\emph{$0$-s-stationary in $\alpha$} for short) if and only if $A$ is unbounded in $\alpha$. For $0<\xi<\alpha^+$, where $\alpha$ is regular, we say that $A$ is \emph{$\xi$-simultaneously stationary in $\alpha$} (\emph{$\xi$-s-stationary in $\alpha$} for short) if and only if for every $\zeta<\xi$, every pair of subsets $S$ and $T$ that are $\zeta$-s-stationary in $\alpha$ \emph{simultaneously $\zeta$-reflect} to some $\beta\in A$, i.e., $S$ and $T$ are both $f^\alpha_\zeta(\beta)$-s-stationary in $\beta$. We say that $\alpha$ is $\xi$-s-reflecting if it is $\xi$-s-stationary in $\alpha$.
\end{definition}

\begin{remark}\label{remark_nontrivial}
Bagaria defined a set $A\subseteq\mu$ to be $\xi$-stationary in $\alpha<\mu$ if and only if for every $\zeta<\xi$, for every $S\subseteq\mu$ that is $\zeta$-stationary in $\alpha$ there is a $\beta\in A\cap\alpha$ such that $S$ is $\zeta$-stationary in $\beta$. Since $f^\alpha_\zeta$ equals the constant function $\zeta$ when $\zeta<\alpha$, it follows that Bagaria's notion of $A$ being $\xi$-stationary in $\alpha$ is equivalent to ours when $\xi<\alpha$. Bagaria comments in the paragraphs following \cite[Definition 2.6]{MR3894041} that, under his definition, no ordinal $\alpha$ can be $(\alpha+1)$-reflecting, because if $\alpha$ is the least such ordinal there is a $\beta<\alpha$ such that $\alpha\cap\beta=\beta$ is $\alpha$-stationary and thus $(\beta+1)$-stationary in $\beta$. Let us show that such an argument does \emph{not} work to rule out the existence of ordinals $\alpha$ which are $\alpha+1$-reflecting under our definition. Suppose $\alpha$ is $(\alpha+1)$-reflecting, as in Definition \ref{definition_xi_stationary}. Then there is some $\beta<\alpha$ that is $f^\alpha_\alpha(\beta)$-reflecting, but $f^\alpha_\alpha(\beta)=\beta$ and thus the conclusion is that $\beta$ is $\beta$-reflecting, and Bagaria shows that some ordinals (namely some large cardinals) $\beta$ can be $\beta$-reflecting.
\end{remark}

In order to streamline the proof of the characterization of the nonisolated points in $(\mu,\tau_\xi)$ for $\xi<\mu^+$, we will use the following auxiliary notion of $\xi$-\s-stationarity, which is often equivalent to $f^\mu_\xi(\alpha)$-s-stationarity as shown in Lemma \ref{lemma_s_hat} below.

\begin{definition}\label{definition_xi_s_hat_stationary}
Suppose $\mu$ is a regular cardinal. A set $A\subseteq\mu$ is \emph{$0$-simultaneously hat stationary in $\alpha$} ($0$-\s-stationary for short) if and only if $A$ is unbounded in $\alpha$. For $0<\xi<\mu^+$, we say that $A$ is \emph{$\xi$-simultaneously hat stationary in $\alpha$} (\emph{$\xi$-\s-stationary in $\alpha$} for short) if and only if for every $\zeta\in F^\mu_\xi(\alpha)$, every pair of subsets $S$ and $T$ of $\mu$ that are $\zeta$-\s-stationary in $\alpha$ \emph{simultaneously $\zeta$-\s-reflect} to some $\beta\in A$, i.e., $S$ and $T$ are both $\zeta$-\s-stationary in $\beta$. We say that $\alpha$ is \emph{$\xi$-\s-reflecting} if it is $\xi$-\s-stationary in $\alpha$.
\end{definition}

\begin{remark}
Notice that when $\xi<\mu$, a set $A\subseteq\mu$ is $\xi$-\s-stationary in $\alpha$ if and only if for all $\zeta\in F^\mu_\xi(\alpha)=\xi$, every pair of subsets $S$ and $T$ of $\mu$ that are $\zeta$-\s-stationary in $\alpha$ simultaneously $\zeta$-\s-reflect to some $\beta\in A$. Thus, for $\xi<\mu$, Definition \ref{definition_xi_s_hat_stationary} agrees with \cite[Definition 2.8]{MR3894041}. Furthermore, $A$ is $\mu$-\s-stationary in $\alpha$ if and only if it is $\alpha$-\s-stationary in $\alpha$. Also notice that for $\zeta<\xi<\mu$, there is a club $C_{\zeta,\xi}$ in $\mu$ such that for all $\alpha\in C_{\zeta,\xi}$ we have $F^\mu_\zeta(\alpha)\subseteq F^\mu_\xi(\alpha)$ and hence $\alpha$ being $\xi$-\s-stationary implies $\alpha$ is $\zeta$-\s-stationary.
\end{remark}

We will need the following lemma, which generalizes \cite[Proposition 2.9]{MR3894041}.

\begin{lemma}\label{lemma_intersect_with_club}
Suppose $\mu$ is a regular cardinal and $\xi<\mu^+$. There is a club $B_\xi\subseteq\mu$ such that for all regular $\alpha\in B_\xi$ if $A$ is $\xi$-\s-stationary in $\alpha$ ($f^\mu_\xi(\alpha)$-s-stationary in $\alpha$) and $C$ is a club subset of $\alpha$, then $A\cap C$ is also $\xi$-\s-stationary ($f^\mu_\xi(\alpha)$-s-stationary in $\alpha$) in $\alpha$.\end{lemma}

\begin{proof}
We will prove the lemma for $\xi$-\s-stationarity; the proof for $\xi$-s-stationarity is similar. When $\xi<\mu$ the lemma follows directly from \cite[Proposition 2.9]{MR3894041}. 

Suppose $\xi\in\mu^+\setminus\mu$ is a limit ordinal and the result holds for $\zeta<\xi$. For each $\zeta<\xi$ let $B_\zeta$ be the club subset of $\mu$ obtained from the inductive hypothesis. Let $B_\xi$ be a club subset of $\mu$ such that for all $\alpha\in B_\xi$ we have
\begin{enumerate}
\item[(i)] $\alpha\in\bigcap_{\zeta\in F^\mu_\xi(\alpha)}B_\zeta$,
\item[(ii)] $\ot(F^\mu_\xi(\alpha))$ is a limit ordinal and
\item[(iii)] $(\forall\zeta\in F^\mu_\xi(\alpha))$ $F^\mu_\xi(\alpha)\cap\zeta=F^\mu_\zeta(\alpha)$.
\end{enumerate}
Suppose $\alpha\in B_\xi$, let $A\subseteq\mu$ be $\xi$-\s-stationary in $\alpha$ and let $C$ be a club subset of $\alpha$. Since $\ot(F^\mu_\xi(\alpha))$ is a limit ordinal, to show that $A\cap C$ is $\xi$-\s-stationary in $\alpha$, it suffices to show that $A\cap C$ is $\eta$-\s-stationary in $\alpha$ for all $\eta\in F^\mu_\xi(\alpha)$. Since $F^\mu_\eta(\alpha)\subseteq F^\mu_\xi(\alpha)$ it follows that $A$ is $\eta$-\s-stationary in $\alpha$. Then, because $\alpha\in B_\eta$, it follows by the inductive hypothesis that $A\cap C$ is $\eta$-\s-stationary in $\alpha$.

Now suppose $\xi\in\mu^+\setminus\mu$ and the result holds for $\zeta\leq\xi$. We will show that it holds for $\xi+1$. For each $\zeta\leq\xi$, let $B_\zeta$ be the club subset of $\mu$ obtained by the inductive hypothesis. Let $B_{\xi+1}$ be a club subset of $\mu$ such that for all $\alpha\in B_{\xi+1}$ we have $\alpha\in \bigcap_{\zeta\in F^\mu_{\xi+1}(\alpha)} B_\zeta$. Suppose $\alpha\in B_{\xi+1}$, let $A$ be $(\xi+1)$-\s-stationary in $\alpha$ and suppose $C$ is a club subset of $\alpha$. To show that $A\cap C$ is $(\xi+1)$-\s-stationary in $\alpha$, fix any sets $S,T\subseteq\alpha$ that are $\zeta$-\s-stationary in $\alpha$ for some $\zeta\in F^\mu_{\xi+1}(\alpha)$. Since $\alpha\in B_\zeta$, it follows by the inductive hypothesis that both $S\cap C$ and $T\cap C$ are $\zeta$-\s-stationary in $\alpha$. Since $A$ is $(\xi+1)$-\s-stationary in $\alpha$, there is some $\beta\in A$ such that both $S\cap C$ and $T\cap C$ are $\zeta$-\s-stationary in $\beta$. Thus, $\beta\in A\cap C$ and both $S$ and $T$ are $\zeta$-\s-stationary in $\beta$, which establishes that $A\cap C$ is $(\xi+1)$-\s-stationary in $\alpha$.
\end{proof}

\begin{lemma}\label{lemma_s_hat}
Suppose $\mu$ is a regular cardinal. For all $\xi<\mu^+$ there is a club $C_\xi\subseteq\mu$ such that for all regular $\alpha\in C_\xi$ a set $X\subseteq\alpha$ is $f^\mu_\xi(\alpha)$-s-stationary in $\alpha$ if and only if it is $\xi$-\s-stationary in $\alpha$.
\end{lemma}

\begin{proof}
Suppose $\xi<\mu$. Let $C_\xi=\mu\setminus\xi$ and suppose $\alpha\in C_\xi$ is regular and $X\subseteq\alpha$. Since $F^\mu_\xi(\alpha)=f^\mu_\xi(\alpha)=\xi$ it is easy to see that from the definitions that $X$ is $\xi$-s-stationarity in $\alpha$ if and only if it is $\xi$-\s-stationarity in $\alpha$.



Now suppose that $\xi\in\mu^+\setminus\mu$ is a limit ordinal and the result holds for $\zeta<\xi$; the case in which $\xi\in\mu^+\setminus\mu$ is a successor is easier and is therefore left to the reader. For each $\zeta<\xi$ let $C_\zeta$ be the club obtained by the inductive hypothesis, and let $B_\xi$ be the club obtained from Lemma \ref{lemma_intersect_with_club}. Let $C_\xi$ be a club subset of $\mu$ such that for all $\alpha\in C_\xi$ we have 
\begin{enumerate}
\item[(i)] $\alpha\in B_\xi\cap \bigcap_{\zeta\in F^\mu_\xi(\alpha)}d_0(C_\zeta)$,
\item[(ii)] $\ot(F^\mu_\xi(\alpha))$ is a limit ordinal,
\item[(iii)] $f^\mu_\xi(\alpha)=\bigcup_{\zeta\in F^\mu_\xi(\alpha)}f^\mu_\zeta(\alpha)$,
\item[(iv)] for all $\zeta\in F^\mu_\xi(\alpha)$ there is a club $D^\alpha_{\xi,\zeta}$ in $\alpha$ such that for all $\beta\in D^\alpha_{\xi,\zeta}$ we have $f^\alpha_{f^\mu_\xi(\alpha)}(\beta)=f^\mu_\xi(\beta)$,
\item[(v)] $(\forall\zeta\in F^\mu_\xi(\alpha))$ $F^\mu_\xi(\alpha)\cap \zeta=F^\mu_\zeta(\alpha)$.
\end{enumerate}
Suppose $X$ is $f^\mu_\xi(\alpha)$-s-stationary in $\alpha$ and fix sets $S,T\subseteq\alpha$ that are $\eta$-\s-stationary in $\alpha$ for some $\eta\in F^\mu_\xi(\alpha)$. Since $\alpha\in d_0(C_\eta)\subseteq C_\eta$, it follows by the inductive hypothesis that both $S$ and $T$ are $f^\mu_\eta(\alpha)$-s-stationary in $\alpha$.  By (iii) we have $f^\mu_\eta(\alpha)<f^\mu_\xi(\alpha)$, and since $\alpha\in B_\xi$ it follows by Lemma \ref{lemma_intersect_with_club} that $X\cap C_\eta\cap D^\alpha_{\xi,\zeta}$ is $f^\mu_\xi(\alpha)$-s-stationarity in $\alpha$. Hence there is a $\beta\in X\cap C_\eta\cap D^\alpha_{\xi,\zeta}$ such that both $S$ and $T$ are $f^\alpha_{f^\mu_\eta(\alpha)}(\beta)$-s-stationary in $\beta$. Since it follows from (v) that $f^\alpha_{f^\mu_\eta(\alpha)}(\beta)=f^\mu_\eta(\beta)$, both $S$ and $T$ are $f^\mu_\eta(\beta)$-s-stationary in $\beta$. Since $\beta\in C_\eta$ it follows that $S$ and $T$ are both $\eta$-\s-stationary in $\beta$. Thus $S$ is $\xi$-\s-stationar in $\alpha$. Conversely, suppose $X$ is $\xi$-\s-stationary in $\alpha$ and fix sets $S,T\subseteq\alpha$ that are $\eta$-s-stationary in $\alpha$ for some $\eta<f^\mu_\xi(\alpha)$. Let $\pi^\mu_{\xi,\alpha}:F^\mu_\xi(\alpha)\to f^\mu_\xi(\alpha)$ be the transitive collapse of $F^\mu_\xi(\alpha)$ and let $\hat\eta=(\pi^\mu_{\xi,\alpha})^{-1}(\eta)$. Since $\alpha\in C_{\hat\eta}$, it follows that $S$ and $T$ are $\hat\eta$-\s-stationary in $\alpha$. Since $X\cap D^\alpha_{\xi,\hat\eta}$ is $\xi$-\s-stationary in $\alpha$, there is a $\beta\in X\cap D^\alpha_{\xi,\hat\eta}\cap\alpha$ such that $S$ and $T$ are both $f^\mu_{\hat\eta}(\beta)$-\s-stationary in $\beta$. Since $\beta\in D^\alpha_{\xi,\hat\eta}$, the sets $S$ and $T$ are $f^\alpha_{f^\mu_{\hat\eta}(\alpha)}(\beta)$-\s-stationary in $\beta$. Since $\eta=f^\mu_{\hat\eta}(\alpha)$ we see that both $S$ and $T$ are $f^\alpha_\eta(\beta)$-s-stationary in $\beta$. Thus $X$ is $\xi$-s-stationary.
\end{proof}

In order to characterize the nonisolated points of the spaces $(\mu,\tau_\xi)$, for $\xi<\mu^+$, in terms of $\eta$-s-reflecting cardinals, we will need the following proposition, which generalizes \cite[Proposition 2.10]{MR3894041}.

\begin{proposition}\label{proposition_meat}
Suppose $\mu$ is a regular cardinal.
\begin{enumerate}
\item For all $\xi<\mu^+$ there is a club $C_\xi\subseteq\mu$ such that for all $A\subseteq\mu$ we have
\[d_\xi(A)\cap C_\xi=\{\alpha<\mu\st\text{$A$ is $\xi$-\s-stationary in $\alpha$}\}\cap C_\xi.\]
\item For all $\xi<\mu^+$ there is a club $D_\xi\subseteq\mu$ such that for all $\alpha\in D_\xi$ and all $A\subseteq\mu$ we have that $A$ is $(\xi+1)$-\s-stationary in $\alpha$ if and only if $A\cap d_\zeta(S)\cap d_\zeta(T)\neq\emptyset$ (equivalently, if and only if $A\cap d_\zeta(S)\cap d_\zeta(T)$ is $\zeta$-\s-stationary in $\alpha$) for every $\zeta\in F^\mu_{\xi+1}(\alpha)$ and every pair $S,T$ of subsets of $\alpha$ that are $\zeta$-\s-stationary in $\alpha$.
\item For all $\xi<\mu^+$ there is a club $E_\xi\subseteq\mu$ such that for all $\alpha\in E_\xi$ and all $A\subseteq\mu$, if $A$ is $\xi$-\s-stationary in $\alpha$ and $A_i$ is $\zeta_i$-\s-stationary in $\alpha$ for some $\zeta_i\in F^\mu_\xi(\alpha)$, for all $i<n$ where $n<\omega$, then $A\cap d_{\zeta_0}(A_0)\cap \cdots\cap d_{\zeta_{n-1}}(A_{n-1})$ is $\xi$-\s-stationary in $\alpha$.
\end{enumerate}
\end{proposition}

\begin{proof}
We will prove (1) -- (3) by simultaneous induction on $\xi$, for $\xi$-\s-stationarity. For $\xi<\mu$, (1) -- (3) follow directly from \cite[Proposition 2.10]{MR3894041}, taking $C_\xi=D_\xi=E_\xi=\mu$.

Let us first show that if $\xi\in\mu^+\setminus\mu$ is a limit ordinal and (1) -- (3) hold for all $\zeta<\xi$, then (1) -- (3) hold for $\xi$.

First we will show that for $\xi\in\mu^+\setminus\mu$ a limit ordinal, if (1) holds for $\zeta<\xi$ then (1) holds for $\xi$. For each $\zeta<\xi$, let $C_\zeta$ be the club subset of $\mu$ obtained from (1). Let $C_\xi$ be a club subset of $\mu$ such that for all $\alpha\in C_\xi$ we have
\begin{enumerate}
\item[(i)] $\alpha\in\bigcap_{\zeta\in F^\mu_\xi(\alpha)}C_\zeta$,
\item[(ii)] $\ot(F^\mu_\xi(\alpha))$ is a limit ordinal and
\item[(iii)] $(\forall\zeta\in F^\mu_\xi(\alpha))$ $F^\mu_\xi(\alpha)\cap\zeta=F^\mu_\zeta(\alpha)$.
\end{enumerate}
Now fix $A\subseteq\mu$ and suppose $\alpha\in d_\xi(A)\cap C_\xi$. Then $\alpha$ is a limit point of $A$ in the $\T_\xi(\alpha)$ topology on $\mu$. For each $\zeta\in F^\mu_\xi(\alpha)$ we have $F^\mu_\zeta(\alpha)\subseteq F^\mu_\xi(\alpha)$, which implies $\T_\zeta(\alpha)\subseteq\T_\xi(\alpha)$, and hence $\alpha$ is a limit point of $A$ in the $\T_\zeta(\alpha)$ topology on $\mu$. Thus $\alpha\in \bigcap_{\zeta\in F^\mu_\xi(\alpha)}d_\zeta(A)$. Since $\alpha\in \bigcap_{\zeta\in F^\mu_\xi(\alpha)}C_\zeta$, it follows that $A$ is $\zeta$-\s-stationary in $\alpha$ for all $\zeta\in F^\mu_\xi(\alpha)$. By (ii) and (iii), this implies that $A$ is $\xi$-\s-stationary in $\alpha$. Conversely, suppose $A$ is $\xi$-\s-stationary in $\alpha$ and $\alpha\in C_\xi$. To show that $\alpha\in d_\xi(A)$ we must show that $\alpha$ is a limit point of $A$ in the $\T_\xi(\alpha)$ topology on $\mu$ generated by $\B_\xi(\alpha)$. Fix a basic open neighborhood $U$ of $\alpha$ in $\T_\xi(\alpha)$. Then $U$ is of the form 
\[I\cap d_{\zeta_0}(A_0)\cap\cdots\cap d_{\zeta_{n-1}}(A_{n-1})\]
where $I$ is an interval in $\mu$, $n<\omega$, and for all $i<n$ we have $\zeta_i\in F^\mu_\xi(\alpha)$ and $A_i\subseteq\mu$. By (ii) we can choose some $\eta\in F^\mu_\xi(\alpha)$ with $\eta>\max\{\zeta_i\st i<n\}$. By (iii), for each $i<n$ we have $\zeta_i\in F^\mu_\xi(\alpha)\cap\eta=F^\mu_\eta(\alpha)$ and hence $U$ is an open neighborhood of $\alpha$ in the $\T_\eta(\alpha)$ topology. Since $F^\mu_\eta(\alpha)\subseteq F^\mu_\xi(\alpha)$, it follows that $A$ is $\eta$-\s-stationary in $\alpha$, and since $\alpha\in C_\eta$ we have that $\alpha\in d_\eta(A)$. Thus $\alpha$ is a limit point of $A$ in the $\T_\eta(\alpha)$ topology, so $A\cap U\setminus\{\alpha\}\neq\emptyset$. This shows that $\alpha$ is a limit point of $A$ in the $\T_\xi(\alpha)$ topology.

Let us show that for $\xi\in \mu^+\setminus\mu$ a limit ordinal, if (3) holds for $\zeta<\xi$, then (3) holds for $\xi$. Let $E_\xi$ be a club subset of $\mu$ such that for all $\alpha\in E_\xi$ we have
\begin{enumerate}
\item[(i)] $\alpha\in \bigcap_{\zeta\in F^\mu_\xi(\alpha)}E_\zeta$,
\item[(ii)] $\ot(F^\mu_\xi(\alpha))$ is a limit ordinal and
\item[(iii)] $(\forall\zeta\in F^\mu_\xi(\alpha))$ $F^\mu_\xi(\alpha)\cap\zeta=F^\mu_\zeta(\alpha)$.
\end{enumerate}
Suppose $\alpha\in E_\xi$. Let $A\subseteq\mu$ be $\xi$-\s-stationary in $\alpha$ and, for $i<n$, suppose $A_i$ is $\zeta_i$-\s-stationary in $\alpha$ for some $\zeta_i\in F^\mu_\xi(\alpha)$. We must show that $A\cap d_{\zeta_0}(A_0)\cap\cdots\cap d_{\zeta_{n-1}}(A_{n-1})$ is $\xi$-\s-stationary in $\alpha$. Fix a pair of sets $S,T\subseteq\mu$ that are $\zeta$-\s-stationary in $\alpha$ for some $\zeta\in F^\mu_\xi(\alpha)$. Using (ii), choose $\eta\in F^\mu_\xi(\alpha)$ with $\eta>\max(\{\zeta_i\st i< n\}\cup\{\zeta\})$. Since $F^\mu_\xi(\alpha)\cap\eta=F^\mu_\eta(\alpha)$, it follows that $A$ is $\eta$-\s-stationary in $\alpha$, and by our assumption that (3) holds for $\eta<\xi$ and the fact that $\alpha\in E_\eta$, it follows that $A\cap d_{\zeta_0}(A_0)\cap\cdots\cap d_{\zeta_{n-1}}(A_{n-1})$ is $\eta$-\s-stationary in $\alpha$. Thus, there is a $\beta\in A\cap d_{\zeta_0}(A_0)\cap\cdots\cap d_{\zeta_{n-1}}(A_{n-1})$ such that both $S$ and $T$ are $\zeta$-\s-stationary in $\beta$.

Now we will show that for a limit ordinal $\xi\in\mu^+\setminus\mu$, if (1) and (3) hold for $\zeta\leq\xi$, then (2) holds for $\xi$. For each $\zeta\leq\xi$, let $B_\zeta$ be the club subset of $\mu$ obtained from Lemma \ref{lemma_intersect_with_club}. Let $D_\xi$ be a club subset of $\mu$ such that for all $\alpha\in D_\xi$ we have
\begin{enumerate}[(i)]
\item $\ot(F^\mu_\xi(\alpha))$ is a limit ordinal,
\item $(\forall\zeta\in F^\mu_\xi(\alpha))$ $F^\mu_\xi(\alpha)\cap\zeta=F^\mu_\zeta(\alpha)$ and
\item $\alpha\in\bigcap_{\zeta\in F^\mu_{\xi+1}(\alpha)}(B_\zeta\cap d_0(C_\zeta)\cap d_0(E_\zeta))$ where the $C_\zeta$'s and $E_\zeta$'s are obtained by the inductive hypothesis from (1) and (3) respectively.
\end{enumerate}
Suppose $\alpha\in D_\xi$. For the forward direction of (2), let $A$ be $(\xi+1)$-\s-stationary in $\alpha$ and fix a pair $S,T$ of subsets of $\alpha$ that are $\zeta$-\s-stationary in $\alpha$ for some $\zeta\in F^\mu_{\xi+1}(\alpha)$. Since $\alpha\in d_0(C_\zeta)$, it follows that $C_\zeta$ is closed and unbounded in $\alpha$ and hence the set $A\cap C_\zeta$ is $(\xi+1)$-\s-stationary in $\alpha$. Hence there exists a $\beta\in A\cap C_\zeta$ such that $S$ and $T$ are both $\zeta$-\s-stationary in $\beta$, and since $\beta\in C_\zeta$ we have $\beta\in A\cap d_\zeta(S)\cap d_\zeta(T)$ by (1). To see that $A\cap d_\zeta(S)\cap d_\zeta(T)$ is $\zeta$-\s-stationary in $\alpha$, fix sets $X,Y\subseteq\alpha$ that are $\eta$-\s-stationary in $\alpha$ for some $\eta\in F^\mu_\zeta(\alpha)$. Since $\alpha\in E_\zeta$, it follows by (3) that $S\cap d_\eta(X)$ and $T\cap d_\eta(Y)$ are $\zeta$-\s-stationary in $\alpha$. Since $A\cap C_\zeta$ is $(\xi+1)$-\s-stationary in $\alpha$ there is some $\beta\in A\cap C_\zeta$ such that $S\cap d_\eta(X)$ and $T\cap d_\eta(Y)$ are both $\zeta$-\s-stationary in $\beta$. Since $\beta\in C_\zeta$, it follows that $\beta\in A\cap C_\zeta\cap d_\zeta(S\cap d_\eta(X))\cap d_\zeta(T\cap d_\eta(Y))\neq\emptyset$.
Now we have
\[\emptyset\neq A\cap C_\zeta\cap d_\zeta(S\cap d_\eta(X))\cap d_\zeta(T\cap d_\eta(Y))\subseteq A\cap C_\zeta\cap d_\zeta(S)\cap d_\zeta(T)\cap d_\eta(X)\cap d_\eta(Y),\]
and hence $A\cap d_\zeta(S)\cap d_\zeta(T)$ is $\zeta$-\s-stationary in $\alpha$.

For the reverse direction of (2), suppose that $\alpha\in D_\xi$ and for all $\zeta\in F^\mu_{\xi+1}(\alpha)$, if $S,T\subseteq\alpha$ are both $\zeta$-\s-stationary in $\alpha$ then $A\cap d_\zeta(S)\cap d_\zeta(T)\neq\emptyset$. To show that $A$ is $(\xi+1)$-\s-stationary in $\alpha$, fix $\zeta\in F^\mu_{\xi+1}(\alpha)$ and suppose $S,T\subseteq\alpha$ are $\zeta$-\s-stationary in $\alpha$. By Lemma \ref{lemma_intersect_with_club} and the fact that $\alpha\in B_\zeta\cap d_0(C_\zeta)$, it follows that $S\cap C_\zeta$ and $T$ are both $\zeta$-\s-stationary in $\alpha$. Thus, by (1), there is a $\beta\in A\cap d_\zeta(S\cap C_\zeta)\cap d_\zeta(T)$. Now since $\beta\in C_\zeta\cap d_\zeta(S)\cap d_\zeta(T)$, it follows by (1) that $S$ and $T$ are both $\zeta$-\s-stationary in $\beta$. Hence $A$ is $(\xi+1)$-\s-stationary in $\alpha$.

It remains to show that if $\xi\in\mu^+\setminus\mu$ is an ordinal and (1), (2) and (3) hold for $\zeta\leq\xi$, then (1), (2) and (3) also hold for $\xi+1$.

Given that (1), (2) and (3) hold for $\zeta\leq\xi$, let us show that (3) holds for $\xi+1$. For $\zeta\leq\xi$, let $C_\zeta$, $D_\zeta$ and $E_\zeta$ be the club subsets of $\mu$ obtained from (1), (2) and (3) respectively. For each $\zeta\leq\xi$, let $B_\zeta$ be the club subset of $\mu$ obtained from Lemma \ref{lemma_intersect_with_club}. Let $E_{\xi+1}$ be a club subset of $\mu$ such that for all $\alpha\in E_{\xi+1}$ we have
\begin{enumerate}
\item[(i)] $\alpha\in \bigcap_{\zeta\in F^\mu_{\xi+1}(\alpha)} (B_\zeta\cap d_0(C_\zeta)\cap D_\zeta\cap E_\zeta)$,
\item[(ii)] $\alpha\in C_\xi\cap D_\xi\cap E_\xi$ and
\item[(iii)] $(\forall\zeta\in F^\mu_{\xi+1}(\alpha))$ $F^\mu_{\xi+1}(\alpha)\cap\zeta=F^\mu_\zeta(\alpha)$.
\end{enumerate}
Suppose $\alpha\in E_{\xi+1}$ and $A\subseteq\mu$ is $(\xi+1)$-\s-stationary in $\alpha$. Let $n<\omega$ and for each $i<n$ suppose $\zeta_i\in F^\mu_{\xi+1}(\alpha)$ and $A_i$ is $\zeta_i$-\s-stationary in $\alpha$. We must show that $A\cap d_{\zeta_0}(A_0)\cap \cdots d_{\zeta_{n-1}}(A_{n-1})$ is $(\xi+1)$-\s-stationary in $\alpha$. Fix sets $S,T\subseteq\alpha$ which are $\zeta$-\s-stationary in $\alpha$ for some $\zeta\in F^\mu_{\xi+1}(\alpha)$. We must show that there is a $\beta\in A\cap d_{\zeta_0}(A_0)\cap \cdots d_{\zeta_{n-1}}(A_{n-1})$ such that $S\cap\beta$ and $T\cap\beta$ are $\zeta$-\s-stationary in $\beta$. Since $\alpha\in C_\zeta$, it follows by an inductive application of (1) that in order to prove (3) holds for $\xi+1$, it suffices to show that
\begin{align}
A\cap d_{\zeta_0}(A_0)\cap\cdots\cap d_{\zeta_{n-1}}(A_{n-1})\cap d_\zeta(S)\cap d_\zeta(T)\cap C_\zeta&\neq\emptyset.\label{eqn_for_3}
\end{align}

Let us proceed to prove \ref{eqn_for_3} by induction on $n$. First, let us consider the case in which $\zeta_0=\zeta$. Since $\alpha\in D_\zeta$ and $A$ is $(\xi+1)$-\s-stationary in $\alpha$, it follows inductively from (2) that  the set $d_\zeta(S)\cap d_\zeta(T)\supseteq A\cap d_\zeta(S)\cap d_\zeta(T)$ is $\zeta$-\s-stationary in $\alpha$. Now since $A$ is $(\xi+1)$-\s-stationary in $\alpha$, the sets $A_0$ and $d_\zeta(S)\cap d_\zeta(T)$ are $\zeta$-\s-stationary in $\alpha$ and since $\alpha\in D_\zeta$, it follows from (2) that $A\cap d_{\zeta_0}(A_0)\cap d_\zeta(d_\zeta(S)\cap d_\zeta(T))$ is $\zeta$-\s-stationary in $\alpha$. By Lemma \ref{lemma_intersect_with_club}, since $\alpha\in B_\zeta\cap d_0(C_\zeta)$ we have that $A\cap d_{\zeta_0}(A_0)\cap d_\zeta(d_\zeta(S)\cap d_\zeta(T))\cap C_\zeta$ is $\zeta$-\s-stationary in $\alpha$. Since $A\cap d_{\zeta_0}(A_0)\cap d_\zeta(d_\zeta(S)\cap d_\zeta(T))\cap C_\zeta\subseteq A\cap d_{\zeta_0}(A_0)\cap d_\zeta(S)\cap d_\zeta(T)\cap C_\zeta$, this establishes (\ref{eqn_for_3}) in case $n=1$ and $\zeta_0=\zeta$. Second, let us consider the case in which $n=1$ and $\zeta_0<\zeta$. Since $F^\mu_\zeta(\alpha)\subseteq F^\mu_{\xi+1}(\alpha)$, it follows that $A$ is $\zeta$-\s-stationary in $\alpha$. Since $\alpha\in E_\zeta$ and $\zeta\in F^\mu_{\xi+1}(\alpha)$, we may inductively apply (3) to see that $A\cap d_{\zeta_0}(A_0)$ and thus also $d_{\zeta_0}(A_0)$ is $\zeta$-\s-stationary in $\alpha$. Hence because $\alpha\in D_\zeta$, it follows by an inductive application of (2) that $A\cap d_\zeta(d_{\zeta_0}(A_0))\cap d_\zeta(d_\zeta(S)\cap d_\zeta(T))$ is $\zeta$-\s-stationary in $\alpha$ and by Lemma \ref{lemma_intersect_with_club} and the fact that $\alpha\in B_\zeta\cap d_0(C_\zeta)$, we see that the set $A\cap d_\zeta(d_{\zeta_0}(A_0))\cap d_\zeta(d_\zeta(S)\cap d_\zeta(T))\cap C_\zeta$ is also $\zeta$-\s-stationary in $\alpha$. Since $A\cap d_\zeta(d_{\zeta_0}(A_0))\cap d_\zeta(d_\zeta(S)\cap d_\zeta(T))\cap C_\zeta\subseteq A\cap d_{\zeta_0}(A_0)\cap d_\zeta(S)\cap d_\zeta(T)\cap C_\zeta$, this establishes (\ref{eqn_for_3}) in the second case where $n=1$ and $\zeta_0<\zeta$. Thirdly, suppose $n=1$ and $\zeta_0>\zeta$. Then by an inductive application of (2), the set $A\cap d_{\zeta_0}(A_0)$ is $\zeta_0$-\s-stationary in $\alpha$. Since $\alpha \in B_\zeta\cap d_0(C_{\zeta_0})$, it follows from by Lemma \ref{lemma_intersect_with_club} that $A\cap d_{\zeta_0}(A_0)\cap C_\zeta$ is also $\zeta_0$-\s-stationary in $\alpha$. Since $\zeta\in F^\mu_{\zeta_0}(\alpha)$, we see that there is some $\beta\in A\cap d_{\zeta_0}(A_0)\cap C_\zeta$ such that both $S$ and $T$ are $\zeta$-\s-stationary in $\beta$, and thus $\beta\in A\cap d_{\zeta_0}(A_0)\cap d_\zeta(S)\cap d_\zeta(T)\cap C_\zeta$. This establishes that (6) holds for $n=1$.

Now suppose $n>1$. Since $\alpha\in D_\zeta$, it follows inductively by (2) that $d_\zeta(S)\cap d_\zeta(T)$ is $\zeta$-\s-stationary in $\alpha$. Since $\alpha\in E_\zeta$, it follows by an inductive application of (3) that $d_\zeta(d_\zeta(S)\cap d_\zeta(T))$ is $\mu$-\s-stationary in $\alpha$. Furthermore, since $\alpha\in C_{\zeta_{n-1}}$ we have that $A_{n-1}$ is $\zeta_{n-1}$-\s-stationary in $\alpha$ and thus we see that, again by an inductive application of (3) using the fact that $\alpha\in E_{\zeta_{n-1}}$ the set $d_{\zeta_{n-1}}(A_{n-1})$ is $\zeta_{n-1}$-\s-stationary in $\alpha$. Also, by the inductive hypothesis on $n$, the set $A\cap d_{\zeta_0}(A_0)\cap \cdots\cap d_{\zeta_{n-1}}(A_{n-2})$ is $(\xi+1)$-\s-stationary in $\alpha$. Therefore, by an inductive application of (2), the set
\[A\cap d_{\zeta_0}(A_0)\cap \cdots\cap d_{\zeta_{n-2}}(A_{n-2})\cap d_{\zeta_{n-1}}(d_{\zeta_{n-1}}(A_{n-1}))\cap d_\zeta(d_\zeta(S)\cap d_\zeta(T))\cap C_\zeta\]
which is contained in
\[A\cap d_{\zeta_0}(A_0)\cap \cdots\cap d_{\zeta_{n-2}}(A_{n-2})\cap d_{\zeta_{n-1}}(A_{n-1})\cap d_\zeta(S)\cap d_\zeta(T)\cap C_\zeta\]
is $\mu$-\s-stationary in $\alpha$. This establishes (\ref{eqn_for_3}) and hence (3) holds for $\xi+1$.

Next, given that (1) and (2) hold for $\zeta\leq\xi$ and (3) holds for $\zeta\leq\xi+1$, let us show that (1) holds for $\xi+1$. For each $\zeta\leq\xi$, let $C_\zeta$ be the club subset of $\mu$ obtained from (1). Also let $E_{\xi+1}$ be the club obtained from (3). For each $\zeta\leq\xi+1$, let $B_\zeta$ be the club subset of $\mu$ obtained from Lemma \ref{lemma_intersect_with_club}. Now we let $C_{\xi+1}$ be a club subset of $\mu$ such that for all $\alpha\in C_{\xi+1}$ we have
\begin{enumerate}
\item[(i)] $\alpha\in \bigcap_{\zeta\in F^\mu_{\xi+1}(\alpha)} (B_\zeta\cap d_0(C_\zeta))$ and
\item[(ii)] $\alpha\in B_{\xi+1}\cap E_{\xi+1}$.
\end{enumerate} 
Suppose $\alpha\in d_{\xi+1}(A)\cap C_{\xi+1}$. Then $\alpha$ is a limit point of $A$ in the $\T_{\xi+1}(\alpha)$ topology on $\mu$. To show that $A$ is $(\xi+1)$-\s-stationary in $\alpha$, fix $\zeta\in F^\mu_{\xi+1}(\alpha)$ and suppose $S,T\subseteq\alpha$ are $\zeta$-\s-stationary in $\alpha$. Since $\alpha\in C_{\xi+1}$ we have $\alpha\in d_0(C_\zeta)$ and thus $\alpha\in d_\zeta(S)\cap d_\zeta(T)$. Since $d_0(C_\zeta)\cap d_\zeta(S)\cap d_\zeta(T)\in \B_{\xi+1}(\alpha)$ is a basic open neighborhood of $\alpha$ in the $\T_{\xi+1}(\alpha)$ topology, and since $\alpha\in d_{\xi+1}(A)$, it follows that there is some $\beta\in A\cap d_0(C_\zeta)\cap d_\zeta(S)\cap d_\zeta(T)\setminus\{\alpha\}$. Since $\beta\in C_\zeta$, it follows by the inductive hypothesis on (1) that both $S$ and $T$ are $\zeta$-\s-stationary in $\beta$. Thus, $A$ is $(\xi+1)$-\s-stationary in $\alpha$.

Now suppose $\alpha\in C_{\xi+1}$ and $A$ is $(\xi+1)$-\s-stationary in $\alpha$. To show that $\alpha\in d_{\xi+1}(A)$, fix a basic open set $U\in\B_{\xi+1}(\alpha)$ in the $\T_{\xi+1}(\alpha)$ topology with $\alpha\in U$. Then $U$ is of the form
$I\cap d_{\zeta_0}(A_0)\cap\cdots\cap d_{\zeta_{n-1}}(A_{n-1})$, where $I\in\B_0(\alpha)$ is an interval in $\mu$, $n<\omega$ and for all $i<n$ we have $\zeta_i\in F^\mu_{\xi+1}(\alpha)$ and $A_i\subseteq\mu$. Since $\alpha\in C_{\xi+1}$ and $\alpha\in I\cap d_{\zeta_0}(A_0)\cap\cdots\cap d_{\zeta_{n-1}}(A_{n-1})$, it follows by the inductive hypothesis that $A_i$ is $\zeta_i$-\s-stationary in $\alpha$ for all $i<n$. Then since $A$ is $(\xi+1)$-\s-stationary in $\alpha$, it follows from the fact that (3) holds for $\xi+1$ and $\alpha\in E_{\xi+1}$, that $A\cap d_{\zeta_0}(A_0)\cap\cdots\cap d_{\zeta_{n-1}}(A_{n-1})$ is $(\xi+1)$-\s-stationary in $\alpha$. Furthermore, since $I\cap\alpha$ is a club subset of $\alpha$ and $\alpha\in B_{\xi+1}$, Lemma \ref{lemma_intersect_with_club} implies that $A\cap I\cap d_{\zeta_0}(A_0)\cap\cdots\cap d_{\zeta_{n-1}}(A_{n-1})\cap\alpha$ is $(\xi+1)$-\s-stationary in $\alpha$ and is thus nonempty. This establishes that $\alpha$ is a limit point of $A$ in the $\T_{\xi+1}(\alpha)$ topology, that is, $\alpha\in d_{\xi+1}(A)$.

Finally, given that (1) and (3) hold for $\zeta\leq\xi+1$, let us show that (2) holds for $\xi+1$. Let $D_{\xi+1}$ be a club subset of $\mu$ such that for all $\alpha\in D_{\xi+1}$ we have
\begin{enumerate}
\item[(i)] $(\forall\zeta\in F^\mu_{\xi+2}(\alpha))$ $F^\mu_{\xi+2}(\alpha)\cap\zeta=F^\mu_\zeta(\alpha)$;
\item[(ii)] $\alpha\in\bigcap_{\zeta\in F^\mu_{\xi+2}(\alpha)}C_\zeta\cap E_\zeta$;
\item[(iii)] $\alpha\in B_{\xi+2}\cap C_{\xi+1}$ where $B_{\xi+2}$ is the club subset of $\mu$ obtained from Lemma \ref{lemma_intersect_with_club} and $C_{\xi+1}$ is obtained from our inductive assumption on (1); and
\item[(iv)] for all $\zeta\in F^\mu_{\xi+2}(\alpha)$ and all $\eta\in F^\mu_\zeta(\alpha)$ we have 
\[\alpha\in d_0(\{\beta<\mu\st F^\mu_\eta(\beta)\subseteq F^\mu_\zeta(\beta)\}).\]
\end{enumerate}
Suppose $A$ is $(\xi+2)$-\s-stationary in $\alpha$ and $\alpha\in D_{\xi+1}$. Let $S,T\subseteq\alpha$ be $\zeta$-\s-stationary in $\alpha$ for some $\zeta\in F^\mu_{\xi+2}(\alpha)$. We must show that $A\cap d_\zeta(S)\cap d_\zeta(T)$ is $\zeta$-\s-stationary in $\alpha$. Fix $\eta\in F^\mu_\zeta(\alpha)$ and suppose $X$ and $Y$ are $\eta$-\s-stationary subsets of $\alpha$. By an inductive application of (1), it will suffice to show that
\[A\cap d_\zeta(S)\cap d_\zeta(T)\cap d_\eta(X)\cap d_\eta(Y)\cap C_\eta\neq\emptyset.\]
Since $\alpha\in E_\zeta$ it follows inductively by (3) that the sets $S\cap d_\eta(X)$ and $T\cap d_\eta(Y)$ are $\zeta$-\s-stationary in $\alpha$. By (iv) the set $\{\beta<\alpha\st F^\mu_\eta(\beta)\subseteq F^\mu_\zeta(\beta)\}$ is club in $\alpha$ and therefore, by Lemma \ref{lemma_intersect_with_club} since $\alpha\in B_{\xi+2}$, the set 
\[A\cap\{\beta<\alpha\st F^\mu_\eta(\beta)\subseteq F^\mu_\zeta(\beta)\}\] is $(\xi+2)$-\s-stationary in $\alpha$. Since $\zeta\in F^\mu_{\xi+2}(\alpha)$ and since $\alpha\in E_\zeta$, it follows by (3) for $\zeta$ that there is some $\beta\in A$ such that
\[\beta\in d_\zeta(S\cap d_\eta(X))\cap d_\zeta(T\cap d_\eta(Y)).\] 
Thus we have
\[\beta\in d_\zeta(S)\cap d_\zeta(d_\eta(X))\cap d_\zeta(T)\cap d_\zeta(d_\eta(Y)).\]
Since $F^\mu_\eta(\beta)\subseteq F^\mu_\zeta(\beta)$, it follows that $\T_\eta(\beta)\subseteq \T_\zeta(\beta)$, and therefore we obtain
\[\beta\in d_\zeta(S)\cap d_\eta(d_\eta(X))\cap d_\zeta(T)\cap d_\eta(d_\eta(Y)).\]
By Lemma \ref{lemma_d_xi_is_cantor}, we have
\[\beta\in d_\zeta(S)\cap d_\zeta(T)\cap d_\eta(X) \cap d_\eta(Y)\]
as desired.
\end{proof}

Now we are ready to characterize the nonisolated points of the spaces $(\mu_\xi,\tau_\xi)$ (on a club) in terms of $\eta$-\s-reflecting cardinals. The following is a generalization of \cite[Theorem 2.11]{MR3894041}.

\begin{theorem}\label{theorem_xi_s_hat_nonisolated}
Suppose $\mu$ is a regular cardinal. For all $\xi<\mu^+$ if $C_\xi$ is the club subset of $\mu$ obtained from Proposition \ref{proposition_meat}(1), then for all $\alpha\in C_\xi$, the ordinal $\alpha$ is not isolated in the $\tau_\xi$ topology on $\mu$ if and only if $\alpha$ is $\xi$-\s-reflecting.
\end{theorem}

\begin{proof}
For $\xi<\mu$ the result follows directly from \cite[Theorem 2.11]{MR3894041}. 

Suppose $\xi\in\mu^+\setminus\mu$ and $\alpha\in C_\xi$. By Definition \ref{definition_tau_xi}, we have
\begin{align*}
\text{$\alpha$ is not isolated in the $\tau_\xi$ topology} &\iff \{\alpha\}\notin\tau_\xi\\
	&\iff d_\xi(\mu\setminus\{\alpha\})\not\subseteq\mu\setminus\{\alpha\}\\
	&\iff \alpha\in d_\xi(\mu\setminus\{\alpha\})\\
	&\iff \text{$\alpha$ is $\xi$-\s-stationary in $\alpha$}\\
	&\iff \text{$\alpha$ is $\xi$-\s-reflecting.}
\end{align*}
\end{proof}

By applying Lemma \ref{lemma_s_hat} and Theorem \ref{theorem_xi_s_hat_nonisolated} we easily obtain the following.

\begin{theorem}\label{theorem_xi_s_nonisolated}
Suppose $\mu$ is a regular cardinal. For all $\xi<\mu^+$ there is a club $C_\xi\subseteq\mu$ such that for all $\alpha \in C_\xi$ we have that $\alpha$ is not isolated in the $\tau_\xi$ topology on $\mu$ if and only if $\alpha$ is $f^\mu_\xi(\alpha)$-s-reflecting.
\end{theorem}

In order to show that $\Pi^1_\xi$-indescribability can be used to obtain the nondiscreteness of the topologies $\tau_\xi$ on $\mu$ for $\xi<\mu^+$, we need the following expressibility result.

\begin{lemma}\label{lemma_expressing_s_stationarity}
Suppose $\kappa$ is a regular cardinal. For all $\xi<\kappa^+$ there is a formula $\Pi^1_\xi$ formula $\varphi_\xi(X)$ over $V_\kappa$ and a club $C_\xi$ subset of $\kappa$ such that for all $A\subseteq\kappa$ we have
\[\text{$A$ is $\xi$-s-stationary in $\kappa$ if and only if $V_\kappa\models\varphi_\xi(A)$}\]
and for all $\alpha\in C_\xi$ we have
\[\text{$A$ is $f^\kappa_\xi(\alpha)$-s-stationary in $\alpha$ if and only if $V_\alpha\models\varphi_\xi(A)\res^\kappa_\alpha$}\]
\end{lemma}

\begin{proof}
We follow the proof of \cite[Proposition 4.3]{MR3894041} and proceed by induction on $\xi$. We let $\varphi_0(X)$ be the natural $\Pi^1_0$ formula asserting that $X$ is $0$-s-stationary (i.e. unbounded) in $\kappa$.

Suppose $\xi<\kappa^+$ is a limit ordinal and the result holds for $\zeta<\xi$. We let 
\[\varphi_\xi(X)=\bigwedge_{\zeta<\xi}\varphi_\zeta(X).\] 
Clearly $\varphi_\xi(X)$ is $\Pi^1_\xi$ over $V_\kappa$. Using our inductive assumption about the $\varphi_\zeta$'s, it is easy to verify that $A$ is $\xi$-s-stationary in $\kappa$ if and only if $V_\kappa\models\varphi_\xi(A)$. Furthermore, using an argument involving generic ultrapowers similar to those of Theorem \ref{theorem_expressing_indescribability} and Theorem \ref{theorem_xi_clubs}(2), the existence of the desired club $C_\xi$ is straightforward, and is therefore left to the reader.

Suppose $\xi=\zeta+1<\kappa^+$ is a successor. We let $\varphi_{\zeta+1}(X)$ be the natural $\Pi^1_{\zeta+1}$ formula equivalent to
\begin{align*}
\left(\bigwedge_{\eta<\zeta}\varphi_\eta(X)\right)&\land \forall S\forall T(\varphi_\zeta(S)\land\varphi_\zeta(T)\rightarrow\\
&(\exists\beta\in A)(\text{$S$ and $T$ are $f^\kappa_\zeta(\beta)$-s-stationary in $\beta$})).
\end{align*}
Note that, by an argument similar to that for Theorem \ref{theorem_xi_clubs}(2), we can code information about which subsets of $\beta$ are $f^\kappa_\zeta(\beta)$-s-stationary in $\beta$ into a subset of $\kappa$ and verify that the above formula is in fact equivalent to a $\Pi^1_{\zeta+1}$ formula over $V_\kappa$. The verification that the desired club $C_{\zeta+1}$ exists and that $\varphi_{\zeta+1}$ satisfies the requirements of the lemma is similar to the proof of Theorem \ref{theorem_expressing_indescribability} and Theorem \ref{theorem_xi_clubs}(2), and is thus left to the reader.
\end{proof}

The next proposition, which is a generalization of \cite[Proposition 4.3]{MR3894041}, will allow us to obtain the nondiscreteness of the topologies $\tau_\xi$ from an indescribability hypothesis.

\begin{proposition}\label{proposition_indescribability_implies_reflecting}
If a cardinal $\kappa$ is $\Pi^1_\xi$-indescribable for some $\xi<\kappa^+$, then it is $(\xi+1)$-s-reflecting.
\end{proposition}

\begin{proof}
Suppose $\kappa$ is $\Pi^1_\xi$-indescribable and suppose that $S$ and $T$ are $\zeta$-s-stationary in $\kappa$ where $\zeta\leq\xi$. Then we have
\[V_\kappa\models\varphi_\zeta(S)\land\varphi_\zeta(T)\]
where $\varphi_\zeta(X)$ is the $\Pi^1_\zeta$ formula obtained in Lemma \ref{lemma_expressing_s_stationarity}. Let $C_\zeta$ be the club subset of $\kappa$ from the statement of Lemma \ref{lemma_expressing_s_stationarity}. Since $\kappa$ is $\Pi^1_\xi$-indescribable, there is an $\alpha\in C_\zeta$ such that 
\[V_\alpha\models\varphi_\zeta(S)\res^\kappa_\alpha\land\varphi_\zeta(T)\res^\kappa_\alpha,\]
which implies that $S$ and $T$ are both $f^\kappa_\zeta(\alpha)$-s-stationary in $\alpha$. Hence $\kappa$ is $(\xi+1)$-s-stationary.
\end{proof}

Finally, we conclude that from an indescribability hypothesis, one can prove that the $\tau_{\xi+1}$ topology is not discrete.

\begin{corollary}\label{corollary_nondiscreteness_from_indescribability}
Suppose $\mu$ is a regular cardinal and $\xi<\mu^+$. If the set 
\[S=\{\alpha<\mu\st\text{$\alpha$ is $f^\mu_\xi(\alpha)$-indescribable}\}\]
is stationary in $\mu$ (for example, this will occur if $\mu$ is $\Pi^1_{\xi+1}$-indescribable), then there is an $\alpha<\mu$ which is nonisolated in the space $(\mu,\tau_{\xi+1})$.
\end{corollary}

\begin{proof}
Fix $\xi<\mu^+$. Let $C$ be the club subset of $\mu$ obtained from Theorem \ref{theorem_xi_s_nonisolated}; that is, $C\subseteq\mu$ is club such that for all $\alpha\in C$ we have $\alpha$ is $f^\mu_{\xi+1}(\alpha)$-s-reflecting if and only if $\alpha$ is not isolated in the $\tau_{\xi+1}$ topology. Let $D=\{\alpha<\mu\st f^\mu_{\xi+1}(\alpha)=f^\mu_\xi(\alpha)+1\}$ be the club subset of $\mu$ obtained from Lemma \ref{lemma_successor}. Now, if $\alpha\in S\cap C\cap D$ then $\alpha$ is $f^\mu_{\xi+1}(\alpha)$-s-reflecting by Proposition \ref{proposition_indescribability_implies_reflecting}, and is hence not isolated in the $\tau_{\xi+1}$ topology.
\end{proof}



\end{document}